\documentclass[11pt,reqno]{amsart}
\usepackage{amscd,amsfonts,amssymb,epic,eepic,hyperref}
\usepackage{mathptmx}

\providecommand{\BBb}[1]{{\mathbb{#1}}}
\providecommand{\Bold}[1]{{\mathbf{#1}}}
\providecommand{\cal}[1]{{\mathcal{#1}}}  
\newcommand{\goth}[1]{{\mathfrak #1}}

\newcommand{\ang}[1]{\langle#1\rangle}
\newcommand{\bigdot}{\mathbin{\raise.65\jot\hbox{$\scriptscriptstyle\bullet$}}}
\newcommand{\C}{{\BBb C}}
\newcommand{\Cn}{{\BBb C}^n}
\newcommand{\class}{\operatorname{class}}
\newcommand{\Dm}{\BBb D}
\newcommand{\dual}[2]{\langle\,#1,\,#2\,\rangle}
\newcommand{\eOm}{e_{\Omega}}
\newcommand{\fracc}[2]{{
                \textstyle\frac{#1}{\raise 1pt\hbox{$\scriptstyle #2$}}}}
\newcommand{\fracp}{\fracc1p}
\newcommand{\fracnp}{\fracc np}

\newcommand{\fracci}[2]{{\frac{#1}{\raise 1pt\hbox{$\scriptscriptstyle #2$}}}}
\newcommand{\fracpi}{\fracci1p}

\newcommand{\g}{\gamma_0}
\newcommand{\im}{\operatorname{i}}
\newcommand{\kd}[1]{\pmb{[\![}{#1}\pmb{]\!]}} 
\newcommand{\lap}{\operatorname{\Delta}}
\newcommand{\loc}{\operatorname{loc}}
\newcommand{\lOm}{\ell_{\Omega}}
\newcommand{\mlap}{-\!\operatorname{\Delta}}
\newcommand{\nrm}[2]{\|#1\|_{#2}}
\newcommand{\Nrm}[2]{\bigl\|#1\bigr\|_{#2}}
\newcommand{\norm}[2]{\mathinner{\|}#1\,|#2\|}
\newcommand{\Norm}[2]{\mathinner{\bigl\|\,#1\,\big|#2\bigr\|}}
\newcommand{\order}{\operatorname{order}}
\newcommand{\op}[1]{\operatorname{#1}}
\newcommand{\OP}{\operatorname{OP}}
\newcommand{\N}{\BBb N}
\newcommand{\pr}{\operatorname{pr}}     
\newcommand{\ord}[1]{$\raisebox{1ex}{{\footnotesize #1}}$}
\renewcommand{\Re}{\operatorname{Re}}
\newcommand{\R}{{\BBb R}}
\newcommand{\Rn}{{\BBb R}^{n}}
\newcommand{\Rp}{\overline{{\BBb R}}_+}
\newcommand{\rOm}{r_{\Omega}}
\newcommand{\Sdm}{\BBb S}
\newcommand{\set}[2]{\{\,#1 \mid #2\,\}}
\newcommand{\Set}[2]{\bigl\{\,#1\bigm| #2\,\bigr\}}
\newcommand{\singsupp}{\operatorname{sing\,supp}}
\newcommand{\supp}{\operatorname{supp}}
\newcommand{\Z}{\BBb Z}

\renewcommand{\hat}[1]{\overset{{\scriptscriptstyle \wedge}}{#1}}

\newcounter{rmcount}\renewcommand{\thermcount}{{\rm\roman{rmcount}}}
\newenvironment{rmlist}{%
\begin{list}{{\rm(\thermcount)}}{\setlength{\labelwidth}{\leftmargin}%
\usecounter{rmcount}}}{\end{list}}
\newcounter{Rmcount}\renewcommand{\theRmcount}{{\rm\Roman{Rmcount}}}
\newenvironment{Rmlist}{%
\begin{list}{{\rm(\theRmcount)}}{\setlength{\labelwidth}{\leftmargin}%
\usecounter{Rmcount}}}{\end{list}}

\numberwithin{equation}{section}
\newtheorem{thm}{Theorem}
\numberwithin{thm}{section}
\newtheorem{prop}[thm]{Proposition}
\newtheorem{lem}[thm]{Lemma}
\newtheorem{cor}[thm]{Corollary}
\theoremstyle{definition}
\newtheorem{defn}[thm]{Definition}
\newtheorem{exmp}[thm]{Example}
\theoremstyle{remark}
\newtheorem{rem}[thm]{Remark}

\title[Parametrices of semi-linear problems]{%
Parametrices and\\ exact paralinearisation
\\  of semi-linear boundary problems}
\author{Jon Johnsen}
\address{Department of Mathematical Sciences,
Aalborg University, Fredrik~Bajers Vej~7G, 9220 Aalborg~{\O}st, Denmark} 
\email{jjohnsen@math.aau.dk}
\thanks{~\\[4\jot]{\tt Appeared in 
Communications in Partial Differential Equations, {\bf 33\/}~(2008), 1729--1787.}%
}
\begin{document}

\begin{abstract}
The subject is 
parametrices for semi-linear problems,
based on parametrices for linear boundary problems 
and on non-linearities that decompose into 
solution-dependent linear operators acting on the solutions.
Non-linearities of product type 
are shown to admit this via exact paralinearisation. 
The parametrices give regularity properties under weak conditions;
improvements in subdomains result from pseudo-locality of type~$1,1$-operators.
The framework encompasses a broad class of boundary
problems in H{\"o}lder and $L_p$-Sobolev spaces
(and also Besov and Lizorkin--Triebel spaces). 
The Besov analyses of homogeneous distributions, tensor products and
halfspace extensions have been revised.
Examples include the von~Karman equation.
\end{abstract}
\maketitle%
\noindent\section{Introduction}
   \label{intr-sect}

This article presents a parametrix construction for semi-linear
boundary problems 
as well as the resulting regularity properties 
in $L_p$-Sobolev spaces.
The work is based on investigations of 
pseudo-differential boundary operators, paramultiplication
and function spaces of J.-M.~Bony, G.~Grubb, V.~Rychkov and the author 
\cite{Bon,G4,Ry99,JJ94mlt,JJ96ell}; it is also inspired by joint
work with T.~Runst \cite{JoRu97} on solvability of semi-linear problems.

Assume eg that $A$ is an elliptic
differential operator, that $\{A,T\}$ is a linear elliptic boundary problem
on a domain $\Omega\subset\Rn$ and that,
for a suitable non-linear operator $Q$, the function $u$ is a given solution
of the problem
\begin{equation}
  Au+Q(u)=f\quad\text{in $\Omega$},\qquad Tu=\varphi \quad\text{on
$\partial\Omega$}. 
  \label{intrprb-eq}
\end{equation}
It is then a main point
to establish a family of parametrices $P^{(N)}_u$, $N\in \N$, that are 
linear operators yielding the following new formula for $u$: 
\begin{equation}
  u= P^{(N)}_{u}(Rf +K\varphi+\cal Ru)+(RL_u)^Nu.
  \label{intrpar-eq}
\end{equation}
Here $\left(\begin{smallmatrix} R& K\end{smallmatrix}\right)$  
is a left-parametrix of the linear problem, ie 
$\left(\begin{smallmatrix} R& K\end{smallmatrix}\right)
\left(\begin{smallmatrix}A\\ T\end{smallmatrix}\right)=I-\cal R$ where $\cal
R$ has range in $C^\infty(\overline{\Omega})$, 
while $L_u$ is an exact paralinearisation
of $Q(u)$. 
In \eqref{intrpar-eq}, $P^{(N)}_u$ has order zero and
can roughly be seen as a modifier of
data's contribution to $u$, while $(RL_u)^N$ is an error term analogous
to the negligible errors in pseudo-differential calculi; it can have any
finite degree of smoothness by choosing $N$ large enough.
Precise assumptions on $\{A,T\}$ and especially $Q$ will 
follow further below.

The motivation was partly to provide an alternative to 
boot-strap arguments, for in the general $L_p$-setting these can 
require somewhat lengthy descriptions, even though the strategy is clear.
It was also hoped to find purely analytical
proofs, without iteration, of the regularity properties.

These goals are achieved with the parametrix formula
\eqref{intrpar-eq}, for the regularity of $u$ can be read
off in a simple way from the right hand side, as explained below. 
And along with stronger a priori regularity of the solution,
the parametrices allow increasingly weaker assumptions on the data. 
Moreover, the formula \eqref{intrpar-eq} clearly gives a structural
information, that here is utilised to prove that
additional regularity properties in
subregions also carry over to the solutions.

Furthermore, as a gratis consequence of the
method, the parametrix formulae may,  
depending on the problem and its data,
yield that the solution belongs to spaces, on which the non-linear terms are
of higher order than the linear terms, or are ill-defined. 
(Such results can often also be obtained with
iteration, if the a priori information of the solution is used in each step.)

Compared to results derived from the paradifferential calculus 
of J.-M.~Bony \cite{Bon}, the set-up is restricted here to
non-linearities of product type, as defined below,
but in the present work 
the regularity of non-zero boundary data $\varphi$ is taken fully into account
via the term $K\varphi$ (this was undiscussed in \cite{Bon}). 
Non-linear boundary conditions can also be covered with the present methods,
but this will be a straightforward extension, and therefore left out.

As usual, the differential operator $Au+Q(u)$ is called semi-linear
when it depends linearly on the highest order derivatives of $u$. 
For such operators, it could be natural to introduce (as below) four
\emph{parameter domains} $\Dm_{\kappa}$, $\Dm(Q)$, $\Dm(A,Q)$ and
$\Dm_u$. Whilst the first two describe $\{A,T\}$ 
($\kappa$ is the class of $T$) 
and $Q$, 
the others account for spaces on which
\eqref{intrprb-eq} has regularity properties resp.\ parametrices as 
expected for a semi-linear problem.

\bigskip

Notation and preliminaries are settled in Section~\ref{prelim-sect}.
In a general framework the main result follows in
Section~\ref{main-sect}. 
Some needed facts on paramultiplication are given in Section~\ref{prod-sect}.
In Section~\ref{pdty-sect} the exact paralinearisation of
non-linearities of product type is studied.
Section~\ref{vK-sect} presents the consequences for
the stationary von~Karman problem, and the weak solutions are carried over to
general $L_p$-Sobolev spaces.
The subject of Section~\ref{sys-sect} is the parametrix and regularity results
obtained for general systems of semi-linear elliptic boundary problems 
in vector bundles; this set-up should be natural in view of
the von~Karman problem treated in Section~\ref{vK-sect}.
Concluding remarks follow in  Section~\ref{finrem-sect}.

\subsection{The model problem}   \label{model-ssect}
Throughout $\Omega\subset\Rn$ is an open set with $C^\infty$-boundary
$\Gamma:=\partial \Omega$; $n\ge2$. It is an essential, standing assumption
that $\Omega$ is bounded.
The subject is exemplified in the rest of the introduction by
the following model problem, where 
$\lap=\partial^2_{x_1}+\dots+\partial^2_{x_n}$ is the Laplacian,
$\gamma_0 u=u|_{\Gamma}$ the trace,
\begin{equation}
  \begin{aligned}
  \mlap u+ u\cdot\partial_{x_1}u&=f\quad\text{in}\quad\Omega,  \\
  \gamma_0 u&=\varphi\quad\text{on}\quad\Gamma.
  \end{aligned}
  \label{bsc-eq}
\end{equation}
In relation to the parametrices, \eqref{bsc-eq} has much in common with
the stationary Navier--Stokes equation,  but
it is not a system, so it is simpler to present.

Denoting 
the inverse of $\left(\begin{smallmatrix}\mlap\\
\gamma_0\end{smallmatrix}\right)$ 
by $\left(\begin{smallmatrix} R_D& K_D\end{smallmatrix}\right)$,  
where the subscript $D$ refers to the Dirichl\'et problem for $\mlap$,
the formula \eqref{intrpar-eq} amounts to the following, when applied to a
given solution $u$,
\begin{equation}
  u= P^{(N)}_{u}(R_Df +K_D\varphi)+(R_DL_u)^Nu.
  \label{bscpar-eq}
\end{equation}
This expression should be new even when data and solutions are
given in the Sobolev spaces $H^s$. But the usefulness of
parametrices gets an extra dimension when the $L_p$-theory is discussed, so
it will be natural to consider at least Sobolev spaces 
$H^{s}_{p}(\overline{\Omega})$ and H{\"o}lder--Zygmund
classes $C^s_*(\overline{\Omega})$. 

However, these are special cases of 
Besov spaces $B^{s}_{p,q}(\overline{\Omega})$ and
Lizorkin--Triebel spaces $F^{s}_{p,q}(\overline{\Omega})$
(the definition is recalled in \eqref{bspq-eq}--\eqref{fspq-eq}
below), since 
\begin{align}
  H^s_p&=F^{s}_{p,2}  \quad\text{for $1<p<\infty$ and $s\in \R$},
  \label{HF-eq}  \\
  C^s_*&=B^s_{\infty,\infty} \quad\text{for $s\in \R$}.
\end{align}
For the well-known $W^s_p$ spaces, $W^s_p=B^s_{p,p}$ for non-integer
$s>0$ and $W^m_p=F^m_{p,2}$ for $m\in\N$, $1<p<\infty $. 
To avoid formulations with many scales,
the exposition will be based on the $B^{s}_{p,q}$ and
$F^{s}_{p,q}$ spaces, and for brevity
$E^{s}_{p,q}$ will denote a space
that is either $B^{s}_{p,q}$ or $F^{s}_{p,q}$ (in every occurrence
within, say the same formula or theorem).

Moreover, $B^{s}_{p,q}(\overline{\Omega})$ and
$F^{s}_{p,q}(\overline{\Omega})$ are defined for $p$, $q\in\,]0,\infty]$
($p<\infty$ for $F^{s}_{p,q}$) and $s\in\R$,
where the incorporation of $p$, $q<1$ is convenient for non-linear
problems, for as non-linear maps often have natural co-domains with
$p<1$, the $H^s$- and $H^s_p$-scales would be too tight frameworks.
The price one pays for this roughly equals  the burdening of the 
exposition that would result from a limitation to $p$, $q\ge1$.

Furthermore, $F^m_{p,1}$, $1\le p<\infty$ was in 
\cite{JJ04Dcr,JJ05DTL} shown to be maximal domains for type $1,1$-operators,
ie pseudo-differential operators in $\OP(S^m_{1,1})$;
cf Section~\ref{sym-ssect} below.
Such operators show up in the linearisations,
so the $F$-scale is likely to appear anyway in connection with the
parametrices. 

If desired, the reader can of course specialise to, 
say $H^s_p$ by setting $q=2$ in the $F$-scale, 
cf \eqref{HF-eq}.
The main part of the paper deals with the parametrix construction
and its consequences, and it does not rely on a specific
choice of $L_p$-Sobolev spaces.

\bigskip

For simplicity, \eqref{bsc-eq} will in the introduction be
discussed in the Besov scale $B^{s}_{p,q}$. As a basic
requirement the spaces should fulfil the following two inequalities,
where for brevity $t_+=\max(0,t)$ stands for the positive part of $t$,
\begin{subequations}
    \label{bsc-ineq}
  \begin{align}
  s&>\fracp +(n-1)(\fracp-1)_+
  \label{bsc'-ineq}
\\
  s&>\tfrac12+n(\fracp-\tfrac12)_+.
  \label{bsc''-ineq}
  \end{align}
\end{subequations}
It is known how these allow one to make sense of the trace and the product,
respectively. Working under such conditions, a main question for 
\eqref{bsc-eq} is the following 
\emph{inverse regularity} problem:
\par
\medskip\noindent
\hspace*{\textwidth}%
\newlength{\tagw}\settowidth{\tagw}{\mbox{(IR)}}%
\hspace*{-\tagw}%
\mbox{(IR)}\hspace*{-\textwidth}
\newlength{\problemwidth}\setlength{\problemwidth}{\textwidth}%
\addtolength{\problemwidth}{-2.5\parindent}%
\addtolength{\problemwidth}{-\tagw}%
\addtolength{\problemwidth}{-2em}%
\hspace*{2.5\parindent}%
\begin{minipage}{\problemwidth}
{  given a solution $u$ in one Besov space
  $B^{s}_{p,q}(\overline{\Omega})$,\\
  for  data $f$  in
  $B^{t-2}_{r,o}(\overline{\Omega})$ and $\varphi$ in
  $B^{t-\fracci1r}_{r,o}(\Gamma)$,
  \\
  \vphantom{$B^{t-\fracci1r}_{r,o}$}
  will $u$ be in $B^{t}_{r,o}(\overline{\Omega})$ too?%
}%
\end{minipage}%
\providecommand{\pbref}[1]{\mbox{(#1)}}%
\par
\medskip\noindent
Consider eg a solution
$u$ in $H^1(\overline{\Omega})$ for data 
$f\in C^{\alpha}(\overline{\Omega})$,
$\varphi\in C^{2+\alpha}(\Gamma)$, $0<\alpha<1$.
(For $\varphi=0$ and `small' $f\in  H^{-1}$ solutions exist in $H^1_0$ 
for $n=3$ by the below Proposition~\ref{small-prop}.)
The question is then whether $u$ also 
belongs to $C^{2+\alpha}(\overline{\Omega})$. The
latter space equals $B^{2+\alpha}_{\infty,\infty}(\overline{\Omega})$ while
$H^1= B^1_{2,2}$, so problem 
\pbref{IR} 
clearly contains a classical issue; actually \pbref{IR} 
is somewhat sharper because of the third parameter. 

In comparison with \pbref{IR}, \emph{direct} regularity properties
are used for the collection of mapping properties of eg
$u\mapsto u\partial_1 u$ or $\mlap u+u\partial_1 u$.
An account of these clearly constitutes
another regularity problem (often addressed before 
\pbref{IR} is solved), so it is proposed to distinguish this 
from \pbref{IR} by using the terms direct/inverse.

In connection with \pbref{IR}, one purpose of this paper is to test how weak
conditions one can impose in addition to \eqref{bsc-ineq}. 
Along with this, it is described how the \emph{parametrix}
formula in \eqref{bscpar-eq} 
(cf also \eqref{bsc''-eq} and Theorems~\ref{ir-thm} and \ref{semiell-thm} 
below) 
yields the expected regularity properties.
The result is a flexible framework implying
that $u\in B^t_{r,o}$, also in certain cases when 
the map $u\mapsto u\partial_1 u$ 
has higher order than $\mlap$ on 
the target space 
$B^t_{r,o}$, or when $u\partial_1 u$ is
ill-defined on $B^t_{r,o}$. Examples of this are given 
in Theorem~\ref{mlap-thm}; cf Remark~\ref{illdef-rem}.

Briefly stated, the above results and their generalisations are deduced from
an exact paralinearisation $L_u$ of $u\partial_1 u$
together with the parametrix $\left(\begin{smallmatrix} R_D &
K_D\end{smallmatrix}\right)$ of 
$\left(\begin{smallmatrix}\mlap\\ \g\end{smallmatrix}\right)$, belonging
to the Boutet de~Monvel calculus of pseudo-differential boundary
operators. When combined with a
Neumann series, these ingredients yield $P^{(N)}_u$ and
the parametrix formula \eqref{bscpar-eq}. 
This resembles the usual elliptic theory at the place where 
non-principal terms are included, but for one thing a finite
series suffices here, as in \cite{Bon}, 
since the error term $(R_DL_u)^Nu$ in \eqref{bscpar-eq}
only needs to belong to $B^{t}_{r,o}$; secondly, it is less simple in the
present context to keep track of the spaces on which the various steps are
meaningful.

As another consequence of \eqref{bscpar-eq}, 
if in an open subregion $\Xi\Subset\Omega$ 
(ie $\Xi$ has compact closure in
$\Omega$, hence positive distance to the boundary) data locally have
additional properties such as $f\in B^{t_1-2}_{r_1,o_1}(\Xi,\loc)$, 
then $u\in
B^{t_1}_{r_1,o_1}(\Xi,\loc)$ also holds.
These local improvements are deduced from the 
pseudo-local property of type $1,1$-operators, 
which was proved recently by the author
in \cite{JJ08vfm}; 
cf Section~\ref{sym-ssect}.

\subsection{On the parametrices}   \label{par-ssect}
It is perhaps instructive first to review the
corresponding linear problem, with $u$, $f$ and $\varphi$ as in \pbref{IR}: 
\begin{equation}
   \begin{pmatrix} \mlap\\ \g\end{pmatrix} u
   =  \begin{pmatrix} f\\ \varphi \end{pmatrix}.
  \label{lin-eq}
\end{equation}
For the proof that $u\in B^t_{r,o}(\overline{\Omega})$, 
there is a straightforward method introduced 
by G.~Grubb in \cite[Thm.~5.4]{G3} in a context of $H^s_p$ 
and classical Besov spaces with $1<p<\infty$.  

The argument uses that
$\left(\begin{smallmatrix}\mlap\\
\g\end{smallmatrix}\right)$ is an elliptic Green operator belonging to the
calculus of L.~Boutet de~Monvel \cite{BM71}, hence has a parametrix 
$\left(\begin{smallmatrix}R_D& K_D\end{smallmatrix}\right)$ there
(this calculus is used throughout,
not just for \eqref{lin-eq} but also for the semi-linear problems,
cf~Section~\ref{sys-sect} below). 
As shown in \cite[Ex.~3.15]{G3}, it is possible to take 
the singular Green operator part of $R_D$ such that the class%
\footnote{The class is the minimal $r\in\Z\cup\{\,\pm\infty\,\}$
with continuity  $H^r\to\cal D'$ of the operator.}
of $R_D$ equals
\begin{equation}
  \class(\g)-\order(\mlap)=1-2=-1.
\end{equation}
With this choice,
$\left(\begin{smallmatrix}R_D& K_D\end{smallmatrix}\right)$ has 
continuity properties in $H^s_p$ spaces as accounted for in \cite[Thm.~5.4]{G3}; 
under the assumptions in \eqref{bsc'-ineq},
continuity from
$B^{t-2}_{r,o}(\overline{\Omega})\oplus B^{t-\fracci1r}_{r,o}(\Gamma)$ 
to $B^{t}_{r,o}(\overline{\Omega})$ follows from
\cite[Thm.~5.5]{JJ96ell}.

Being a parametrix, $\left(\begin{smallmatrix}R_D& K_D\end{smallmatrix}\right)
\left(\begin{smallmatrix}\mlap\\\g\end{smallmatrix}\right)=I-\cal R$ 
for some regularising operator $\cal R$ with range in
$C^\infty(\overline{\Omega})$, and class $1$ (although 
$\left(\begin{smallmatrix} \mlap\\\gamma_0\end{smallmatrix}\right)$ is
invertible, 
$\cal R$ has been retained here for easier comparison with the general case). 
So, using the just mentioned continuity, 
an application of $\left(\begin{smallmatrix}R_D&
K_D\end{smallmatrix}\right)$ to both sides of \eqref{lin-eq} gives that 
\begin{equation}
  u=R_D f+ K_D\varphi +\cal Ru \quad\text{belongs to}\quad
    B^{t}_{r,o}(\overline{\Omega}).
  \label{lin-eq'}
\end{equation}
This only requires the mapping properties of
$\left(\begin{smallmatrix}\mlap\\ \g\end{smallmatrix}\right)$ and 
$\left(\begin{smallmatrix}R_D& K_D\end{smallmatrix}\right)$, 
that are as stated whenever $(s,p,q)$ and
$(t,r,o)$ both satisfy \eqref{bsc'-ineq}.

\bigskip

In the \emph{parametrix construction }
the first step is this: given a solution $u$
of \eqref{bsc-eq}, find a \emph{linear},
$u$-dependent operator $L_u$ such that, with a sign convention,  
\begin{equation}
  L_uu=-u\partial_1 u.
  \label{Luu-eq}
\end{equation}
Here it seems decisive to utilise paralinearisation.
On $\Rn$ this departs from paramultiplication, that 
yields a decomposition of the usual `pointwise' product 
\begin{equation}
  v\cdot w= \pi_{1}(v,w)+\pi_{2}(v,w)+\pi_{3}(v,w),
\end{equation}
where the $\pi_{j}$ are 
paraproducts (cf \eqref{paramultiplication-eq} below). In the
notation of J.-M.~Bony \cite{Bon}, 
paramultiplication by $v$ is written $T_vw$ instead of
$\pi_1(v,w)$, and $\pi_3(v,w)=T_wv=\pi_1(w,v)$, 
whilst $R(v,w)=vw-T_vw-T_wv=\pi_2(v,w)$ is the remainder.

More specifically, the linearisation $L_u$ has the following form
for $\Omega=\Rn$,
\begin{equation}
  \begin{split}
  -L_u g&= \pi_1(u,\partial_1g)+ \pi_2(u,\partial_1g)+\pi_3(g,\partial_1u)
\\
          &=T_u(\partial_1g)+R(u,\partial_1 g)+T_{\partial_1u}(g).
  \end{split}
  \label{paral-eq} 
\end{equation}
Here the last line should emphasise how $u$ and $g$ enter.
As a comparison $g\mapsto u\partial_1g$ can be written
$T_u(\partial_1\cdot)+R(u,\partial_1\cdot)+T_{\partial_1(\cdot)}(u)$;
otherwise this notation will not be used.

In the usual paralinearisation, the $\pi_2$-term is omitted since it is of
higher regularity (leading to the famous formula
$F(u(x))=\pi_1(F'(u(x)),u(x))\,\,+$ smoother terms). 
But $\pi_2(u,\partial_1\cdot)$ is first of all 
\emph{not} regularising in the present
context, where $u$ may be given in $B^{s}_{p,q}$ or $F^{s}_{p,q}$ also 
for $s<\fracnp$ (this is possible by \eqref{bsc''-ineq}), thus allowing $u$
to be  
unbounded. Secondly, \eqref{bsc''-ineq} is the only `non-linear' limitation 
within the theory, and this arises because $\pi_{2}(u,\partial_1 g)$ may or
may not be defined; by incorporation of this term
into $L_u$ as in \eqref{paral-eq}, the resulting limitation is
whether or not $L_u$ itself is defined on~$g$. 

In view of this, $L_u$ in \eqref{paral-eq} is throughout referred to
as an \emph{exact} paralinearisation of $u\partial_1 u$.
As explained in Section~\ref{sym-ssect} below, linearisation at 
$u\in B^{s_0}_{p_0,q_0}$ leads to a pseudo-differential operator 
in $\op{OP}(S^{\omega}_{1,1})$
for $\omega=1+(\fracci n{p_0}-s_0)_++\varepsilon$.
Besides the number $1$ coming from $\partial_{x_1}$, the term
$(\fracc n{p_0}-s_0)_++\varepsilon$ appears because 
$u(x)$ may be unbounded on $\Omega$
($\varepsilon\ge0$, non-trivial only for $s_0=\fracc n{p_0}$). 

As accounted for in Section~\ref{pdty-sect} below, $L_u$ has this order on 
\emph{all}
spaces $B^{s}_{p,q}$ where it is shown to be defined; the collection of
these spaces will from Section~\ref{Dm-ssect} onwards be referred to as the
parameter domain of $L_u$, denoted by $\Dm(L_u)$.
Moreover, the order is the same as that
of $Q(u):=u\partial_1 u$ on $B^{s_0}_{p_0,q_0}$.
Therefore the Exact Paralinearisation Theorem (Theorem~\ref{mod-thm}) can be
summed up thus:

\begin{thm}   \label{intr-thm}
On every space in $\Dm(L_u)$, the exact paralinearisation $g\mapsto L_u(g)$ 
is of the same order as the non-linear map $Q$ on the space 
$B^{s_0}_{p_0,q_0}\ni u$.
\end{thm}

This is shown for arbitrary product type operators in Theorem~\ref{mod-thm},
and in a vector bundle set-up in Theorem~\ref{ptyp-thm} below. 
For composition operators $F(u(x))$ it is known that
the theorem holds if $s_0>\fracc n{p_0}$ since then
$L_u\in \op{OP}(S^{0}_{1,1})$.

On an open set $\Omega\subset\Rn$ one can combine the linearisation in 
\eqref{paral-eq} with
prolongation and restriction. When $\rOm$ denotes restriction from $\Rn$ to
$\Omega$, prolongation $\lOm$ is as usual a continuous linear map 
\begin{equation}
  \lOm\colon E^{s}_{p,q}(\overline{\Omega})\to E^{s}_{p,q}(\Rn);
  \qquad \rOm\circ\lOm=I.
  \label{ext-eq}
\end{equation}
In \cite{T2,T3} there was given a construction, for each $N$, of $\lOm$ such
that \eqref{ext-eq} holds for $|s|<N$, $p,q>1/N$.
While it would be possible to work with this here,
it is a more convenient result of V.~Rychkov \cite{Ry99,Ry99BPT} 
that $\lOm$ can be so constructed that \eqref{ext-eq} holds for all $s\in
\R$ and all $p$, $q\in \,]0,\infty]$ ($p<\infty$ in the $F$-case), a
so-called \emph{universal} extension operator.
This construction was made for bounded Lipschitz domains. Briefly stated,
the basic step is to apply a fine version of the Calderon reproducing
formula $u=\sum \varphi_\nu*(\psi_\nu*u)$ 
near a boundary point, where the convolution $\psi_\nu*u$ 
(is defined when both $u$ and $\psi$ are supported in a cone and)
has a meningful extension by $0$ to $\Rn\setminus\Omega$ since it is a
function; whereafter the convolution by $\varphi_\nu$ gives a smooth
function on $\Rn$; the whole process is controlled in $B^{s}_{p,q}$-
and $F^{s}_{p,q}$-spaces via equivalent norms involving maximal functions,
established for this purpose in \cite{Ry99}.

Using this, the operator $L_u$ in \eqref{Luu-eq} is for the boundary problem
\eqref{bsc-eq} taken as
\begin{equation}
  L_u g= -\rOm\pi_1(\lOm u,\partial_1\lOm g)
         -\rOm\pi_2(\lOm u,\partial_1\lOm g)
         -\rOm\pi_3(\lOm g,\partial_1\lOm u).
  \label{Lu-eq}
\end{equation}
As a convenient abuse, this is also called the exact
paralinearisation of $u\partial_1 u$. 
It is not surprising that the mapping
properties given in and before Theorem~\ref{intr-thm} carry over to $L_u$ on
$\Omega$, and this turns out to be decisive for the construction.

To focus on the simple algebra behind the parametrix formula,
precise assumptions on the spaces will be suppressed until
Section~\ref{Dm-ssect}. First it is noted that 
equation \eqref{bsc-eq}, by application of $\left(\begin{smallmatrix} R_D&
K_D\end{smallmatrix}\right)$ and insertion of \eqref{Luu-eq},
will entail that
\begin{equation}
  u-R_DL_uu=R_Df+K_D\varphi+\cal Ru.
  \label{bsc'-eq}
\end{equation}
The idea is now to apply the finite Neumann series 
\begin{equation}
  P^{(N)}_{u}:=I + R_DL_u +\dots+ (R_DL_u)^{N-1}.
\end{equation}
This will constitute the desired parametrix.
Because $(R_DL_u)^j$ is {\em linear\/}
\begin{equation}
  P^{(N)}_u(I-R_DL_u)=I-(R_DL_u)^N,   
\end{equation}
hence the resulting parametrix formula is
\begin{equation}
  u=P^{(N)}_{u}(R_Df+K_D\varphi+\cal Ru)
    +(R_DL_u)^N (u).
  \label{bsc''-eq}
\end{equation}
Note that in comparison with \eqref{lin-eq'}, there are two extra
ingredients here, namely $P^{(N)}_u$ and $(R_DL_u)^Nu$, that describe the
effects of the non-linear terms.

As a main application of \eqref{bsc''-eq}, one can read off the regularity
of a given solution $u\in B^{s}_{p,q}(\overline{\Omega})$ in the following
way: 
An uncomplicated analysis given in Theorem~\ref{ir-thm} below
shows two new fundamental results, namely 
\begin{align}
  \exists N &\colon B^{s}_{p,q}(\overline{\Omega}) 
   \xrightarrow{\ (R_DL_u)^N\ }
   B^{t}_{r,o}(\overline{\Omega})
  \label{exN-eq}  \\
  \forall N &\colon B^{t}_{r,o}(\overline{\Omega}) 
   \xrightarrow{\ P^{(N)}_u\ }
   B^{t}_{r,o}(\overline{\Omega}).
  \label{allN-eq}
\end{align}
Since $R_Df+K_D\varphi+\cal Ru$ is in $B^t_{r,o}(\overline{\Omega})$ by the
linear theory, it is therefore clear that all terms on the right hand side of
\eqref{bsc''-eq}  
belong to $B^{t}_{r,o}$, as desired, provided $N$ is chosen
as in \eqref{exN-eq}.

The possibility of picking $P^{(N)}_{u}$ sufficiently regularising 
resembles the Hadamard parametrices, cf the description in \cite[17.4]{H}.
It is not intended to give a symbolic calculus containing $P^{(N)}_u$ 
(the difficulties in this are elucidated in 
Remark~\ref{symb-rem});
it is rather a point that
the parametrices and resulting regularity properties
may be obtained by simpler means.

Seemingly \eqref{bsc''-eq}--\eqref{allN-eq} have not
been crystallised before in connection with boundary problems. 
This might be a little surprising, since in a sense
they boil down to the fact that $R_DL_u$ is of negative order. 
Along with the algebra above, it is of course all-important
to account for the spaces on which  the various steps are both 
meaningful and
give the conclusions \eqref{bsc''-eq}--\eqref{allN-eq}.
However, first some terminology is settled.

\subsection{Maps, orders and parameter domains}
  \label{Dm-ssect}

A (possibly) non-linear operator $T$ is said to have order
$\omega$ on $E^{s}_{p,q}$ if $T$ maps this space into $E^{s-\omega}_{p,q}$
and $\norm{T(f)}{E^{s-\omega}_{p,q}}\le c\norm{f}{E^s_{p,q}}$ for some
constant $c$.  In general this leads to a function $\omega(s,p,q)$, for
typically $T$ is given along with a natural range of
parameters $(s,p,q)$ for which it makes sense on $E^s_{p,q}$; then the set
of such $(s,p,q)$ is denoted by $\Dm(T)$ and is called the \emph{parameter}
domain of $T$.

The order is differently defined if
$E^s_{p,q}$ and $E^{s-\omega}_{p,q}$ are considered over 
manifolds of unequal dimensions. But here it suffices to note that for 
the outward normal derivative of order $k-1$ at $\Gamma$, ie for $\gamma_{k-1}
f:=((\frac{\partial}{\partial\vec n})^{k-1} f)|_{\Gamma}$, there is a
well-known parameter domain $\Dm_k$ given by
\begin{equation}
  \label{Dmk-eq}
  \Dm_k=\bigl\{\,(s,p,q) \bigm| s>k+\fracp-1+(n-1)(\fracp-1)_+\,\bigr\}.
\end{equation}
For if $(s,p,q)\in \Dm_k$ there is continuity of the trace $\gamma_k\colon
B^{s}_{p,q}(\overline{\Omega})\to B^{s-k-\fracpi}_{p,q}(\Gamma)$ and of
$\gamma_k\colon F^{s}_{p,q}(\overline{\Omega})\to
B^{s-k-\fracpi}_{p,p}(\Gamma)$.  
The $k^{\op{th}}$ domain $\Dm_k$ 
is also the usual choice for elliptic boundary problems of class $k\in\Z$.

The notion of parameter domains (that was introduced jointly with T.~Runst
\cite{JoRu97}) will be convenient throughout.
Indeed, despite its simple nature, the model problem
\eqref{bsc-eq} requires four different
parameter domains for the analysis of \pbref{IR}; further below
these will be introduced as 
$\Dm(\cal A)$, $\Dm(Q)$, $\Dm(\cal A,Q)$ and $\Dm(L_u)$
along with their general analogues.

To characterise the properties leading to parametrices,
let $\cal N$ be a non-linear operator defined on $E^s_{p,q}$ for $(s,p,q)$
running in a parameter domain $\Dm(\cal N)$.
When compared to a linear operator $A$ having order $d_A$ on a domain
$\Dm(A)$, then $\cal N$ is said to be
$A$-moderate \emph{on} $E^{s}_{p,q}$ in $\Dm(A)\cap\Dm(\cal N)$ 
if $\cal N$ is a map $E^{s}_{p,q}\to E^{s-\sigma}_{p,q}$ for some $\sigma<d_A$.
For short $\cal N$ is simply called 
$A$-\emph{moderate} if such a $\sigma$ exists on every
space in $\Dm(A)\cap\Dm(\cal N)$.

To generalise this notion, a linear operator $L_u$ will be called a
linearisation of ${\cal N}$ if  
for every $u\in E^s_{p,q}$ with $(s,p,q)$ in $\Dm(\cal N)$,
\begin{equation}
  \cal N(u)=-L_u(u).
\end{equation}
Here $L_u$ should be a meaningful linear operator parametrised by the
$u$ (running through the spaces) in $\Dm(\cal N)$, 
or possibly for $u$ in a larger parameter domain $\Dm(\cal L)$. 

It will be required that, for $u\in
E^{s_0}_{p_0,q_0}$ fixed,  $g\mapsto L_u(g)$ 
should be of order $\omega(s,p,q)$,
ie be a map $E^s_{p,q}\to E^{s-\omega(s,p,q)}_{p,q}$, 
on every $E^s_{p,q}$ in a parameter domain denoted $\Dm(L_u)$.
(It will be seen in Theorem~\ref{mod-thm}, ie the full version of the Exact
Paralinearisation Theorem, that $\Dm(\cal L)=\R\times\,]0,\infty]^2$ because
the operator $L_u$ is a meaningful object for all $u$; but once $u$ is fixed,
the parameter domain of $g\mapsto L_u(g)$ is much smaller, and its
determination is a main point in Theorem~\ref{mod-thm}.)
Although $\omega$ is a function $\omega(s,p,q,s_0,p_0,q_0)$, 
the arguments $s_0,p_0,q_0$ are often left out, since $u$ is fixed in 
$E^{s_0}_{p_0,q_0}$; but 
for generality's sake $(s,p,q)$ is kept though $\omega$ often is a constant
in this paper.

\begin{defn}
  \label{mod-defn}
A linearisation $L_u$ with parameter domain $\Dm(L_u)\supset \Dm(\cal N)$ 
is said to be \emph{moderate} if, for every linearisation point 
$u$ in an arbitrary $E^{s_0}_{p_0,q_0}$ in $\Dm(\cal N)$, 
\begin{equation}
 \omega_{\max}:= \sup_{\Dm(L_{u})\times\Dm(\cal N)}
\omega(s,p,q,s_0,p_0,q_0)<\infty .
\end{equation}
In case there is some $(s_0,p_0,q_0)$ in $\Dm(\cal N)$ such that
$\sup_{(s,p,q)\in \Dm(L_{u})} \omega(s,p,q)<\infty$, then $L_u$ is said to
be \emph{moderate on $E^{s_0}_{p_0,q_0}$}.
And $L_u$ is said to be \emph{$A$-moderate} 
on $E^{s_0}_{p_0,q_0}\ni u$ if $(s,p,q)\in \Dm(A)\cap \Dm(L_u)$ implies
\begin{equation}
  \omega(s,p,q,s_0,p_0,q_0)< d_A.
  \label{order-ineq}
\end{equation}
\end{defn}

Moderate linearisations are therefore those that,
regardless of the linearisation point $u$, have 
uniformly bounded  orders on their entire parameter domains.
Clearly $\cal N$ is $A$-moderate on $E^{s_0}_{p_0,q_0}$ (in $\Dm(\cal N)$)
if $L_u$ is so, for since
$-L_uu=\cal N(u)$ holds at $(s_0,p_0,q_0)$ 
it is trivial that $\cal N$ is a map $E^{s_0}_{p_0,q_0}\to
E^{s_0-\omega(s_0,p_0,q_0)}_{p_0,q_0}\subset E^{s_0-d_A}_{p_0,q_0}$.

\begin{rem}
  \label{Lu-rem}
With the third term of \eqref{Luu-eq} equal
to $\rOm\pi_3(\lOm\cdot,\partial_1 \lOm u)$, 
the regularity of $L_u g$
is known to depend mainly on $g$. Indeed, if
$u\in B^{s_0}_{p_0,q_0}(\overline{\Omega})$, then $L_u g$
has in general only $(\fracc n{p_0}-s_0)_++1+\varepsilon$ 
derivatives less than $g$; cf Theorem~\ref{intr-thm}.
This value is a constant independent of $g$ 
and  $\fracc n{p_0}-s_0<\tfrac{n}{2}$
holds by \eqref{bsc''-ineq}, so $\omega_{\op{max}}<\infty$ and
$L_u$ is moderate; and $\lap$-moderate if eg $s_0\ge \fracc n{p_0}$. 

The linearisation $g\mapsto u\partial_1g$ might 
look natural, but since $u\partial_1 g\in B^{s}_{p,q}$
can be shown to hold if $s\le s_0$, it is of non-constant order  
$\omega(t,r,o)\ge t-s_0$ 
on $B^{t}_{r,o}\ni g$, hence not moderate 
because $\omega_{\max}\ge\sup_t t-s_0=\infty$.
Moreover, this order is larger than that of
$\mlap$ when $t>s_0+2$, so in this region it is not $\lap$-moderate.
\end{rem}

Before justifying the formal steps in
\eqref{bsc'-eq}--\eqref{bsc''-eq},
it is convenient to present the parameter
domains for problem \pbref{IR} first. This is done by merely stating
the consequences of the following sections, 
with reference to the general results.

Departing from the linear part of \eqref{bsc-eq},
the Dirichl\'et condition leads to \eqref{bsc'-ineq}, and
since the problem has 
class~$1$, one can reformulate this using \eqref{Dmk-eq}, by introduction of
the parameter domain of $\cal A=\left(\begin{smallmatrix} \mlap\\
\gamma_0\end{smallmatrix}\right)$ as
\begin{equation}
  \Dm(\cal A)=\Dm_1=\{\,(s,p,q)\mid s>\fracp+(n-1)(\fracp-1)_+ \,\}.
\end{equation}

For the quadratic operator $Q(u):=u\partial_1 u$ 
one should have a parameter domain $\Dm(Q)$ such that
$Q$ is well defined on all $B^{s}_{p,q}$ and $F^{s}_{p,q}$ in this domain.
This question is treated in Proposition~\ref{nlnr2-prop} below,
in a context of product type operators studied in
Section~\ref{est-ssect}. This yields precisely the
condition \eqref{bsc''-ineq}, cf \eqref{stQdom-eq} and
Figure~\ref{stdom-fig} there; this amounts to the 
\emph{quadratic} standard domain of $Q$,
\begin{equation}
 \Dm(Q)=\{\,(s,p,q)\mid s>\tfrac{1}{2}+(\fracnp-\tfrac{n}{2})_+\,\}.
  \label{Qdm-eq}
\end{equation}

In the important determination of the spaces on which $Q$ is $\lap$-moderate,
one can depart from the conclusion of Proposition~\ref{nlnr2-prop} below
that $Q$ is of order $\sigma(s,p,q)=1+(\fracnp-s)_++\varepsilon$,
with an $\varepsilon\ge0$ nontrivial only for $s=\fracnp$.
Ie $Q$ is a bounded map
\begin{equation}
  Q\colon B^s_{p,q}\to B^{s-\sigma(s,p,q)}_{p,q}.
\end{equation}
(More precisely, one should instead of $Q$ consider $\left(\begin{smallmatrix}
Q\\0\end{smallmatrix}\right)$ and check where it is $\cal A$-moderate, but
it is a convenient abuse to focus on $Q$ and $\lap$ instead.)

In principle one can now introduce a parameter domain of $\lap$-moderacy for
$Q$ by solving the inequality $\sigma(s,p,q)<2$ on $\Dm(\cal A)\cap\Dm(Q)$,
cf~\eqref{order-ineq}; this leads to
\begin{equation}
  \Dm(\cal A,Q) :=
  \Dm(\cal A)\cap \bigl\{\,(s,p,q)\in \Dm(Q) 
         \bigm| \sigma(s,p,q)<2 \,\bigr\}.
  \label{delQ-id}
\end{equation}
However, this calculation is made for a general semi-linear problem with the
result summed up in Corollary~\ref{Amod-cor} below.
If $n\ge 3$ for simplicity, one finds from this result and 
the obvious inclusion $\Dm(\cal A)=\Dm_1\subset\Dm(Q)$ that  
\begin{equation}
  \Dm(\cal A,Q)=
  \bigl\{\,(s,p,q) \bigm|  s>\tfrac{1}{2}+(\fracnp-\tfrac{3}{2})_+\,\bigr\}.
  \qquad (n\ge 3)
  \label{delQ-eq}
\end{equation}
So far the considerations are classical in nature (even if formulated for
the $B^{s}_{p,q}$-spaces). But the use of parameter domains and the concise
$\Dm$-notation will be particularly useful for the next remarks, that also
explain how general regularity results the present methods can give.

Using the exact paralinearisation,
$Q(u)=-L_u(u)$ holds on the entire quadratic standard domain
$\Dm(Q)$, as verified in Lemma~\ref{Lu-lem} below.
But as a new observation, $g\mapsto L_u(g)$ is for a fixed 
$u\in B^{s_0}_{p_0,q_0}$ defined on every space in  
\begin{equation}
  \Dm(L_u)=
  \{\,(s,p,q)\mid s>1-s_0+(\fracnp+\fracc n{p_0}-n)_+\,\}.
  \label{bscDmLu-eq}
\end{equation}
This is part of the content of the Exact Paralinearisation Theorem
in Section~\ref{mod-ssect} below.

It is not difficult to infer that 
$\Dm(L_u)\supset\Dm(Q)$ holds for $(s_0,p_0,q_0)\in \Dm(Q)$, 
in general with a considerable gap\,---\,for the borderline of $\Dm(Q)$ is
obtained from $\Dm(L_u)$ by setting $(s,p,q)$ and $(s_0,p_0,q_0)$ equal, so
when $(s_0,p_0,q_0)\in \Dm(Q)$, then $(s,p,q)$ can lie
an exterior part of $\Dm(Q)$ 
without violating the inequality in \eqref{bscDmLu-eq}.
It is also clear
that $\Dm(L_u)$ increases with improving a priori regularity of $u$, 
ie with increasing $s_0$ or $p_0$. 

Moreover, given a solution $u$ in
some $B^{s_0}_{p_0,q_0}$ in $\Dm(\cal A,Q)$, the parametrices and the
resulting inverse regularity properties are established in the domain
\begin{equation}
  \Dm_u=\Dm_1\cap\Dm(L_u).
  \label{Dmu-id}
\end{equation}
This is larger than $\Dm(\cal A,Q)$, for \eqref{delQ-id} gives
$\Dm(\cal A,Q)\subset \Dm_1\cap\Dm(Q)\subset \Dm_1\cap\Dm(L_u)$.

It is now possible to sketch a proof of the parametrix formula
\eqref{bsc''-eq} and the
crucial properties in \eqref{allN-eq}--\eqref{exN-eq}.
Given a solution $u$ of \eqref{bsc-eq} in, say $B^{s_0}_{p_0,q_0}$ with
$(s_0,p_0,q_0)$ in $\Dm(\cal A,Q)$,
Theorem~\ref{intr-thm} shows that $L_u$ has order
$\sigma(s_0,p_0,q_0)<2$ on all spaces in $\Dm(L_u)$. 
Therefore $R_DL_u$ is defined and has order $-\delta$, for some $\delta>0$, 
on \emph{all} spaces $B^{s}_{p,q}$ in $\Dm_u$. 
Since $\Dm_u$ is upwards unbounded, 
the composite $(R_DL_u)^N$ is defined and  
has order $-N\delta$ on $\Dm_u$. 
So via embeddings, 
$(R_DL_u)^{N}$ maps any $B^{s}_{p,q}$ in $\Dm_u$ to $C^k(\overline{\Omega})$
for all sufficiently large $N$, hence it fulfills \eqref{exN-eq}. 
(This breaks down for the other linearisations
 in Remark~\ref{Lu-rem}, since they are not moderate.)
Clearly $(R_DL_u)^N$ is then also of order $0$ on every space in $\Dm_u$, so
since $P^{(N)}_{u}$ is a sum of such powers, it satisfies \eqref{allN-eq}.
Then \eqref{bsc'-eq}--\eqref{bsc''-eq} 
follow as identities in $B^{s_0}_{p_0,q_0}$,
since this space is in $\Dm(\cal A,Q)\subset\Dm_u$, 
in particular the parametrix formula is obtained.
As seen after \eqref{allN-eq} this also gives the desired regularity 
$u\in B^{t}_{r,o}$ at once.

\bigskip

The deduction of the parametrix formula after \eqref{Dmu-id}
is of course rather straightforward. However, this is partly because 
a few commutative diagrams have been suppressed in the
explanation. Moreover, it is easy to envisage that the arguments
extend to a whole range of eg semi-linear elliptic problems, and perhaps
it is most natural to comment on the generalisations first.

For other problems the domain $\Dm(\cal A,Q)$ of $\cal A$-moderacy 
will generally be more complicated than the polygon in \eqref{delQ-eq}. 
Eg it may be non-convex and operators corresponding to $(R_DL_u)^N$
can have orders bounded with respect to $N$ (unlike 
$-N\delta$). This is the case for composition type problems
with $Q(u)=F\circ u$ in \cite{JoRu97}; cf Figure~1 there.
Furthermore, the parameter domains can be `tight'
in the sense that they (unlike the above examples)
need not be upwards unbounded; here parabolic initial and
boundary value problems could be mentioned, for if the given data 
only fulfill finitely many compatibility conditions, then solutions can only
exist in $B^{s}_{p,q}$ for $s$ below a certain limit. Cf \cite{G4} for
the determination of specific compatibility conditions for fully
inhomogeneous problems.

In view of this, it seems practical to assume only that
the parameter domain $\Dm_u$ is connected.
Under this hypothesis it is possible to prove the existence of the
desired $N$ in \eqref{exN-eq} by continuous induction along an arbitrary 
curve from $(\fracnp,s)$ to $(\fracc nr,t)$, running inside the parameter
domain $\Dm_u$. 

These techniques are presented in Section~\ref{main-sect},
where the brief argument after \eqref{Dmu-id} is replaced by an analytical
proof of the parametrix formula. In fact 
the set-up in Section~\ref{main-sect} is both axiomatic 
and general, 
allowing also parabolic problems and linearisations of non-constant order. 
The last aspect might be important for problems with
linearisations of $F(u(x))$ at unbounded solutions $u$.

However, this paper mainly focuses on non-linearities with a product
structure, as there are ample examples of such problems,
and because more general classes would burden the exposition with 
more technicalities, or even atypical phenomena. Therefore generalised
multiplication is reviewed in Section~\ref{prod-sect}, and a class of
non-linearites of product type has been introduced in
Section~\ref{pdty-sect}; these are of the form $P_2(P_0u\cdot P_1u)$
for linear differential operators $P_j$ with constant coefficients.
As an example the von Karman equation is treated in Section~\ref{vK-sect}.
The abstract results of Section~\ref{main-sect} are exploited systematically
in Section~\ref{sys-sect} on general systems of semi-linear elliptic
boundary problems of product type.

\section{Preliminaries}   \label{prelim-sect}
\subsection{Notation}   \label{notation-ssect}
For simplicity $t_{\pm}:=\max(0,\pm t)$ for $t\in\R$.
The bracket $\kd{\mathsf{A}}$ stands for $1$ and $0$
when the assertion $\mathsf{A}$ is true resp. false.  
When $\alpha\in \N_0^n$ is a multiindex,
$D^\alpha:=(-\im)^{|\alpha|}\partial^{\alpha_1}_{x_1}\dots
\partial^{\alpha_n}_{x_n}$ where $|\alpha|=\alpha_1+\dots +\alpha_n$. 

The space of smooth functions with  compact support is denoted 
by $C^\infty_0(\Omega)$ or $\cal D(\Omega)$, 
when $\Omega\subset\Rn$ is open;
$\cal D'(\Omega)$ is the dual space of distributions on $\Omega$.
$\ang{u,\varphi}$ denotes the action of $u\in\cal D'(\Omega)$
on $\varphi\in C^\infty_0(\Omega)$. The restriction
$\rOm\colon \cal D'(\Rn)\to\cal D'(\Omega)$ is the
transpose of the extension by 0 outside of $\Omega$, denoted
$\eOm\colon C^\infty_0(\Omega)\to C^\infty_0(\Rn)$. 
Using this, $C^{\infty}(\overline{\Omega})=r_\Omega C^\infty(\Rn)$ etc.

The Schwartz space of rapidly decreasing $C^\infty$-functions 
is written $\cal S$ or $\cal S(\Rn)$, while $\cal S'(\Rn)$ stands for
the space of tempered distributions.
The Fourier transformation of $u$ is $\cal Fu(\xi)=\hat u(\xi)=\int_{\Rn}
e^{-ix\cdot\xi}u(x)\,dx$, with inverse  $\cal F^{-1}v(x)=\check v(x)$. 
The space of slowly increasing functions, ie $C^\infty$-functions $f$
fulfilling $|D^\alpha f(x)|\le c_\alpha\ang{x}^{N_\alpha}$ for all
mulitindices $\alpha$ is written $\cal O_M(\Rn)$;
hereby $\ang{x}=(1+|x|^2)^{1/2}$.

The singular support of $u\in \cal D'$, denoted $\singsupp u$, is the
complement of the largest open set on which $u$ acts a $C^\infty$-function.
Outside of $F:=\singsupp u$, mollification behaves as nicely as 
one could expect (the following could be folklore):
for $\psi\in C^\infty_0(\Rn)$, with 
$\psi_k(x)=\varepsilon_k^{-n}\psi(\varepsilon_k^{-1}x)$ for 
$0\le\varepsilon_k\to0$, one has 
\begin{equation}
  \psi_k*u\to c_0u \quad\text{in}\quad C^\infty(\Rn\setminus F); 
  \quad c_0=\int \psi\,dx.
  \label{psi*u-eq}
\end{equation}
For if $K\Subset\Rn$ with $K\cap F=\emptyset$ and 
$1=\varphi+\eta$  with $\varphi\in C^\infty_0(\Rn)$ 
and $\supp \varphi\cap F=\emptyset$, $K\cap\supp\eta=\emptyset$, 
uniform continuity of $D^\alpha(\varphi u)$ gives
$\sup_K|D^\alpha(\psi_k*\varphi u-c_0\varphi u)|\searrow 0$.
And by the theorem of supports, $\psi_k*(\eta u)=0$ near $K$, eventually.

It will later be convenient that this holds more generally for
$\psi\in \cal S(\Rn)$, even though
the convolution $\psi_k*(\eta u)$ need not vanish in $K$:

\begin{lem}   \label{singsupp-lem}
For $u\in \cal S'(\Rn)$, $\psi\in \cal S(\Rn)$, the regularising sequence
$\psi_k*u$ converges to $(\int\psi\,dx)\cdot u$ in the
$C^\infty$-topology on $\Rn\setminus\singsupp u$; 
ie it fulfills \eqref{psi*u-eq}.
\end{lem}
\begin{proof}
Continuing the above,
one has $0<\op{dist}(K, \supp \eta)\le 
\varepsilon_k \ang{\varepsilon_k^{-1}(x-y)}$ 
for $x\in K$, $ y\in \supp\eta$, and
\begin{equation}
  |\dual{\eta u}{D^\alpha_x\psi_k(x-\cdot)}|  
  \le c\sup_{y\in \Rn,\ |\beta|\le N}
        \ang{y}^N \bigl| D^{\beta}_y
        (\frac{\eta(y)}{\varepsilon_k^{n+|\alpha|}}
        D^\alpha\psi(\frac{x-y}{\varepsilon_k})) \bigr|.
\end{equation}
Since $\ang{y}^N\le c_K\ang{\tfrac{x-y}{\varepsilon_k}}^N$ for
$\varepsilon_k<1$, and   
powers of $\ang{\varepsilon_k^{-1}(x-y)}$ may be absorbed in
an $\cal S$-seminorm on $\psi$, it follows that $\sup_K |D^\alpha(\psi_k*(\eta u))|\le
C\varepsilon_k\searrow 0$. 
\end{proof}

\begin{rem}
Lemma~\ref{singsupp-lem} was called the Regular Convergence Lemma 
(and $\Rn\setminus \singsupp u$ the regular set of $u$)
in \cite{JJ08vfm}, where it played a significant role in
investigations of type $1,1$-operators and products.
\end{rem}

\subsection{Spaces}  \label{spaces-ssect}
Norms and quasi-norms 
are written $\norm xX$ for $x$ in a vector space $X$; recall that
$X$ is quasi-normed if the triangle inequality is replaced by the existence
of $c\ge1$ such that all $x$ and $y$ in $X$ fulfil
$\norm{x+y}X\le c(\norm xX+\norm yX)$ (``quasi-'' will be suppressed
when the meaning is settled by the context). 
Eg $L_p(\Rn)$ and $\ell_p(\N)$ for
$p\in\,]0,\infty]$ are quasi-normed with $c=2^{(\fracpi-1)_+}$;
this is seen because both
$\ell_p$ and $L_p$ for $0<p\le1$ satisfy the following, for $\lambda=p$, 
 \begin{equation}
 \nrm{f+g}{}\le(\nrm{f}{}^\lambda+\nrm {g}{}^\lambda)^{1/\lambda},
 \label{1.27}
 \end{equation}
where on the right H{\"o}lder's inequality applies to the dual exponents
$1/p$ and $1/(1-p)$.

For brevity $\nrm{f}{p}:=\norm{f}{L_p}$ for $f\in
L_p(\Omega)$, with $\Omega\subset\Rn$ an open set.
 $X_1\oplus X_2$ denotes the product space topologised by
$\norm{x_1}{X_1}+\norm{x_2}{X_2}$. 
For a bilinear operator $B\colon X_1\oplus X_2\to Y$,
continuity is equivalent to boundedness and to
existence of a constant $c$ such that $\norm{B(x_1,x_2)}{Y}\le
c\norm{x_1}{X_1}\norm{x_2}{X_2}$. 
In the affirmative case, the least possible $c$ is the operator norm  
$\nrm{B}{}=\sup \{\,\norm{B(x_1,x_2)}{Y}
\mid \text{for $j=1,2$}\colon \norm{x_j}{X_j}\le 1\,\}$.

The spaces $B^{s}_{p,q}(\Rn)$ and $F^{s}_{p,q}(\Rn)$
are, with conventions as in \cite{Y1}, defined as follows:
First a Littlewood--Paley decomposition is constructed using a
function $\Psi$ in $C^\infty(\R)$ for which $\Psi\equiv0$ and $\Psi\equiv1$
holds for $t\ge 13/10$ and $t\le 11/10$, respectively. Then $\Psi_j(\xi):=
\Psi(2^{-j}|\xi|)$ and
\begin{equation}
  \Phi_j(\xi)=\Psi_j(\xi)-\Psi_{j-1}(\xi)\qquad (\Psi_{-1}\equiv0)
  \label{LP-eq}
\end{equation}
gives $\Psi_j=\Phi_0+\dots+\Phi_j$ for every $j\in\N_0$, hence
$1\equiv\sum_{j=0}^\infty \Phi_j$ on $\Rn$.
As a shorthand $\varphi(D)$ will denote the pseudo-differential operator
with symbol $\varphi$, ie $\varphi(D)u=\cal F^{-1}(\varphi\cdot \cal
Fu)$, say for $\varphi\in\cal S(\Rn)$. 

For a {\em smoothness index $s\in\R$}, 
an \emph{integral-exponent $p\in\left]0,\infty\right]$} 
and  \emph{sum-exponent} $q\in\left]0,\infty\right]$, 
the {\em Besov space} $B^{s}_{p,q}(\Rn)$ 
and the {\em Lizorkin--Triebel space} $F^{s}_{p,q}(\Rn)$ are defined as
\begin{align}
 B^{s}_{p,q}(\Rn)&=\bigl\{\,u\in\cal S'(\Rn)\bigm|
 \Norm{ \{2^{sj} \norm{\Phi_j(D)u(\cdot)}{L_p} \}_{j=0}
 ^\infty}{\ell_q} <\infty\,\bigr\},
 \label{bspq-eq} 
    \\
 F^{s}_{p,q}(\Rn)&=\bigl\{\,u\in\cal S'(\Rn)\bigm|
 \Norm{ \norm{ \{2^{sj}\Phi_j(D) u\}_{j=0}
  ^\infty}{\ell_q} (\cdot)}{L_p} <\infty\,\bigr\}\,.
 \label{fspq-eq}
\end{align}
Throughout it will be tacitly understood that $p<\infty$ whenever
Lizorkin--Triebel spaces are under consideration.
$B^{s}_{p,q}(\Rn;\loc)$ etc.\  denote the spaces of
distributions that locally belong to the above ones.

The spaces are described in eg \cite{RuSi96,T2,T3,Y1}. 
They are quasi-Banach
spaces with the quasi-norms  given by the finite expressions in
\eqref{bspq-eq} and \eqref{fspq-eq}. Using \eqref{1.27} twice, they are seen
to fulfill \eqref{1.27}
for $\lambda=\min(1,p,q)$.

Among the embedding properties of these spaces  one has
$ B^{s}_{p,q}\hookrightarrow B^{s-\varepsilon}_{p,q}$
for $\varepsilon>0$, and if in the second line 
$\Omega\subset\Rn$ is open and bounded with
$B^{s}_{p,q}(\overline{\Omega}):=\rOm B^{s}_{p,q}(\Rn)$ endowed with the
infimum norm, 
\begin{align}
  B^{s}_{p,q}&\hookrightarrow B^t_{r,o}
  \quad\text{for}\quad s-\frac np=t-\frac nr,\ p>r;\ o=q,
\\
  B^{s}_{p,q}(\overline{\Omega})&\hookrightarrow
B^{s}_{r,q}(\overline{\Omega}) 
  \quad\text{for}\quad p\ge r.
  \label{fmem-eq}
\end{align}
The analogous holds for $F^{s}_{p,q}$, except that
$F^{s}_{p,q}\hookrightarrow F^t_{r,o}$ if only $s-\frac np=t-\frac nr$, $p>r$.
Moreover, $B^{s}_{p,q}\hookrightarrow L_\infty $ holds if and only if
$s>n/p$ or both $s=n/p$ and $q\le 1$; and $F^{s}_{p,q}\hookrightarrow
L_\infty $ if and only if $s>n/p$ or both $s=n/p$ and $p\le 1$.

\bigskip

For the reader's sake a few lemmas are recalled. They are concerned
with convergence of a series $\sum_{j=0}^\infty u_j$ fulfilling the dyadic
ball \emph{condition}: for some $A>0$
\begin{equation}
  \supp \cal Fu_j\subset \{\,\xi\in \Rn\mid |\xi|\le A2^j\,\},
\quad\text{for} \quad j\ge0.
  \label{DBC-cnd}
\end{equation}

\begin{lem}[The dyadic ball criterion]
  \label{B-lem}
Let $s>\max(0,\fracc np-n)$ for $p,q \in \,]0,\infty]$ and suppose
$u_j\in \cal S'(\Rn)$ fulfil \eqref{DBC-cnd} and
\begin{equation}
  B:=
  (\sum_{j=0}^\infty 2^{sjq}\nrm{u_j}{p}^q)^{\fracci1q}<\infty.
  \label{B-id}
\end{equation}
Then $\sum_{j=0}^\infty u_j$ converges in $\cal S'(\Rn)$ to some 
$u$ lying in $B^s_{p,q}(\Rn)$
and $\norm{u}{B^s_{p,q}}\le c B$ for some $c>0$ depending on
$n$, $s$, $p$ and $q$.
\end{lem}

\begin{lem}[The dyadic ball criterion]
  \label{F-lem}
Let $s>\max(0,\fracc np-n)$ for $0<p<\infty$, $0< q\le \infty$, and suppose
$u_j\in \cal S'(\Rn)$ fulfil \eqref{DBC-cnd} and
\begin{equation}
  F(q):=
  \Nrm{(\sum_{j=0}^\infty 2^{sjq}|u_j(\cdot)|^q)^{\fracci1q}}{p}<\infty.
\end{equation}
Then $\sum_{j=0}^\infty u_j$ converges in $\cal S'(\Rn)$ to some 
$u$ lying in $F^s_{p,r}(\Rn)$ for
\begin{equation}
  r\ge q,\quad r>\tfrac{n}{n+s},
  \label{r-eq}
\end{equation}
and $\norm{u}{F^s_{p,r}}\le cF(r)$ for some $c>0$ depending on
$n$, $s$, $p$ and $r$.
\end{lem}
This follows from the usual version in
which $s>\max(0,\fracnp-n,\fracc nq-n)$ is required, for one can just
pass to larger values of $q$ if necessary. Lemma~\ref{F-lem} emphasises
that the interrelationship between $s$ and $q$ is inconsequential for the
mere existence of the sum. 

It is also well known that the restrictions on $s$
can be entirely removed if $\sum u_j$ fulfils the dyadic \emph{corona}
condition: for some $A>0$, 
$\supp\cal Fu_0\subset\{\,|\xi|\le A \,\}$ and
\begin{equation}
  \supp \cal Fu_j\subset \{\,\xi\in \Rn\mid \tfrac{1}{A}2^j\le|\xi|\le
A2^j\,\},
  \quad\text{for}\quad j>0.
  \label{DCC-cnd}
\end{equation}

\begin{lem}[The dyadic corona criterion]
  \label{DCC-lem}
Let $u_j\in \cal S'(\Rn)$ fulfil \eqref{DCC-cnd} and \eqref{B-id}.
Then $\sum_{j=0}^\infty u_j$ converges in $\cal S'(\Rn)$ to some 
$u$ for which $\norm{u}{B^s_{p,q}}\le c B$ for some $c>0$ that depends on
$n$, $s$, $p$ and $q$.
And similarly for $F^{s}_{p,q}(\Rn)$, if $F(q)<\infty$.
\end{lem}

These lemmas are proved in eg \cite{Y1}.
To estimate the numbers $B$ and $F$ in the above criteria, the
following summation lemma is often useful: for 
any sequence $(a_j)$ in $\C$, $s<0$ and $q$, $r\in \,]0,\infty]$, 
\begin{equation}
  (\sum_{j=0}^\infty 2^{sjq}(\sum_{k=0}^j
|a_k|^r)^{\tfrac{q}{r}})^{\fracci1q}
  \le c(s,q,r)
  \norm{ 2^{sj}a_j}{\ell_q}.
  \label{slem-eq}
\end{equation}
For $s>0$ the analogous holds if the second sum is over $k\ge j$ instead.
(Cf \cite[Lem.~3.8]{Y1} for $r=1$.)

For the estimates of the exact paralinearisation in Section~\ref{brdl-ssect}
and \ref{sym-ssect}, the following vector-valued
Nikol$'$ski\u\i--Plancherel--Polya inequality will be convenient. 
\begin{lem}
  \label{vNPP-lem}
Let $0<r<p<\infty$, $0<q\le\infty$ and $A>0$.
There is a constant $c$ such that for every sequence of functions  
$f_k\in L_r(\Rn)\cap\cal S'(\Rn)$
with $\supp \cal Ff_k\subset B(0,A2^k)$, 
\begin{equation}
  \Norm{(\sum_{k=0}^\infty |f_k|^q)^{1/q}}{L_p}\le
  c\Norm{\sup_k 2^{(\fracci nr-\fracci np)k}|f_k|}{L_r}.
\end{equation}
\end{lem}
The ordinary Nikol$'$ski\u\i--Plancherel--Polya inequality results
from this if $f_k\ne0$ holds only for one value of $k$. 
(Lemma~\ref{vNPP-lem} itself can
be reduced to this version by means of an elementary inequality in
\cite[Lem.~4]{BrMi01}, cf \cite{JoSi07}.)

\bigskip

To treat the examples in Proposition~\ref{x32-prop} below
tensor products will be useful. 
However, lacking a thorough reference to 
this, the next result is given.
It improves \cite[Prop.~2.7]{JoRu97} by including the case $s>0$.
A proof using the above dyadic corona criterion is supplied, partly because
\cite[Prop.~2.7]{JoRu97} was stated without details, partly because   
it then is more natural to omit the details behind the better known, but
analogous, paraproduct estimates recalled in Remark~\ref{pmest-rem} below.

\begin{lem}
  \label{tensor-lem}
The continuous map $(u,v)\mapsto u\otimes v$ from $\cal
S'(\R^{n'})\times\cal S'(\R^{n''})$ to $\cal S'(\R^{n'+n''})$
restricts to bounded bilinear maps
\begin{align}
  B^{s}_{p,q}(\R^{n'})\times B^{s}_{p,q}(\R^{n''}) & \to
  B^{s}_{p,q}(\R^{n'+n''}) \quad\text{for}\quad s>0,
  \label{Bs-eq} \\
  B^{s'}_{p,q}(\R^{n'})\times L_{p}(\R^{n''}) & \to
  B^{s'}_{p,q}(\R^{n'+n''}) \quad\text{for}\quad s'<0, 1\le p\le\infty,
  \label{Bs'-eq} \\
  B^{s'}_{p,q'}(\R^{n'})\times B^{s''}_{p,q''}(\R^{n''}) & \to
  B^{s'+s''}_{p,q}(\R^{n'+n''}) \quad\text{for}\quad s',s''<0, 
  \fracc1q=\fracc1{q'}+\fracc1{q''}.
  \label{Bs''-eq} 
\end{align}
\end{lem}

\begin{proof}
For $u\in \cal S'(\R^{n'})$ and $v\in \cal S'(\R^{n''})$ 
there is a decomposition, 
when $\Psi'_N=\Phi'_0+\dots+\Phi'_N$ refers to a
Littlewood--Paley decomposition on 
$\R^{n'}$ with the present conventions,
so that $u_k=\Phi'_k(D)u$, $u^k=\Psi'_k(D)u$, 
and similarly for $v$ on $\R^{n''}$, 
\begin{equation}
  u\otimes v= \lim_{N\to\infty}\cal F^{-1}((\Psi'_N\otimes \Psi''_N)
    \cal F(u\otimes v)) =\sum_{k=0}^\infty (u_kv^{k-1}+u^kv_{k}).
\end{equation}
Both series on the right-hand side 
fulfill the dyadic corona condition
\eqref{DCC-cnd},
since $\xi=(\xi',\xi'')$ for $\xi'\in \R^{n'}$, $\xi''\in
\R^{n''}$ and $|(\xi',0)|\le|\xi|\le|\xi'|+|\xi''|$ yield eg
\begin{equation}
  \xi\in \supp\cal F(u_{k}v^{k-1})\implies
  \tfrac{11}{20}2^k\le |\xi| \le \tfrac{13}{10}(2^k+2^{k-1})=\tfrac{39}{20}2^k.
\end{equation}
For $1\le p\le\infty$ the usual convolution estimate gives
\begin{equation}
  2^{sk}\norm{u_kv^{k-1}}{L_p(\R^{n'+n''})}\le
  2^{ks}\nrm{u_k}{p}\nrm{\cal F^{-1}\Psi''}{1}\nrm{v}{p},
  \label{Lp-eq}
\end{equation}
and since $B^{s}_{p,q}\hookrightarrow L_p$ for $s>0$, $p\ge1$, it follows
from Lemma~\ref{DCC-lem} by calculation of the $\ell_q$-norms that
\begin{equation}
  \Norm{\sum u_kv^{k-1}}{B^{s}_{p,q}(\R^{n'+n''})}\le
  c\norm{u}{B^{s}_{p,q}(\R^{n'})}\norm{v}{B^{s}_{p,q}(\R^{n''})}.
\end{equation}
The other series is treated the same way, and thus follows \eqref{Bs-eq} for
$p\ge1$. For $p<1$ one has $\nrm{v^{k-1}}{p}\le
\nrm{|v_0|+\dots+|v_k|}{p}\le\norm{v}{F^0_{p,1}}\le c\norm{v}{B^{s}_{p,q}}$,
and this instead of \eqref{Lp-eq} extends
\eqref{Bs-eq} to all $p\in \,]0,\infty]$.

Since \eqref{Lp-eq} holds for all $s$, it suffices for \eqref{Bs'-eq} 
to estimate $\sum u^kv_k$.
By \eqref{slem-eq},
\begin{equation}
  \sum 2^{s'kq}\nrm{u^k}{p}^q\le 
  \sum 2^{s'kq}(\nrm{u_0}{p}+\dots+\nrm{u_k}{p})^q\le
  c\sum 2^{s'kq}\nrm{u_k}{p}^q.
  \label{s'-ineq}
\end{equation}
Using Lemma~\ref{DCC-lem}, it follows as above that
$\norm{\sum u^kv_k}{B^{s'}_{p,q}}\le c\norm{u}{B^{s'}_{p,q}}\nrm{v}{p}$.

To prove \eqref{Bs''-eq} one can use the summation lemma for both $u^k$ and
$v^{k-1}$ since both $s'$, $s''<0$. Combining this with
H{\"o}lder's inequality for $\ell_{q}$, the above procedure gives a bound of
$\norm{u\otimes v}{B^{s'+s''}_{p,q}}$ by
$c\norm{u}{B^{s'}_{p,q'}}\norm{v}{B^{s''}_{p,q''}}$. 
\end{proof}

\subsection{Examples}
  \label{exmp-ssect}

The delta measure $\delta_0\in B^{\fracci np-n}_{p,\infty}(\Rn)$ for
$0<p\le\infty$; this well-known fact follows directly from \eqref{bspq-eq}
since $2^{j(\fracci np-n)}\nrm{\check\Phi_j}{p}$ is $j$-independent.

Other examples include $|x|^a$, that for $a>-n$ and $p\ge 2$ was shown 
in \cite{Y3} to be
locally in $B^{\fracci np+a}_{p,\infty}(\Rn)$ at $x=0$. 
For $0<p\le\infty$ there is a technical treatment via
differences in \cite[Sect.2.3]{RuSi96} of $|x|^a|\log|x||^{-b}$, but without
details for the case $b=0$ that is used in the present paper.

As a novelty, the Regular Convergence Lemma (Lemma~\ref{singsupp-lem}) 
yields a direct argument
for a large class of homogeneous \emph{distributions}:  
recall that $u\in \cal D'(\Rn)$ is homogeneous
of degree $a\in \C$ if 
\begin{equation}
  \dual{u}{\varphi}=t^{n+a}\dual{u}{\varphi(t\cdot)},\quad\forall t>0,
  \quad \forall\varphi\in C^\infty_0(\Rn).
    \label{homo-eq}
\end{equation}
When $u\in \cal S'(\Rn)$ this extends to $\varphi\in \cal S(\Rn)$ by
closure. This applies to the Littlewood--Paley decomposition
$1=\sum_{j=0}^{\infty}\Phi_j$, where 
$\Phi_j(\xi)=\Phi(2^{-j}\xi)$ for $j\ge1$ and a fixed $\Phi\in C^\infty$ 
(namely $\Phi=\Psi(|\cdot|)-\Psi(2|\cdot|)$), so \eqref{homo-eq} gives directly
\begin{equation}
  2^{ja}\Phi_j(D)u(x)=\dual{u}{2^{j(a+n)}\check\Phi(2^jx-2^j\cdot)}
    =\Phi(D)u(2^jx).
  \label{Phiu-eq}
\end{equation}
Therefore $2^{j(\fracci np+\Re a)}\nrm{\Phi_j(D)u}{p}=
\nrm{\Phi(D)u}{p}$, which is a constant independent of $j$.
This can be exploited if
$u$ is assumed to have the origin as the only singularity:

\begin{prop}   \label{homo-prop}
Let $u\in \cal D'(\Rn)$ be $C^\infty$ on $\Rn\setminus\{0\}$ and homogeneous
of degree $a\in \C$ there, ie \eqref{homo-eq} holds for all 
$\varphi\in C^\infty_0(\Rn\setminus\{0\})$.

Then $u$ is locally at $x=0$ in 
$B^{\fracci np+\Re a}_{p,\infty}(\Rn)$ for $0<p\le\infty$. 
If $-n<\Re a<0$ it holds for 
$-\tfrac{n}{\Re a}<p\le\infty$ that $u\in B^{\fracci np+\Re
a}_{p,\infty}(\Rn)$; this holds also for $p=\infty$ if $\Re a=0$.
The Besov space conclusions are sharp with
respect to $s$ and $q$, unless $u$ is a homogenenous polynomial (which is
the only case in which $u\in  C^\infty(\Rn)$).
\end{prop}

\begin{proof}
The function $D^\alpha u$ on $\Rn\setminus\{0\}$
acts on $\varphi$ like $t^{a-|\alpha|}D^\alpha u
(t^{-1}x)$. Hence $D^\alpha u$ has degree $a-|\alpha|$, and
$t=|x|$ entails $|D^\alpha u(x)|\le c_\alpha |x|^{\Re a-|\alpha|}$ 
for $x\ne0$, all $|\alpha|\ge 0$.

If $u\in C^\infty(\Rn)$, the homogeneity of $D^\alpha u$ gives
$D^\alpha u\equiv0$ for $\Re a-|\alpha|<0$ 
(otherwise $D^\alpha u$ would be discontinuous at $x=0$), and that
$D^\alpha u(0)=0$ for $\Re a-|\alpha|>0$. Therefore Taylor's formula   
gives at once that $u\equiv0$ if $\Re a\notin\N_0$, or else that
$u$ is a homogeneous polynomial (and $a\in \N_0$). 

The homogeneity and smoothness on $\Rn\setminus\{0\}$ together
imply that $u\in \cal S'$
with $\cal F u$ in $C^\infty(\Rn\setminus\{0\})$. This is known,
cf \cite[Thm~7.1.18]{H},
but easy to see with a few ideas used here anyway: 
for $\chi\in C^\infty_0(\Rn)$, $\chi(0)=1$, one has $u=\chi
u+(1-\chi)u$, where the second term is in
$\cal O_M$ by the above, ie $u\in \cal E'+\cal O_M\subset\cal S'$. And 
$\xi^{\alpha}D^\beta\hat u=
\cal FD^\alpha((-x)^\beta u)$ is in $\cal F(\cal E'+L_1)\subset C^0$ for
$|\alpha|>a+|\beta|+n$, so $\hat u$ is $C^\infty$ for $x\ne 0$. 

By the Paley--Wiener--Schwartz Theorem 
$\Phi_j(D)(\chi u)\in \cal S(\Rn)$. In particular for $j=0$ this gives
$\nrm{\Phi_0(D)(\chi u)}{p}<\infty$, while
for $j\ge1$ it follows from \eqref{homo-eq} ff that, 
cf \eqref{Phiu-eq},
\begin{equation}
  \Phi_j(D)(\chi u)(x)=\dual{u}{\chi2^{jn}\check\Phi(2^jx-2^j\cdot)}
    =2^{-ja}\Phi(D)(\chi(2^{-j}\cdot)u)(2^jx).
  \label{Phichiu-eq}
\end{equation}
Here $\chi(2^{-j}\cdot)$ is handled with Lemma~\ref{singsupp-lem}, for 
since $\Phi=0$ near $\singsupp \hat u=\{0\}$, 
\begin{equation}
  (2\pi)^{-n}\Phi(\hat u*(2^{jn}\hat\chi(2^{j}\cdot)))
  \to \chi(0)\Phi\hat u\quad\text{in}\quad C^\infty_0(\Rn).
\end{equation}
Then the continuity of the embedding $\cal S\hookrightarrow L_p$ and of the
quasi-norm $\nrm{\cdot}{p}$ gives
\begin{equation}
 \lim_{j\to\infty} 2^{j(\fracci np+\Re a)}\nrm{\Phi_j(D)(\chi u)}{p}=
 \nrm{\Phi(D)u}{p}<\infty.
\end{equation}
Hence $\chi u\in  B^{\fracci np+\Re a}_{p,\infty}(\Rn)$ for all $p$.
Note that the right hand side is zero if and only if 
$\Phi\hat u\equiv 0$, that by the homogeneity of $\hat u$
is equivalent to $\supp\hat u\subset\{0\}$, 
that holds if and only if $u$ is a polynomial.

For $-n<\Re a<0$ and some $p\in \,]-\tfrac{n}{\Re
a},\infty]\subset \,]1,\infty]$, note first that
\begin{equation}
  |x|^{\Re a} \text{ is in $L_p$ for $|x|>1$} \iff   p\Re a<-n.
  \label{Lp1-eq}
\end{equation}
Since $L_p* L_1\subset L_p$ for $p\ge1$, 
it follows that $\check\Phi_0*u$ belongs to $L_p+\cal S\subset L_p$.
Likewise $\check\Phi_0*u\in L_\infty$ for $\Re a=0$, for $u$ is bounded for
$|x|>1$ by the first part of this proof.
Now \eqref{Phiu-eq} gives that 
$2^{j(\fracci np+\Re a)}\nrm{\Phi_j(D)u}{p}$ equals $\nrm{\Phi(D)u}{p}$,
which is finite since $\Phi\hat u\in C^\infty_0$. Therefore 
$u\in B^{\fracci np+\Re a}_{p,\infty}(\Rn)$.

Because $L_p\supset B^{\fracci np+\Re a}_{p,\infty}$ for $\Re a+\fracnp>0$, the
range for $p$ is sharp, up to the end point $p=-n/\Re a $, by \eqref{Lp1-eq}.
Since the other
Besov space conclusions follow from \emph{identities}, the spaces
$B^{\fracci np+\Re a}_{p,\infty}$ are optimal (unless $u$ is a polynomial).
\end{proof}

\begin{rem}
By Proposition~\ref{homo-prop},  
$\tfrac{P(x)}{Q(x)}\in B^{\fracci np}_{p,\infty}(\Rn;\loc)$, 
$0<p\le\infty$, for two homogeneous
polynomials $P$, $Q$ both of degree $a\ge1$ 
such that $Q(x)=0$ only for $x=0$.
In case $P\ne Q$ are real and $n\ge2$, this has a special singularity 
since every neighbourhood of the origin 
is mapped onto the proper interval
$[\min_{|x|=1}\tfrac{P}{Q},\max_{|x|=1}\tfrac{P}{Q}]$.
But the obtained Besov regularity 
$B^{\fracci np}_{p,\infty}$
is the same as the well-known one for simple 
jump discontinuities across a hyperplane.
\end{rem}
Invoking Lemma~\ref{tensor-lem},
the above analysis now leads to results for homogeneous distributions that
are constant in $n-k$ variables. A local version is given with optimal results
for $1\le p<\infty$.

\begin{prop}
  \label{x32-prop}
If $\Omega\subset\Rn$, $n\ge2$ is an open set with $0\in\Omega$ and the
variables are split as $x=(x',x'')$ for $x'=(x_1,\dots,x_k)$, 
$x''=(x_{k+1},\dots,x_n)$, then it holds for every 
$u(x')$ in $\cal D'(\R^k)$ that is homogeneous of degree $a\in \C$ and
$C^\infty$ outside of the origin that
\begin{equation}
  f(x)=\rOm [u(x')\otimes 1(x'')]
\end{equation}
belongs to $B^{s}_{p,\infty}(\overline{\Omega})$ for
$s\le\fracc kp+\Re a$, except possibly for $p=\infty$ if $\Re a=0$. 
For $p\ge 1$ this result is sharp with respect to $s$. 
\end{prop}
\begin{proof}
By Proposition~\ref{homo-prop},
it follows from \eqref{Bs-eq}--\eqref{Bs'-eq} 
that $v(x)=(\varphi_1(x')u(x'))\otimes\varphi_2(x'')$
is in $B^{\fracci kp+\Re a}_{p,\infty}(\Rn)$ for 
$\fracc kp+\Re a\ne0$ when the $\varphi_j$ are both in $C^\infty_0$. 
For $\Re a=0$ this excludes $p=\infty$, while for $\Re a<0$ 
a gap is left at $p_0=k/(-\Re a)$, but this  
can be closed by H{\"o}lder's inequality, for if 
$\fracc1{p_0}=\tfrac{1}{2p_1}+\tfrac{1}{2p_2}$ for some
exponents $p_1<p_0<p_2$, 
then each $j\ge 0$ yields
$\nrm{v_j}{p_0}\le
\prod_{m=1,2}(2^{j(\tfrac{k}{p_m}+\Re a)}\nrm{v_j}{p_m})^{\tfrac{1}{2}}
\le \prod \norm{v}{B^{\tfrac{k}{p_m}+\Re a}_{p_m,\infty}}^{1/2}$.
Taking $\varphi_1\otimes\varphi_2$ equal to $1$ on $\Omega$ 
one finds, with the mentioned exception $p=\infty $ for $\Re a=0$,
\begin{equation}
  f=\rOm v\in B^{s}_{p,\infty}(\overline{\Omega})
  \quad\text{for}\quad s=\fracc kp +\Re a,\ 0<p\le\infty.
\end{equation}
Conversely, if $f$ is in this space for some $s$, it holds that
$w=(\theta_1\otimes\theta_2)f\in B^{s}_{p,\infty}(\Rn)$ when the $\theta_j\in 
C^\infty_0$ are supported sufficiently close to the origin.
The support of $\Phi'_j(\xi')\Phi''_0(\xi'')\hat w$ 
intersects that of $\Phi_k(\xi)$ only for $|j-k|\le2$,
so for $p\ge1$ the convolution result $L_p*L_1\subset L_p$ gives
$\nrm{\Phi'_j(D')\Phi''_0(D'') w}{p}\le
c\sum_{|h|\le2}\nrm{\Phi_{j+h}(D)w}{p}$. Consequently
\begin{multline}
  \sup_{j\ge0} 2^{sj}\norm{\Phi''_0(D'')\theta_2}{L_p(\R^{n-2})}
   \norm{\Phi'_j(D')(\theta_1 f)}{L_p(\R^2)}
  \\
  \le c\sup_{j\ge0} \sum_{|h|\le 2} 2^{sj}\nrm{\Phi_{j+h}(D)w}{p}
  \le c_1 \norm{w}{B^{s}_{p,\infty}}.
\end{multline}
Taking $\theta_2$ positive yields $\hat\theta_2(0)=\int
\theta_2\ne0$, so $\nrm{\Phi''_0(D'')\theta_2}{p}>0$ and
as a result of this $\norm{\theta_1f}{B^{s}_{p,\infty}(\R^2)}<\infty$. 
Then Proposition~\ref{homo-prop} gives $s\le\fracc kp+\Re a$. 
\end{proof}

Since $\delta_0$ has degree $-n$ on $\Rn$, it is a special case that, 
for $0\in \Omega$, $x=(x',x_n)$,
\begin{equation}
  f(x)=1(x')\otimes\delta_0(x_n)\quad\text{is in}\quad
  B^{\fracci1p-1}_{p,\infty}(\overline{\Omega}),\ 0< p\le\infty.
  \label{1delta-ex}
\end{equation}

\section{The general parametrix construction}
  \label{main-sect}
\subsection{An abstract framework}
For the applicability's sake Theorem~\ref{ir-thm} below is proved 
in a general set-up. If desired, the
reader may think of the spaces $X^s_p$ as
$H^s_p(\overline{\Omega})$ and consider $A$ to be an elliptic operator like
$\left(\begin{smallmatrix}\mlap\\ \gamma_0\end{smallmatrix}\right)$ etc.
The concepts in Section~\ref{Dm-ssect} are used freely, in particular
this is so for parameter domains.

\bigskip

In the following five axioms, $n\in\N$ and $d\in\R$ are two fixed numbers,
playing the role of the dimension and the order of the linear
operator $A$, respectively, and $I$ denotes the identity map:

\begin{Rmlist}
  \item   \label{XY-cnd}
Two scales $X^s_p$ and $Y^s_p$ of vector spaces are given with $(s,p)$
in a common parameter set ${\Sdm}\subset\R\times\,]0,\infty]$. 
In the $X^s_p$-scale there are the usual simple, Sobolev and finite-measure
embeddings; ie for $(s,p),\,(t,r)\in {\Sdm}$,
\begin{align}
  X^s_p&\subset X^{s-\varepsilon}_p \quad\text{when}\quad
   \varepsilon>0,
  \label{s-emb} 
  \\
  X^s_p&\subset X^t_r \quad\text{when}\quad
 s\ge t \quad\text{and}\quad s-\fracnp=t-\fracc nr,
  \label{sob-emb}
\\
  X^s_p&\subset X^{s}_r \quad\text{when}\quad       
  p\ge r.
  \label{fm-emb}
\end{align}

  \item  \label{AXY-cnd} There is a linear map $A:=A_{(s,p)}$, 
with  parameter domain $\Dm(A)\subset\Sdm$,
\begin{equation}
  A\colon X^s_p\to Y^{s-d}_p,\quad (s,p)\in \Dm(A).
\end{equation}
There is also for all $(s,p)\in\Dm(A)$, 
a linear map $\widetilde{A}\colon Y^{s-d}_p\to X^s_p$  such that 
\begin{equation}
  \cal R:=I_{X^s_p}-\widetilde{A}A \quad\text{has range in}\quad
  \bigcap_{(s,p)\in\Dm(A)} X^s_p.
  \label{R-eq}
\end{equation}

Inclusions $\bigcup_{\Dm(A)}X^s_p\subset\cal X$ and
$\bigcup_{\Dm(A)}Y^{s-d}_p\subset\cal Y$ 
hold for some  vector spaces $\cal X$, $\cal Y$;
and for $(s,p)$, $(t,r)\in\Dm(A)$ there is a commutative diagram
\begin{equation}
\begin{CD}
 X^s_p\cap X^t_r    @>I>> X^s_p    \\
 @VIVV               @VVA_{(s,p)}V \\
 X^t_r      @>A_{(t,r)}>> \cal Y .
\end{CD}  
  \label{Acd-eq} 
\end{equation}
Likewise $\widetilde{A}$ should be
unambiguously
defined on $Y^{s-d}_p\cap Y^{t-d}_r$.

  \item   \label{B-cnd} There is a non-linear operator $\cal N$, with
parameter domain $\Dm(\cal N)\subset\Sdm$, which for
every $(s_0,p_0)$ in $\Dm(\cal N)$ and 
every  $u\in X^{s_0}_{p_0}$ has a linearisation $B_u$, 
ie $\cal N(u)=-B_u(u)$, where $B_u$ is a linear map
\begin{equation}
  B_u\colon X^s_p\to Y^{s-d+\delta(s,p)}_p
\quad\text{with}\quad \Dm(B_u)\supset \Dm(\cal N). 
\end{equation}
For $(s,p)$, $(t,r)\in \Dm(B_u)$ there is a commutative diagram analogous
to \eqref{Acd-eq} for $B_u$ (hence for $\cal N$).

  \item   \label{Dm-cnd} For $u$ as in \eqref{B-cnd}, the domain
$\Dm(A)\cap\Dm(B_u)$ is connected with
respect to the metric $\op{dist}((s,p),(t,r))$ given by 
$((s-t)^2+(\fracnp-\fracc nt)^2)^{1/2}$.

  \item   \label{del-cnd} For $u$ as in \eqref{B-cnd}, the function
$\delta(s,p)$ satisfies 
\begin{gather}
  (s+\delta(s,p),p)\in\Dm(A) 
   \quad\text{for every}\quad (s,p)\in\Dm(A)\cap\Dm(B_u), 
  \label{delta-lift}\\
  \inf\Set{\delta(s,p)}{(s,p)\in K} >0
  \quad\text{for every}\quad K\Subset \Dm(A)\cap\Dm(B_u).
  \label{delK-cnd}
\end{gather}  
\end{Rmlist}
For the proof of Theorem~\ref{ir-thm} below it is unnecessary to assume that
the embeddings in \eqref{XY-cnd} should hold for the $Y^s_p$ spaces too
(although they often do so in practice). As it stands
\eqref{XY-cnd} is easier to verify in applications to
parabolic boundary problems; cf Remark~\ref{parabolic-rem} below.

For $X^s_p=H^s_p(\overline{\Omega})$ it is natural to let
$\Sdm=\R\times\,]1,\infty[\,$; the $L_2$-theory comes out for $\Sdm=\R\times
\{2\}$. Besov spaces $B^{s}_{p,q}$ would often require $q$ to be fixed and
$\Sdm=\R\times\,]0,\infty]$. Anyhow $\cal X=\cal D'(\Omega)$ could be a
typical choice.
Continuity of $A$ and $\widetilde{A}$ is not required
(although both will be bounded in most applications). 

Suppressing $(s,p)$ in $A$ is harmless in the sense that 
$A$ by \eqref{Acd-eq} is a well-defined map with domain
$\bigcup_{\Dm(A)}X^s_p$ in $\cal X$; it is linear only on each `fibre' $X^s_p$.
Similarly $\widetilde{A}$ is a map on $\bigcup_{\Dm(A)}Y^{s-d}_p$.
Moreover, $A$ eg extends to a linear map on the algebraic direct sum
$\bigoplus X^s_p\subset\cal X$ if and only if
(when $'$ indicates finitely many non-trivial vectors)
\begin{equation*}
  0=\sideset{}{'}\sum_{\Dm(A)}v_{(s,p)}\implies
  \sum_{\Dm(A)}A_{(s,p)}(v_{(s,p)})=0.
\end{equation*}
By \eqref{Acd-eq} ff, $\cal R$ may be thought of as an operator from
$\bigcup_{\Dm(A)} X^s_p$ to $\bigcap_{\Dm(A)} X^s_p$.

For brevity the arguments $s_0$, $p_0$ are suppressed in the
function $\delta$. By \eqref{B-cnd}, the map $\cal N$ 
sends $X^s_p$ into $Y^{s-d+\delta(s,p)}_{p}$ for 
each $(s,p)$ in $\Dm(\cal N)$ 
(since $\Dm(\cal N)\subset\Dm(B_u)$ for every $u$ in $X^s_p$).
This fact will be used tacitly.
Note that $\delta(s,p)>0$ by \eqref{delK-cnd}, so \eqref{B-cnd} implies
that $\cal N(u)$ has $B_u$ as a moderate linearisation with
$\omega=d-\delta(s,p)$, according to Definition~\ref{mod-defn}.

Via the transformation $(s,p)\mapsto(\fracnp,s)$, the reader should
constantly think of $\Dm(A)$, $\Dm(\cal N)$ and $\Dm(B_u)$ 
as subsets of  $[0,\infty[\,\times \R$.  
In the examples the boundary of $\Dm(\cal N)$ (or a part thereof) often
consists of the 
 $(s_0,p_0)$ for which $\delta\equiv 0$, so it may seem natural to require
 $\Dm(\cal N)$ to be open in $[0,\infty[\,\times\R$. 
However, such an assumption is avoided because it is unnecessary and
potentially might exclude application to weak solutions of certain problems;
cf the below Section~\ref{vK-sect}. 

The function $\delta$ is in practice 
often constant with respect to $(s,p)$, but  
depending effectively on $(s_0,p_0)$; cf Remark~\ref{Lu-rem}.  
When this is the case and furthermore
 $\cal N$ has a natural parameter domain $\Dm(\cal N)$ on which $\delta$ can
take both positive and negative values, it is natural to use
\begin{equation}
 \Dm(\cal N,\delta)=\{\,(s_0,p_0)\in\Dm(\cal N)\mid \delta>0\,\} 
\end{equation}
as the parameter domain of
$\cal N$, instead of $\Dm(\cal N)$. 
Then $\cal N$ will be $A$-moderate on the domain
\begin{equation}
 \Dm(A,\cal N)=\Dm(A)\cap\Dm(\cal N,\delta).
\end{equation}
With $\sigma(s,p):=d-\delta(s,p)$, it is clear that $\Dm(A,\cal N)$
is a generalisation of the domain $\Dm(\cal A,Q)$ introduced for the model
problem in \eqref{delQ-id}. We now return to this.

\begin{exmp}
  \label{modl-ex}
To elucidate \eqref{XY-cnd}--\eqref{del-cnd} above, one may
in \eqref{bsc-eq}
set $A=\left(\begin{smallmatrix}\mlap\\ \g \end{smallmatrix}\right)$
and $X^s_p=B^{s}_{p,q}(\overline{\Omega})$, whereby $q\in\,]0,\infty]$
is kept fixed. For the operator $\widetilde{A}$ there is a
parametrix of $A$ belonging to the Boutet de~Monvel calculus
(cf Section~\ref{ell-ssect} below). Using $L_u$ from \eqref{Lu-eq}, 
$B_u$ and $Y^s_p$ are taken as
\begin{equation}
  B_uv=  \begin{pmatrix} L_u v\\ 0\end{pmatrix}, 
  \qquad
  Y^{s-2}_p= 
  B^{s-2}_{p,q}(\overline{\Omega})
  \oplus
  B^{s-\fracpi}_{p,q}(\Gamma).
\end{equation}
For any $\varepsilon\in\,]0,1[$ it is possible to take $\delta(s,p)$
as the constant function 
\begin{equation}
  \delta(s,p)=
  \begin{cases}
   1 &\text{for } s_0>\fracc n{p_0},\\
   1-\varepsilon &\text{for } s_0=\fracc n{p_0},\\
   s_0-\fracc n{p_0}+1 &\text{for } \fracc n{p_0}> s_0>\fracc n{p_0}-1.
  \end{cases}
  \label{dsp-eq}
\end{equation}
See the below Theorem~\ref{mod-thm}. As mentioned in Remark~\ref{mod-rem},
this theorem and Corollary~\ref{Amod-cor} also gives the parameter domains,
for any fixed $u\in X^{s_0}_{p_0}$, 
\begin{align}
  \Dm(A)&=
        \set{(s,p)}{s>\fracp+(n-1)(\fracp-1)_+}=\Dm_1,
  \\
  \Dm(\cal N)&=\Set{(s,p)}{s>\tfrac{1}{2}+
         (\fracnp-\tfrac{n}{2})_+},
  \label{Ndm-eq}
\\
  \Dm(\cal N,\delta)&=\Set{(s,p)}{s>\tfrac{1}{2}+
         (\fracnp-\tfrac{3}{2}+\tfrac{1}{2}\kd{n=2})_+}
  = \Dm(A,\cal N),
  \label{Ndm-eq'}
\\
    \Dm(B_u)&=\Set{(s,p)}{s>1-s_0+
         (\fracnp+\fracc n{p_0}-n)_+}.
  \label{Bdm-eq}
\end{align}
Being isometric to a polygon in $[0,\infty[\,\times\R$,
the set $\Dm(A)\cap\Dm(B_u)$ clearly satisfies \eqref{Dm-cnd}; 
when $(s_0,p_0)\in \Dm(A)\cap\Dm(B_u)$, then
condition~\eqref{del-cnd} may be verified directly from \eqref{dsp-eq}.
\end{exmp}

\subsection{The Parametrix Theorem}
Using the above abstract framework, it is now possible to establish
a main result of the article in a widely applicable version.
 
\begin{thm}
  \label{ir-thm}
Let $X^s_p$, $Y^s_p$ and the mappings $A$ and $\cal N$ be given
such that conditions \eqref{XY-cnd}--\eqref{del-cnd} above are satisfied.
\begin{itemize}
  \item[(1)] For every 
\begin{equation}
  u\in X^{s_0}_{p_0} \quad\text{with $(s_0,p_0)\in\Dm(A)\cap\Dm(\cal N)$}
  \label{uhyp-eq}
\end{equation}
the parametrix $P^{(N)}=\sum_{k=0}^{N-1}(\widetilde{A}B_u)^{k}$
is for every $N\in\N$ a linear operator 
\begin{equation}
  P^{(N)}\colon X^{s}_{p}\to X^{s}_{p}  
\quad\text{for all $(s,p)\in\Dm(A)\cap\Dm(B_u)=:\Dm_u$.}
\end{equation}
And for every $(s',p')$, $(s'',p'')\in\Dm_u$ 
there exists $N'\in\N$ such that the
``error term'' $(\widetilde{A}B_u)^{N}$ is a linear map
\begin{equation}
  (\widetilde{A}B_u)^{N}\colon X^{s'}_{p'}\to X^{s''}_{p''}
  \quad\text{for}\quad N\ge N'.
  \label{err-eq}
\end{equation}

  \item[{(2)}]
If some $u$ fulfils \eqref{uhyp-eq} and solves the equation 
\begin{align}
  Au+\cal N(u)&=f
  \label{smln-eq} \\
  \text{with data}\quad f&\in Y^{t-d}_r\quad\text{for some $(t,r)\in \Dm_u$},
  \label{fhyp-eq}  
\end{align}
one has for every $N\in\N$ the parametrix formula
\begin{equation}
  u=P^{(N)}(\widetilde{A}f+\cal Ru)+ (\widetilde{A}B_u)^Nu.
  \label{par-eq}
\end{equation}
And consequently $u\in X^t_r$ too.
\end{itemize}
\end{thm}

\begin{proof}
For arbitrary $(s,p)\in \Dm_u$, one can use
\eqref{AXY-cnd} and \eqref{delta-lift} to see 
that $\widetilde{A}$ is defined on $Y^{s-d+\delta(s,p)}_{p}$, hence 
that $\widetilde{A}B_u$ is a well defined composite
\begin{equation}
  X^{s}_{p}\overset{B_u}{\longrightarrow} Y^{s-d+\delta(s,p)}_{p}
  \overset{\widetilde{A}}{\longrightarrow} X^{s+\delta(s,p)}_{p}.
  \label{pos-eq}
\end{equation}
Since $X^{s+\delta}_{p}\hookrightarrow X^{s}_{p}$ by
\eqref{XY-cnd}, the operator $\widetilde{A}B_u$ is of order $0$ 
on $X^{s}_{p}$; hence 
$P^{(N)}:= \sum_{j=0}^{N-1} (\widetilde{A}B_u)^j$
is a linear map $X^{s}_{p}\longrightarrow X^{s}_{p}$. 
This shows the claim on $P^{(N)}$. 

Concerning $(\widetilde{A}B_u)^N$,
there is, by \eqref{Dm-cnd}, a continuous
map $k\colon I\to\Dm_u$, with $I=[a,b]$, such that 
\begin{equation}
  k(a)=(s',p'),\qquad k(b)=(s'',p'').
   \label{ab-eq}
\end{equation}
Clearly $\delta_k:=\inf\bigl\{\,\delta(s,p) \bigm| 
(s,p)\in k(I)\,\bigr\}>0$ by \eqref{del-cnd}, and for $(s,p)\in k(I)$
\begin{equation}
  X^{s}_{p}
  \overset{\widetilde{A}B_u}{\longrightarrow} X^{s+\delta(s,p)}_{p}
  \hookrightarrow X^{s+\delta_k}_{p}.
  \label{pos'-eq}
\end{equation}
With $X_{k(\tau)}:= X^{s}_{p}$ when
$k(\tau)=(s,p)$, let $M:=\sup T$ for
\begin{equation}
  T=\bigl\{\,\tau\in I \bigm| \exists N\in\N\colon 
  (\widetilde{A}B_u)^N(X^{s'}_{p'})
  \subset X_{k(\tau)}\,\bigr\}.
\end{equation}
Then $a\le M\le b$ since 
$\widetilde{A}B_u(X^{s'}_{p'})\subset X^{s'+\delta}_{p'}\subset X_{k(a)}$. 
It now suffices to show that $b\in T$, for 
$(\widetilde{A}B_u)^N (X^{s'}_{p'})\subset X_{k(b)}=X^{s''}_{p''}$ 
for some $N\in\N$ then; and 
$(\widetilde{A}B_u)^{N}$ equals
$(\widetilde{A}B_u)^{N-N'}(\widetilde{A}B_u)^{N'}$ 
for $N>N'$, so the full claim on $\widetilde{A}B_u$ would follow 
because \eqref{pos'-eq} shows that $(\widetilde{A}B_u)^{N-N'}$ is of order 
$0$ on $X^{s''}_{p''}$.

For one thing $M\in T$: by continuity
of $k$ there is a $\tau'<M$ in  $T$ such that, when $|\cdot|$ denotes the
Euclidean norm on $\R^2$ 
and the isometry $(s,p)\leftrightarrow(\fracc np,s)$ is suppressed,  
\begin{equation}
  |k(\tau')-k(M)|<\delta_k/2
  \label{phi-eq}
\end{equation}
and $(\widetilde{A}B_u)^{N-1}(X^{s'}_{p'})\subset X_{k(\tau')}$ for some $N$. 
But by \eqref{pos'-eq} this
entails that $(\widetilde{A}B_u)^{N}(X^{s'}_{p'})$ is a subset of a space with
upper index at least $\delta_k$ higher than that of $X_{k(\tau')}$,
so the embeddings in \eqref{XY-cnd} show that 
$(\widetilde{A}B_u)^{N}(X^{s'}_{p'})$ 
is contained in any space in the intersection of ${\Sdm}$ and a convex
polygon; cf the dashed line in ~Figure~\ref{ball-fig} below. 
It
follows that $(\widetilde{A}B_u)^{N}(X^{s'}_{p'})$ is contained in every 
$X^{s}_{p}$ lying in ${\Sdm}$ and fulfilling
\begin{equation}
    |k(\tau')-(s,p)|<\delta_k/\sqrt{2},
  \label{phi'-eq}
\end{equation}
so in particular $(\widetilde{A}B_u)^{N}(X^{s'}_{p'})\subset X_{k(M)}$ 
is found from \eqref{phi-eq}; whence $M\in T$.

\begin{figure}[htbp]
\setlength{\unitlength}{0.00083333in}
{\renewcommand{\dashlinestretch}{30}
\begin{picture}(5000,3300)(-550,-10)
\thicklines
\path(-500,-10)(-500,3000)
\path(-520,3000)(-500,3100)(-480,3000)
\thinlines
\put(2748,1297){\ellipse{2580}{2580}}
\path(2748,1297)(1836,2209)
\dashline{60.000}(32,405)(2748,3121)(4600,3121)
\put(2748,3121){\makebox(0,0){$\bullet$}}
\put(2748,1297){\makebox(0,0){$\bullet$}}
\put(2748,1297){\makebox(600,100)[r]{$X_{k(\tau')}$}}
\put(1836,841){\makebox(0,0)[lb]{$\bullet$}}
\put(1836,841){\makebox(600,0)[rt]{$X_{k(M)}$}}
\put(2235,1924){\makebox(0,0)[lb]{$\delta_k/\sqrt2$}}
\put(2500,3200){\makebox(0,0)[r]{$(\widetilde{A}B_u)^{N}(X^{s'}_{p'})$ }}
\put(-500,3200){\makebox(0,0)[c]{$s$}}
\end{picture}
}
\caption{The $(\fracnp,s)$-plane with the ball in \eqref{phi'-eq} 
and a polygon of spaces containing
$(\widetilde{A}B_u)^{N}(X^{s'}_{p'})$.}
  \label{ball-fig}
\end{figure}
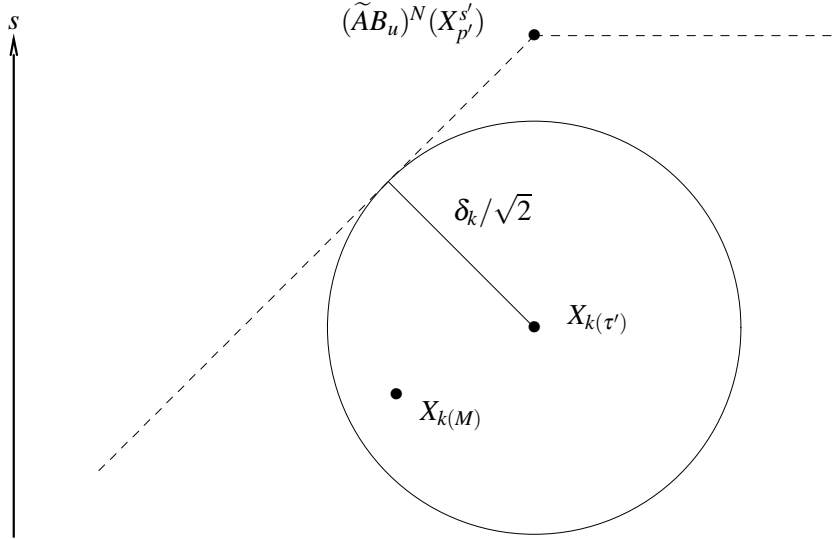

Secondly, $M=b$ follows from $k(I)$'s connectedness:
assuming that $\tau\in \,]M,b]$ exists, the curve $k(\tau)$
would for some $\tau>M$ lie 
in the open $\tfrac{\delta_k}{2}$-ball around $k(M)$. Then 
$|k(\tau)-k(M)|\le\tfrac{\delta_k}{2}<\tfrac{\delta_k}{\sqrt2}$ would
hold, so that the proved fact $M\in T$ would imply (as above) that 
$(\widetilde{A}B_u)^N(X^{s'}_{p'})\subset X_{k(\tau)}$,
contradicting that $\tau\notin T$.

According to  \eqref{AXY-cnd}, \eqref{B-cnd} and the assumptions in the
theorem, the mapping $\widetilde{A}$ has the same meaning on both sides of
\eqref{smln-eq}, regardless of whether one refers to $Y^{s_0-d}_{p_0}$ or to
$Y^{t-d}_r$ (on the left and the right hand sides, respectively).  Therefore
\eqref{R-eq} and the assumption $(s_0,p_0)\in \Dm(\cal N)$ entail 
\begin{equation}
  (I-\cal R)u- \widetilde{A}B_uu= \widetilde{A}f.
  \label{smln'-eq}
\end{equation}
For the given $u$ and $f$, it follows by calculation of the telescopic sum
that 
\begin{equation}
  P^{(N)}(I-\widetilde{A}B_u)u=
  \sum_{j=0}^{N-1}(\widetilde{A}B_u)^j(I-\widetilde{A}B_u)u=
 (I-(\widetilde{A}B_u)^N)u,
  \label{PN-eq}
\end{equation}
and $P^{(N)}$ has the same meaning when applied to both sides of
\eqref{smln'-eq}. Therefore \eqref{smln'-eq}, \eqref{PN-eq} yield
\eqref{par-eq}. 

Note that the term $P^{(N)}(\widetilde{A}f+\cal Ru)$ in \eqref{par-eq} 
is in $X^t_r$ in view of \eqref{R-eq} and 
the proved fact that $P^{(N)}$ has order $0$ on $X^t_r$. By \eqref{err-eq} also
$(\widetilde{A}B_u)^{N}u$ is in $X^t_r$, so this holds for every given
solution $u$ too. 
\end{proof}

Applications of Theorem~\ref{ir-thm} to systems of elliptic boundary
problems are developed in Section~\ref{sys-sect} below for non-linear terms
of product type. In this context
\eqref{delta-lift} in \eqref{del-cnd} is redundant, 
for with $\Dm(A)$ equal to one of the standard domains
$\Dm_k$ it is for any $\eta>0$ clear that $(s+\eta,p)$ belongs to
$\Dm(A)$ when $(s,p)$ does so. But \eqref{delta-lift} is inserted in
preparation for applications to other non-linearities,
like $|u|^a$ with non-integer $a>0$;
and to parabolic problems, cf Remark~\ref{parabolic-rem} below. 

\subsection{A solvability result} 
As an addendum to the Parametrix Theorem, 
it is used for the solvability in eg problem
\pbref{IR} and Theorem~\ref{mlap-thm} below that bilinear
perturbations of linear homeomorphisms always give well-posed
problems locally, ie for sufficiently small data.

It should be folklore how to obtain this from the fixed-point theorem of
contractions. The proof extends to
any quasi-Banach space $X$ for which
$\nrm{\cdot}{}^\lambda$ is subadditive for some $\lambda\in\,]0,1]$, for
$d(x,y)=\nrm{x-y}{}^\lambda$ is a complete metric on $X$ then.
(For $B^{s}_{p,q}$ and $F^{s}_{p,q}$ the existence of $\lambda$ is easy to
see, cf \eqref{fspq-eq} ff.) 
In lack of a reference the next result is given, including details for the
lesser known quasi-Banach space case.

\begin{prop}
  \label{small-prop}
Let $A\colon X\to Y$ be a linear homeomorphism between two quasi-Banach spaces
and $B\colon X\oplus X\to Y$ be a bilinear bounded map. When
$\norm{\cdot}{X}^\lambda$ is subadditive for some $\lambda\in \,]0,1]$ and
$y\in  Y$ fulfills
\begin{equation}
  \norm{A^{-1}y}{X}< {\nrm{A^{-1}B}{}}^{-1}4^{{-1}/{\lambda}},
  \label{y-ball}
\end{equation}
then the ball
$\norm{x}{X}<\nrm{A^{-1}B}{}^{-1}2^{{-1}/{\lambda}}$ contains 
a unique solution of the equation 
\begin{equation}
  Ax+B(x)=y.
\end{equation}
This solution depends continuously on $y$ in the ball \eqref{y-ball}.
\end{prop}
\begin{proof}
When $R:=A^{-1}$, the equation is equivalent to $x=Ry-RB(x)=:F(x)$, where also
$RB=:B'$ is bilinear and  $\nrm{B'}{}\le\nrm{R}{}\nrm{B}{}$. Bilinearity gives
\begin{equation}
  \nrm{F(x)-F(z)}{}^\lambda\le
  \nrm{B'(x,x-z)}{}^\lambda+\nrm{B'(x-z,z)}{}^\lambda\le 
  \nrm{B'}{}^\lambda(\nrm{x}{}^\lambda+\nrm{z}{}^\lambda)\nrm{x-z}{}^\lambda,
  \label{Flambda-eq}
\end{equation}
so $F$ is a contraction on the closed 
ball $K_a=\{\,x\in  X\mid \nrm{x}{}\le a\,\}$ if $a$ fulfills
$2\nrm{B'}{}^\lambda a^\lambda <1$.
By the assumptions $D:=1-4\nrm{Ry}{}^\lambda\nrm{B'}{}^\lambda>0$, so
\begin{equation}
  \nrm{Ry}{}^\lambda+\nrm{B'}{}^\lambda t^{2} <t
\iff
  t\in \,\big]\frac{1-\sqrt{D}}{2\nrm{B'}{}^\lambda}\,,\,
          \frac{1+\sqrt{D}}{2\nrm{B'}{}^\lambda}\big[\, ;
  \label{tint-eq}
\end{equation}
here the interval contains $t=a^\lambda$, when $a^\lambda$ 
is sufficiently close to $(2\nrm{B'}{}^\lambda)^{-1}$.
So, since $x\in K_a$ implies 
$\nrm{F(x)}{}^\lambda\le \nrm{Ry}{}^\lambda+
\nrm{B'}{}^\lambda a^{2\lambda}$, it follows from \eqref{tint-eq} that
$F(x)\in  K_a$, ie $F$ is a map $K_a\to K_a$  for such $a$.
Hence $x=F(x)$ is uniquely solved in $K_a$. If also $Ax'+B(x')=y'$ for some
$x'\in  K_a$,  
\begin{equation}
  \nrm{x-x'}{}^\lambda\le
  \nrm{R(y-y')}{}^\lambda+2a^\lambda\nrm{B'}{}^\lambda\nrm{x-x'}{}^\lambda,
\end{equation}
so $d(x,x')\le cd(Ry,Ry')$ for
$c=(1-2a^\lambda\nrm{B'}{}^\lambda)^{-1}<\infty$. This gives 
the well-posedness in $K_a$, but with the leeway in the choice of $a$ the 
proposition follows.
\end{proof}

\section{Preliminaries on products}
  \label{prod-sect}
A brief review of results on pointwise multiplication
is given before the non-linear operators of product type are introduced in
Sections~\ref{pdty-sect} and \ref{sys-sect} below.

\subsection{Generalised multiplication}
\label{plin-ssect} 

In practice non-linearities often involve multiplication of a
non-smooth function and a distribution in $\cal D'\setminus L_{1}^{\loc}$,
as in $u\partial_1 u$ when $u\in H^{\frac12+\varepsilon}$ 
for small $\varepsilon>0$. 
Although it suffices for a mere construction of solutions to
extend $(u,v)\mapsto u\cdot v$ by continuity to a bounded
bilinear form defined on $H^{s}\times H^{-s}$ for some $s>0$, the proof
of the regularity properties will in general involve extensions to
$F^{s}_{p,q}\times F^{-s}_{p,q}$ for \emph{several} exponents $p$ and $q$. 
This would clearly cause a problem of consistency among the various
extensions, and for $q=\infty$ there would, moreover, not be
density of smooth functions to play on. Commutative
diagrams like \eqref{Acd-eq} would then be demanding to verify for the
product type operators, so a more unified approach to
multiplication is preferred here.  

Since a paper of L.~Schwartz \cite{Swz54} it has been known that
products with a few reasonable properties cannot be everywhere defined 
on $\cal D'\times \cal D'$, and as a 
consequence many notions of multiplication exist, cf the survey
\cite{Ober}. But for the present theory it is important to use a 
product $\pi(\cdot,\cdot)$ that 
works well together with paramultiplication on $\Rn$ and also 
allows a localised version $\pi_{\Omega}$ to be defined on an open set
$\Omega\subset\Rn$. 
A product $\pi$ with these properties was analysed in \cite{JJ94mlt},
cf also \cite{RuSi96},
and for the reader's sake a brief review is given.

\bigskip

The product $\pi$ is defined on $\Rn$ by simultaneous 
regularisation of both factors: for $\psi_k(\xi)=
\psi(2^{-k}\xi)$ with $\psi\in C^\infty_0(\Rn)$
equal to $1$ in a neighbourhood of $\xi=0$, 
\begin{equation}
  \pi(u,v):=\lim_{k\to\infty} (\psi_k(D)u)\cdot(\psi_k(D)v).
  \label{pi-eq}
\end{equation}
Here $u$ and $v\in\cal S'(\Rn)$, and they are required to  
have the properties that this limit should both exist in $\cal D'(\Rn)$ for
\emph{all} $\psi$ of the specified 
type and be independent of the choice of $\psi$. ($\psi_k(D)u:=
\cal F^{-1}(\psi_k \hat u)$ etc.) 

This formal definition is from \cite{JJ94mlt}, but analogous limits have
been known for a long time in the $\cal D'$-context. 
As shown in \cite[Sect.~3.1]{JJ94mlt},
$\pi(u,v)$ coincides with the usual pointwise multiplication:
\begin{gather}
  L_p^{\loc}(\Rn)\times L_q^{\loc}(\Rn)
  \overset{{\displaystyle\cdot}}{\longrightarrow}
  L_r^{\loc}(\Rn),
  \qquad 0\le \fracc1r=\fracp+\fracc1q\le1,
  \label{LLL-eq}
\\
  \cal O_M(\Rn)\times \cal S'(\Rn)
  \overset{\displaystyle\cdot}{\longrightarrow}
  \cal S'(\Rn).
  \label{OSS-eq}
\end{gather} 
For later reference, the main tool for \eqref{LLL-eq} and localisation 
to open sets $\Omega$ is recalled from \cite[Prop.~3.7]{JJ94mlt}: 
if either $u$ or $v$ vanishes in $\Omega$, then any $\psi$ as 
in \eqref{pi-eq} gives
\begin{equation}
 0=\lim_{k\to\infty} \rOm( \psi_k(D)u\cdot \psi_k(D)v)) 
   \quad\text{in $\cal D'(\Omega)$}.
  \label{pi-lim}
\end{equation}
When $\pi(u,v)$ is defined, \eqref{pi-lim} implies 
$\supp\pi(u,v)\subset\supp u\cap\supp v$ 
(this is obvious for \eqref{LLL-eq}--\eqref{OSS-eq}). 
But as a consequence of Lemma~\ref{singsupp-lem}, the
limit in \eqref{pi-lim}  
exists in any case when one of the factors vanish in $\Omega$.

Using \eqref{pi-lim}, 
$\pi_\Omega(u,v)$ is \emph{defined} 
for an arbitrary open set $\Omega\subset\Rn$
on  those $u$, $v$ in $\cal D'(\Omega)$ for 
which $U$, $V\in\cal S'(\Rn)$ exist such that $\rOm U=u$,
$\rOm V=v$ and
\begin{equation}
  \pi_\Omega(u,v):= \lim_{k\to\infty} \rOm
  [(\psi_k(D)U)\cdot  (\psi_k(D)V)]
  \quad\text{exists in $\cal D'(\Omega)$}
  \label{piO-eq}
\end{equation}
independently of $\psi\in C^\infty_0(\Rn)$ with $\psi\equiv1$ near $\xi=0$. 
Hereby $\pi_\Omega$ is well defined, for \eqref{pi-lim} implies that the
limit is independent of the `extension' $(U,V)$, hence the
$\psi$-independence is so  (cf \cite[Def~7.1]{JJ94mlt}). 
But as $\pi(U,V)$ need not be defined, it is clearly 
essential that $\rOm$ is applied before passing to the limit.

\subsection{Boundedness of generalised multiplication}
  \label{pdbdd-ssect}
Using \eqref{piO-eq}, it is easy to see, and well known, 
that $\pi_{\Omega}$ inherits boundedness from $\pi$ on ${\Rn}$ as follows: 

\begin{prop}
  \label{piO-prop}
Let each of the spaces $E_0$, $E_1$ and $E_2$ be either a Besov space
 $B^{s}_{p,q}(\Rn)$ or a Lizorkin--Triebel space $F^{s}_{p,q}(\Rn)$, 
chosen so that $\pi(\cdot,\cdot)$ is a bounded bilinear
operator  
\begin{equation}
  \pi\colon E_0\oplus E_1\to E_2.
  \label{piEEE-eq}
\end{equation}
For the corresponding spaces $E_k(\overline{\Omega}):=\rOm E_k$
over an arbitrary open set $\Omega\subset\Rn$, endowed with the infimum norm,
$\pi_\Omega$ is bounded
\begin{equation}
    \pi_{\Omega}(\cdot,\cdot)\colon E_0(\overline{\Omega})\oplus
    E_1(\overline{\Omega})\to E_2(\overline{\Omega}).
  \label{piOEEE-eq}
\end{equation}
\end{prop} 
In the result above it is a central question under which conditions
\eqref{piEEE-eq} actually holds. This was almost completely analysed
in \cite[Sect.~5]{JJ94mlt} by means of paramultiplication. 
To prepare for the definition and analysis (further below) 
of the exact paralinearisation, this will now be recalled.

First, 
by using \eqref{LP-eq} and setting $\Phi_j\equiv 0\equiv \Psi_j$ for $j<0$,
the paramultiplication operators $\pi_m(\cdot,\cdot)$ 
with $m=1$, $2$, $3$ (in the sense of M.~Yamazaki \cite{Y1,Y2,Y3}), are
defined for those $f$, $g\in\cal S'(\Rn)$ for which the 
series below converge in $\cal D'(\Rn)$:
\begin{subequations}
  \label{paramultiplication-eq}
\begin{align}
  \pi_1(f,g)&=\sum_{j=0}^{\infty} \Psi_{j-2}(D)f\Phi_j(D)g\\
  \pi_2(f,g)&=\sum_{j=0}^{\infty}
  (\Phi_{j-1}(D)f\Phi_j(D)g +\Phi_{j}(D)f\Phi_j(D)g
       +\Phi_{j}(D)f\Phi_{j-1}(D)g)\\
  \pi_3(f,g)&=\sum_{j=0}^{\infty} \Phi_j(D)f\Psi_{j-2}(D)g
\end{align}
\end{subequations}
This applies to \eqref{pi-eq} by taking
$\psi_k=\Psi_k$, for 
$\Psi_k=\Phi_0+\dots+\Phi_k$, so that bilinearity gives that the limit
on the right  
hand side of \eqref{pi-eq} equals $\sum_{m=1,2,3} \pi_{m}(u,v)$,
whenever each $\pi_m(u,v)$ exists\,---\,but this existence is
easily analysed for each $m$ by standard estimates. 
In fact $\pi_1(f,g)$ and $\pi_3(f,g)$ both exist for all $f$, $g\in \cal
S'(\Rn)$, as observed in \cite[Ch.~16]{MeCo97}, so $\pi(u,v)$ is
defined if and only if the second series $\pi_2(u,v)$ is so.
Finally the $\psi$-independence is established post festum;
cf \cite[Sect.~6.4]{JJ94mlt}.

For convenience $E^s_{p,q}$ will now denote a space which
(for every value of $(s,p,q)$) 
may be either a Besov or a Lizorkin--Triebel space on $\Rn$.
It was proved in \cite[Thm.~4.2]{JJ94mlt}, albeit with \eqref{Dpp-cnd} and
\eqref{Dpp'-cnd} essentially covered by \cite{F3}, that if 
\begin{equation}
  \norm{fg}{E^{s_2}_{p_2,q_2}}\le c
  \norm{f}{E^{s_0}_{p_0,q_0}}
  \norm{g}{E^{s_1}_{p_1,q_1}}  
\end{equation}
holds for all Schwartz functions $f$ and $g$, then
\begin{subequations}
  \label{nec1-cnd}
\begin{align}
 s_0+s_1&\ge n(\fracc1{p_0}+\fracc1{p_1}-1),  
  \label{Dp-cnd}\\
 s_0+s_1&\ge0.
  \label{Dpp-cnd} 
\end{align}
\end{subequations}
As a supplement to this, the following were also
established there:
\begin{subequations} 
  \label{nec2-cnd}
\begin{align}
  s_0+s_1&=\fracc{n}{p_0}+\fracc{n}{p_1}-n  
  &\text{implies }&
  \begin{cases}
  \fracc{1}{q_0}+\fracc{1}{q_1}\ge1 \ \text{in $BB\bigdot$-cases},\\
  \fracc{1}{q_0}+\fracc{1}{p_1}\ge1 \ \text{in $BF\bigdot$-cases};
  \end{cases}
  \label{Dp'-cnd}\\
  s_0+s_1&=0
  &\text{implies } & \fracc1{q_0}+\fracc1{q_1}\ge1.
  \label{Dpp'-cnd}
\end{align}
\end{subequations}
The main interest lies in the $BB\bigdot$- and $FF\bigdot$-cases and the
case with $\max(s_0,s_1)>0$ (for $s_0=s_1=0$ H{\"o}lder's
inequality applies). In this situation the sufficiency of the above
conditions was entirely confirmed by means of \eqref{paramultiplication-eq}, cf
the following version of \cite[Cor~6.12]{JJ94mlt} for isotropic spaces:

\begin{thm}  
  \label{dom-thm}
When $\max(s_0,s_1)>0$, then it holds in the 
$BB\bigdot$- and $FF\bigdot$-cases that
$E^{s_0}_{p_0,q_0}$ and $E^{s_1}_{p_1,q_1}$ on $\Rn$ are `multiplicable' 
if and only if both
\eqref{Dp-cnd}--\eqref{Dpp-cnd} and \eqref{Dp'-cnd}--\eqref{Dpp'-cnd}  
hold.
\end{thm}

The spaces that receive $\pi(E^{s_0}_{p_0,q_0},E^{s_1}_{p_1,q_1})$ were
almost characterised in \cite{JJ94mlt}, departing from at least 8 other
necessary conditions, but the below
Theorem~\ref{mod-thm} will imply what is needed in this direction.

\begin{rem}
  \label{pmest-rem}
To prepare for Theorem~\ref{mod-thm} below, a few estimates of the $\pi_j$
are recalled. When $\fracc1{p_2}=\fracc1{p_0}+\fracc1{p_1}$, 
$\fracc1{q_2}=\fracc1{q_0}+\fracc1{q_1}$, there is boundedness 
\begin{align}
  \pi_1&\colon L_\infty\oplus B^{s}_{p,q}\to B^{s}_{p,q}
  \label{pi1LB-eq} \\
  \pi_1&\colon B^{s_0}_{p_0,q_0}\oplus B^{s_1}_{p_1,q_1}
    \to B^{s_0+s_1}_{p_2,q_2}
   \quad\text{for}\quad s_0<0,
  \label{pi1BBB-eq}  \\
  \pi_2&\colon B^{s_0}_{p_0,q_0}\oplus B^{s_1}_{p_1,q_1}
    \to B^{s_0+s_1}_{p_2,q_2}
   \quad\text{for}\quad s_0+s_1>(\fracc n{p_2}-n)_+.
  \label{pi2BBB-eq}
\end{align}
Since $\pi_3(f,g)=\pi_1(g,f)$, also $\pi_3$ is covered by this.
Analogous results hold for the Lizorkin--Triebel spaces, except  
that Lemma~\ref{F-lem} 
for $s_0+s_1>(\fracc n{p_2}-n)_+$ entails
\begin{equation}
    \pi_2\colon F^{s_0}_{p_0,q_0}\oplus  F^{s_1}_{p_1,q_1}
  \to F^{s_0+s_1}_{p_2,t}
\quad\text{when}\quad
\begin{cases}
  t\ge q_2 & \text{for $q_2\ge p_2$}
\\
 t>\tfrac{n}{n+s_0+s_1} & \text{for $q_2<p_2$.}
  \label{pi2FFF-eq}
\end{cases}
\end{equation}
These estimates all follow from the dyadic corona and ball criteria in a way
that is standard by now, so details are omitted
(the arguments can be found in a refined version for a special case in
Proposition~\ref{chi-prop} below, cf also the proof of
Lemma~\ref{tensor-lem}). They go back to the 
paradifferential estimates of M.~Yamazaki \cite{Y1}, but in the simpler context
of paramultiplication an account of the estimates may be found in
eg \cite[Thm~5.1]{JJ94mlt},
though with \eqref{pi2FFF-eq} as a small improvement.  
\end{rem}

\begin{rem}
  \label{Wm-rem}
It is used in Section~\ref{finrem-sect} below that multiplication cannot
define a continuous map $W^m_1\oplus W^m_1\to\cal D'$ when $2m<n$. 
When the range is a
Besov space this follows on $\Rn$ from \eqref{Dp-cnd}, 
but for the general statement
an explicit proof should be in order. If $\rho\in C^\infty_0$ is real and
$\rho_k(x)=\tfrac{1}{k}2^{k(n-m)}\rho(2^k x)$, it is easy to see that
$\norm{\rho_k}{W^m_1}=\cal O(\tfrac{1}{k})\searrow 0$. But for $\varphi\in
C^\infty_0$ non-negative with $\varphi(0)=1$, $2m<n$ implies
\begin{equation}
  \dual{\rho_k^2}{\varphi}=k^{-2}2^{k(n-2m)}\int
\rho^2(y)\varphi(2^{-k}y)\,dy
\to \infty.
\end{equation}
This argument works for open sets
$\Omega\ni0$ and extends to all
$\Omega\subset\Rn$ by translation.
\end{rem}

\subsection{Extension by zero}
  \label{ext-ssect}
Having presented the product $\pi(\cdot,\cdot)$ formally,  
the opportunity is taken to make a digression needed later.

In Section~\ref{vK-sect}--\ref{sys-sect} the operators $A$ and 
$\widetilde{A}$ of
Section~\ref{main-sect} will be realised through 
the Boutet de~Monvel calculus of linear boundary problems, so it will
be all-important to have commuting diagrams like \eqref{Acd-eq} 
for the operators in the calculus. 
Avoiding too many details, the main step is to show that
truncated pseudo-differential operators are defined independently of the
spaces. As the question is local, it is enough to treat them on the
half-space $\Rn_+=\{\,x_n>0\,\}$, where they are of the form
$P_+=r^+Pe^+$ for a ps.d.o.\ $P$ defined on 
$\cal S'(\Rn)$, 
so it suffices to define $e^+$ on all spaces with $s$ close
to $0$ ($e^+:=e_{\Rn_+}$, $r^+:=r_{\Rn_+}$). 
However,  setting
$e^+u=\pi(\chi,v)$ when $r^+v=u$ and $\chi$ denotes the characteristic
function of $\Rn_+$, it follows from \eqref{pi-lim} that $\pi(\chi,v)$ 
at most depends on $v$ in the null set $\{\,x_n=0\,\}$. 
But since the spaces in the next
result only contain trivial distributions supported in this hyperplane,
this suffices for a space-independent definition
of $e^+u$ when $u$ belongs to these spaces.

\begin{prop}   \label{chi-prop}
The characteristic function $\chi$ of $\Rn_+$ yields a bounded map
\begin{equation}
  \pi(\chi,\cdot)\colon E^s_{p,q}(\Rn)\to E^s_{p,q}(\Rn),
  \label{chi-eq}
\end{equation}
for Besov and Lizorkin--Triebel spaces with 
$\fracp-1+(n-1)(\fracp-1)_+<s<\fracp$.
\end{prop}

The $F^{s}_{p,q}$-part of this will be based on a
similar result of J.~Franke \cite[Cor.~3.4.6]{F2}.
In principle Franke analysed another product
as he estimated $\chi v$ for $\supp\hat v$ compact 
and extended by continuity to $F^{s}_{p,q}$ (for $q=\infty$ using
Fatou's lemma).
But the full treatment of $P_+$ in
$B^{s}_{p,q}$ and $F^{s}_{p,q}$-spaces is also based on the 
splitting of $\pi$ in \eqref{paramultiplication-eq}, so it is  important
that Franke's product equals $\pi(\chi,v)$. This 
was exploited  in \cite{JJ96ell}, albeit without details,
so it is natural to take the opportunity to return to this point during the

\begin{proof}[Proof of Proposition~\ref{chi-prop}]
In view of \eqref{paramultiplication-eq} it suffices for $B^{s}_{p,q}$ 
to show bounds
\begin{equation}
  \norm{\pi_m(\chi,u)}{B^s_{p,q}}\le C\norm{u}{B^s_{p,q}}
  \quad\text{for $m=1$, $2$, $3$}.
  \label{pim-eq}
\end{equation}
Using Remark~\ref{pmest-rem},
this holds for $m=1$ for every $s$ because $\chi\in
L_\infty$. And $L_\infty\subset B^0_{\infty,\infty}$, 
so for $m=2$ it holds for $s>(\fracnp-n)_+$, while for $m=3$
it does so for $s<0$. The last two restrictions on $s$ will be
relaxed using the anisotropic structure of $\chi$.

For brevity $u_k:=\Phi_k(D)u$, $u^k:=\Psi_k(D)u$ etc.
Now $\pi_{3}(\chi,u)=\sum_{k\ge2}\chi_ku^{k-2}$.
If $H$ is the Heaviside function, $\chi(x)=1(x')\otimes H(x_n)$ and
\begin{equation}
  \chi_k=
   c\cal F^{-1}(\Phi_k(\xi)\delta_0(\xi')\otimes\hat H(\xi_n))
  =1(x')\otimes
  \cal F^{-1}_{\xi_n\to x_n}(\Phi_k(0,\xi_n)\hat H).
  \label{xk-eq}
\end{equation}
For the second factor, note that $2^{k}\hat H(2^{k}\xi_n)=\hat H(\xi_n)$
since $H$ is homogeneous of degree zero, so
\begin{equation}
  H_k(x_n)=\cal F^{-1}(\Phi_1(0,2^{-k}\cdot)\hat H)(x_n)=
    2^{k}\cal F^{-1}(\Phi_1(0,\cdot)\hat H(2^{k}\cdot))(2^{k}x_n)
  =H_1(2^{k}x_n).
\end{equation}
Here $H_k$ refers to the decomposition
$1=\sum\Phi_j(0,\xi_n)$ on $\R$. For $k\ge1$ this gives 
\begin{equation}
  \norm{H_k}{L_p(\R)}
  =2^{-(k-1)/p}\norm{H_1}{L_p(\R)}<\infty.
  \label{Hk-eq}
\end{equation}
Indeed, $\Phi_1(0,\cdot)\hat H\in \cal S(\R)$
because $\cal FH=\tfrac{-\im}{\tau}\cal F(\partial_t H(t))=\tfrac{1}{\im\tau}$
for $\tau\ne0$; hence $H_1\in L_p$.
Note that $\tilde H:=H-H_0$, 
by \eqref{Hk-eq} and Lemma~\ref{DCC-lem},
is in $B^{1/p}_{p,\infty}(\R)$ for
$0<p\le\infty$. 

To handle the factor $1(x')$ in \eqref{xk-eq}, there is a mixed-norm estimate
\begin{equation}
  \norm{\chi_k u^{k-2}}{L_p}^p\le
  \int (\sup_{t\in \R}|u^{k-2}(x',t)|)^p\,dx' 
  \norm{H_k}{L_p(\R)}^p
\end{equation}
so that $s-\fracp<0$ in view of the summation lemma 
\eqref{slem-eq} yields
\begin{equation}
  \begin{split}
  \sum_{k>1} 2^{skq}\nrm{\chi_k u^{k-2}}{p}^q
  &\le c\sum_{k>1} 2^{(s-\fracp)kq}(\sum_{0\le l\le k}
     \norm{u_l}{L_p(L_\infty)}^{\min(1,p)})^{\frac{q}{\min(1,p)}}\nrm{H_1}{p}^q 
\\&\le c\nrm{H_1}{p}^q \sum_{k\ge0} 2^{(s-\fracpi)kq}
   \norm{u_k}{L_p(L_\infty)}^q
\\&\le c\norm{\tilde H}{B^{\fracpi}_{p,\infty}}^q \norm{u}{B^{s}_{p,q}}^q.
  \end{split}
\end{equation}
Indeed, the last step follows from the Nikol$'$ski\u\i--Plancherel--Polya
inequality, cf Lem.~\ref{vNPP-lem} ff,
when this is used in the $x_n$-variable (for fixed $x'$ the
Paley--Wiener--Schwartz Theorem gives that $u(x',\cdot)$ has its spectrum in
the region $|\xi_n|\le 2^{k+1}$). By the dyadic corona criterion, cf
Lemma~\ref{DCC-lem}, this proves $\pi_3(\chi,u)\in  B^s_{p,q}$, hence
the case $m=3$ for $s<\fracp$. 

For $m=2$ only $\fracp-1<s\le0$ remains; this implies $1<p\le\infty$.
It can be assumed that $u_0=0$, for $u$ may be replaced by $u-u_0-u_1$
because $\chi\in L_\infty$ implies that $\pi_2(\chi,u_0+u_1)$ belongs to 
$\bigcap_{t>0} B^t_{p,q}$ by Lemma~\ref{DCC-lem}.
Then $\pi_2(\chi,u)$ is split in three contributions, with details given for
$\sum\chi_ku_k$ (terms with $\chi_k u_{k-1}$ and $\chi_{k-1}u_k$ are treated
analogously). In the following it is convenient to replace 
$(u_j)$ temporarily by $(0,\dots,0,u_N,\dots,u_{N+M},0,\dots)$, in
which the entries are also called $u_j$ for simplicity. In this way the
below series trivially converge.

Note that the 
Nikol$'$ski\u\i--Plancherel--Polya inequality used in $x_n$ yields
\begin{equation}
  \nrm{\Phi_j(D)\sum_{k\ge j-1} \chi_k u_k}{p}\le c
  \sum_{k\ge j-1}\Norm{\check\Phi_j*(\chi_ku_k)}{L_{p,x'}(L_{1,x_n})}
  2^{j(1-\fracpi)}.
  \label{chi2-eq}
\end{equation}
In this mixed-norm expression, Fubini's theorem gives for $k\ge1$
\begin{equation}
    \int |\check\Phi_j*(\chi_ku_k)(x',x_n)|\,dx_n\le \nrm{H_k}{1}
  \iint|\check\Phi_j(x'-y',x_n)|\,dx_n \sup_{t\in\R}|u_k(y',t)|
  \,dy' .
\end{equation}
Reading this as a convolution on $\R^{n-1}$, the usual $L_p$-estimate leads to 
\begin{equation}
  \norm{\check\Phi_j*(\chi_ku_k)}{L_p(L_1)}\le
  \nrm{H_k}{1}\nrm{\check\Phi_j}{1}\norm{u_k}{L_p(L_\infty)}.
\end{equation}
In view of \eqref{chi2-eq} and \eqref{slem-eq} ff this gives, 
since $s+1-\fracp>0$ and $\supp\cal F(\chi_k u_k)$ is disjoint from 
$\supp\Phi_j$ unless $k>j-2$ (and since $u_0=0$),
\begin{equation}
  \begin{split}
  \sum_{j\ge0}2^{sjq}\nrm{\check\Phi_j*\sum_{k\ge0}\chi_ku_k}{p}^q
  &\le c
  \sum_{j\ge0}2^{(s+1-\fracpi)jq}(\sum_{k\ge j-1}\nrm{H_k}{1}
     \norm{u_k}{L_p(L_\infty)})^q
\\
  &\le c' \sum_{j\ge0} 2^{(s+1-\fracpi)jq}\nrm{H_j}{1}^q
     \norm{u_j}{L_p(L_\infty)}^q    
\\
  &\le c' \norm{\tilde H}{B^1_{1,\infty}(\R)}^q
          \sum_{j\ge0} 2^{sjq}\nrm{u_j}{p}^q
   <\infty.
  \end{split}
\end{equation}
For $q<\infty$ the right hand side tends to $0$ for $N\to\infty$,
so the $\pi_2$-series is fundamental in $B^s_{p,q}$. There
is also convergence for $q=\infty$, since $u\in B^{s-\varepsilon}_{p,1}$ for
some $\varepsilon>0$ such that $s-\varepsilon+1-\fracp>0$. 
The above estimate then also 
applies to the original $(u_j)$, which yields \eqref{pim-eq} for $m=2$. 

To cover the $F^{s}_{p,q}$-case, note the continuity 
$B^{s+\varepsilon}_{p,1}\xrightarrow{\pi(\chi,\cdot)} B^{s+\varepsilon}_{p,1}
\hookrightarrow F^{s}_{p,q}$ for $p<\infty$
and sufficiently small $\varepsilon>0$.
If Franke's multiplication by $\chi$ is denoted $M_\chi$, it follows 
that $B^{s+\varepsilon}_{p,1} \xrightarrow{M_{\chi}} F^{s}_{p,q}$ is
continuous. Since $\cal F^{-1}C^\infty_0$ is dense in
$B^{s+\varepsilon}_{p,1}$ and $M_\chi$ extends the pointwise product
on $\cal F^{-1}C^\infty_0$ by
$\chi$, it follows that $M_\chi$ coincides with $\pi(\chi,\cdot)$ for all
Besov spaces with $(s,p,q)$ as in the theorem, if $p<\infty$. But then they
coincide on all the $F^{s}_{p,q}$ spaces, so $\pi(\chi,\cdot)$ is bounded on
$F^{s}_{p,q}$ as claimed.
\end{proof}

\begin{rem}
The above direct treatment of the Besov spaces should be
of some interest in itself, in view of the mixed-norm estimates 
that allow a concise proof of all cases.
\end{rem}

\section{Product type operators}
  \label{pdty-sect}

A basic class of non-linear operators and their paralinearisations
can now be formally introduced:

\begin{defn}  \label{pdtO-defn}
Operators of \emph{product type} $(d_0,d_1,d_2)$ 
on an open set $\Omega\subset\Rn$ are maps (or finite sums of maps) 
of the form
\begin{equation}
  (v,w)\mapsto P_2\pi_{\Omega}(P_0 v, P_1 w),  
  \label{pdtO-eq}
\end{equation}
for linear partial differential operators $P_j$ of order $d_j$, $j=0,1,2$,
with constant coefficients.
The quadratic map $u\mapsto P_2\pi_\Omega(P_0u,P_1u)$ is also said
to be of product type.
\end{defn}

Although \eqref{pdtO-eq} often just amounts to $P_2(P_0v\cdot P_1w)$,
it is in general essential to use $\pi_\Omega$ from \eqref{piO-eq}
in this definition,
because the product cannot always be reduced to one of the
forms in \eqref{LLL-eq}--\eqref{OSS-eq}.
In Section~\ref{sys-sect} below the notion of product type operators will
be extended to certain maps between vector bundles.

The case with $P_2=I$ is throughout referred to as an operator 
of type $(d_0,d_1)$.
Generally $d_0$, $d_1$, $d_2$ appear in the same order as the $P_j$ are
applied. 

If for simplicity $P_2=I$ is considered, the operator 
$\pi_\Omega(P_0u,P_1u)$ may of course be viewed as a homogeneous second order
polynomial $p(z_1,\dots,z_N)$ composed with a jet $J_k u=(D^\alpha
u)_{|\alpha|\le k}$, $k=\max(d_0,d_1)$.
But in general this jet description is too rigid, 
for a given operator of product type with $P_2=I$
may be the restriction of one  with $P_2\ne I$, cf Example~\ref{MA-ex}. 
And conversely $P_2\pi_\Omega(P_0u, P_1 u)$ may 
extend another one of the type in \eqref{pdtO-eq}. 

These differences lie not only in the various expressions such operators can
be shown to have, but also in the parameter domains
they \emph{may} be given. Consider eg
\begin{equation}
  u\mapsto u\cdot\partial_{1}u,
  \qquad
  u\mapsto \frac{1}{2}\partial_{1}(u^2).
  \label{uu-eq}
\end{equation}
The latter coincides with the former at least for $u\in C^\infty$.
By H{\"o}lder's inequality, $\partial_1(u^2)$ is
a bounded bilinear map $L_4(\Rn)\to H^{-1}(\Rn)$,
so its natural parameter domain contains $(s,p)=(0,4)$. But it is not easy
to make sense of $u\partial_1 u$, as a map $L_4\to H^{-1}$;
even with the product $\pi$ it is problematic, for by \eqref{Dpp-cnd} this is
not well defined on $L_4\oplus H^{-1}_4$. 
Hence it seems best (in analogy with minimal and maximal differential
operators in $L_2(\Omega)$) to treat the expressions in
\eqref{uu-eq} as two different operators, with different parameter domains.

More general classifications of non-linear operators are available in the
literature; the reader may consult eg
\cite[Sect.~5]{Bon} and \cite[\S~2]{Y3}. But as discussed in the introduction,
the product type operators defined above are adequate for fixing ideas and for
important applications.

\begin{exmp}
  \label{MA-ex}
For a useful commutation of differentiations to
the left of the pointwise product, consider as in
Section~\ref{vK-sect} below the `von Karman bracket':
\begin{equation}
    [v,w] :=D^2_{1} v D^2_{2} w +D^2_{2} v D^2_{1}w
           -2D^2_{12} v D^2_{12} w.
  \label{vKbr-eq}
\end{equation}
Introducing the expression
\begin{equation}
  B(v,w)= D^2_{12}(D_1vD_2w+D_2vD_1w)-D^2_1(D_2vD_2w)
        -D^2_2(D_1vD_1w),
  \label{vKB-eq}
\end{equation}
then $B(v,w)=[v,w]$ whenever $v$ and $w$
are regular enough to justify application of Leibniz' rule. Clearly
$B(\cdot,\cdot)$ is a case with $P_2\neq I$.
\end{exmp}

\begin{defn}
  \label{plin-defn}
For each choice of
$\Psi_k$ in \eqref{LP-eq}, the \emph{exact paralinearisation} $L_u$ 
of $Q(u)=P_2\pi(P_0u,P_1u)$ on $\Rn$
is defined as follows, 
\begin{equation}
  L_u g= -P_2\pi_1(P_0 u,P_1 g)-P_2\pi_2(P_0 u,P_1 g)
         -P_2\pi_3(P_0 g,P_1 u).
  \label{plin-eq}  
\end{equation}
For $\Omega\subset\Rn$ and a universal extension operator $\lOm$,
cf \eqref{ext-eq}, 
the composite $g\mapsto \rOm L_U(\lOm g)$ with $U=\lOm u$ is
also referred to as the exact paralinearisation; it is written $L_u$ for
brevity.  
\end{defn}
The rationale is that $L_ug$ has circa the
same regularity as $g$ (contrary to the case of linearisations that are not
moderate). Cf Theorem~\ref{mod-thm} below. 

Conceptually, Definition~\ref{plin-defn} invokes an interchange of the maps
$\lOm$ and $P_j$, compared to \eqref{pdtO-eq}, 
where $P_0$ and $P_1$ are applied before
the implicit extensions to $\Rn$ in $\pi_\Omega$; cf \eqref{piO-eq}.
The advantage is that $L_ug$ then has the structure of a
composite map $r_\Omega\circ P_u\circ \lOm (g)$ for a certain
pseudo-differential operator $P_u$ of type $1,1$; 
cf Theorem~\ref{symb-thm} below. 

However, as justification
$P_j\lOm v=\lOm P_jv$ in $\Omega$, whence the localisation
property in \eqref{pi-lim} implies that $-L_uu$ gives back the original
product type operator:

\begin{lem}
  \label{Lu-lem}
Let $u$ belong to a Besov or Lizorkin--Triebel space
$E^{s}_{p,q}(\overline{\Omega})$ such that the parameters 
$(s-d_j,p,q)_{j=0,1}$ fulfil \eqref{nec1-cnd}--\eqref{nec2-cnd} 
and $s>\max(d_0,d_1)$. Then
\begin{equation}
  P_2\pi_{\Omega}(P_0 u, P_1 u)=-L_u(u).
  \label{piLuu-id}
\end{equation}
This holds for any choice of $\lOm$ and $\Psi_k$ (or $\Phi_k$) 
in the definition of $L_u$.
\end{lem} 
\begin{proof}
Let $P_2=I$ for simplicity.
Theorem~\ref{dom-thm} gives that the parameters $(s-d_j,p,q)_{j=0,1}$ 
belong to the parameter domain of $\pi$ on $\Rn$, so it holds for all
$v$, $w\in  E^s_{p,q}(\overline{\Omega})$ that
\begin{equation}
  \pi_{\Omega}(P_0 v, P_1 w)= \rOm \pi(P_0\lOm v, P_1\lOm w)=
  \rOm\lim_{k\to\infty} (\psi_k(D) P_0 \lOm v) \cdot 
                        (\psi_k(D) P_1 \lOm w)).
  \label{piOLL-eq}
\end{equation}
Indeed, $\pi(P_0\lOm v,P_1\lOm w)$ is defined, and
$\rOm$ commutes with the limit by its $\cal D'$-continuity,
whilst the $P_j\lOm v$ restrict to $P_j v$,
so the product $\pi_\Omega(P_0,P_1 w)$ exists and
\eqref{piOLL-eq} holds.

The choice of $\lOm$ is inconsequential for 
$\rOm \pi(P_0\lOm v, P_1\lOm w)$,
since the left hand side of \eqref{piOLL-eq} does not
depend on this; similarly one can take $\psi_k=\Psi_k$ 
(the formal definition of $\pi_\Omega$ in \eqref{piO-eq} is essential
here). Now \eqref{piLuu-id}
follows upon insertion of $v=w=u$, for by
\eqref{paramultiplication-eq} ff and the formula $\Psi_k=\Phi_0+\dots+\Phi_k$,
the right hand side of \eqref{piOLL-eq} then equals the formula for
$-L_u(u)$ in \eqref{plin-eq}, since the $\pi_j$-series converge by the
assumption on $(s,p,q)$ and the remarks following \eqref{paramultiplication-eq}.  
\end{proof}

The above introduction of paralinearisation is not the only
possible, but the intention here is to make the relation to the `pointwise'
product on $\Omega$ clear. 

\subsection{Estimates of product type operators.}
  \label{est-ssect}

For a general product type operator 
$B(\cdot,\cdot):=\pi(P_0\cdot,P_1\cdot)$
a large collection of boundedness properties now follows from the theory
reviewed in Section~\ref{plin-ssect}--\ref{pdbdd-ssect}. 
Indeed, using Theorem~\ref{dom-thm} it is clear  
that $\pi(P_0\cdot,P_1\cdot)$ 
is bounded from $E^{s_0}_{p_0,q_0}\oplus E^{s_1}_{p_1,q_1}$
to some Besov or Lizorkin--Triebel space provided
\begin{equation}
  s_0+s_1> d_0+d_1+(\fracc n{p_0}+\fracc n{p_1}-n)_+.
  \label{stdom-eq}
\end{equation}
The \emph{standard} domain $\Dm(B)$ of the bilinear operator $B$
is the set of (pairs of triples of) parameters 
$(s_j,p_j,q_j)_{j=0,1}$ that satisfy this inequality. 
Since 
it works equally well for the $BBB$- and 
$FFF$-cases, the notation is the same in the two cases.

For the map $Q(u):=B(u,u)$ the parameter domain $\Dm(Q)$ 
derived from \eqref{stdom-eq} is 
termed the \emph{quadratic} standard domain of $Q$ (or of $B$).
For this domain one has the next result on the \emph{direct} regularity
properties of product type non-linearities.

\begin{prop} 
  \label{nlnr2-prop}
Let $B(v,w)$ be an operator of product type $(d_0,d_1,d_2)$ with $d_0\le d_1$. 
The quadratic standard domain $\Dm(Q)$ consists of the $(s,p,q)$ fulfilling
  \begin{equation}
    s>\tfrac{d_0+d_1}{2}+(\fracc np-\tfrac{n}{2})_+,
    \label{stQdom-eq}
  \end{equation}
and for each such $(s,p,q)$ the non-linear operator $Q$ is  bounded 
\begin{equation}
    Q\colon B^s_{p,q}\to
    B^{s-\sigma(s,p,q)}_{p,q}
    \label{n22}
\end{equation}
  when $\sigma(s,p,q)$, for some $\varepsilon>0$, is taken equal to
  \begin{equation}
    \sigma(s,p,q)=d_2+d_1+(\fracc np+d_0-s)_++\varepsilon
     \kd{\fracc np+d_0=s} \kd{q>1}.
   \label{sgm-eq}
  \end{equation}
Similar results hold for $F^{s}_{p,q}$ provided $\kd{q>1}$ is replaced by
$\kd{p>1}$. 
\end{prop}

Analogous results for open sets $\Omega\subset\Rn$ can be derived from
Proposition~\ref{piO-prop}. Details on this are left out for simplicity,
and so is the proof, for it follows from the below Theorem~\ref{mod-thm}
by application of $L_u$ to $u$, cf Lemma~\ref{Lu-lem}
(note that \eqref{stQdom-eq} implies $(s-d_1)+(s-d_0)>0$, thence
$s>\max(d_0,d_1)$).

\begin{rem}
  \label{dom-rem}
By \eqref{stQdom-eq} the quadratic domain $\Dm(Q)$ only depends on the
orders via the mean $(d_0+d_1)/2$.
The correction $\fracc np-\tfrac n2$ occurring for $p<2$ is
independent of $d_0$ and $d_1$; cf Figure~\ref{stdom-fig}. 
\end{rem}

\begin{figure}[htbp]
\setlength{\unitlength}{0.008mm}
\begingroup\makeatletter\ifx\SetFigFont\undefined%
\gdef\SetFigFont#1#2#3#4#5{%
  \reset@font\fontsize{#1}{#2pt}%
  \fontfamily{#3}\fontseries{#4}\fontshape{#5}%
  \selectfont}%
\fi\endgroup%
{\renewcommand{\dashlinestretch}{30}
\begin{picture}(8424,7000)(0,-10)
\thicklines
\path(672.000,6372.000)(612.000,6612.000)(552.000,6372.000)
\path(612,6612)(612,612)(8412,612)
\path(8172.000,552.000)(8412.000,612.000)(8172.000,672.000)
\path(12,612)(612,612)(612,12)
\thinlines
\put(8650,570){\makebox(0,0){$\fracnp$}}
\put(612,6800){\makebox(0,0){$s$}}
\put(510,1850){\makebox(0,0)[rc]{${\scriptstyle\tfrac{d_0+d_1}{2}}$}}
\path(550,1812)(674,1812)
\put(4212,480){\makebox(0,0)[ct]{${\scriptstyle\tfrac{n}{2}}$}}
\path(4212,550)(4212,674)
\put(7000,4200){\makebox(0,0)[lc]{$s=\tfrac{d_0+d_1}{2}
+(\fracnp-\tfrac{n}{2})_+$}}
\texture{0 0 0 888888 88000000 0 0 80808 
	8000000 0 0 888888 88000000 0 0 80808 
	8000000 0 0 888888 88000000 0 0 80808 
	8000000 0 0 888888 88000000 0 0 80808 }
\shade\path(612,6012)(612,1812)(4212,1812)(8412,6012)
\dottedline{100}(8412,6012)(8812,6412)
\put(4200,5000){\makebox(0,0){$\Dm(Q)$}}
\end{picture}
}
  \caption{The quadratic standard domain $\Dm(Q)$}
  \label{stdom-fig}
\end{figure}
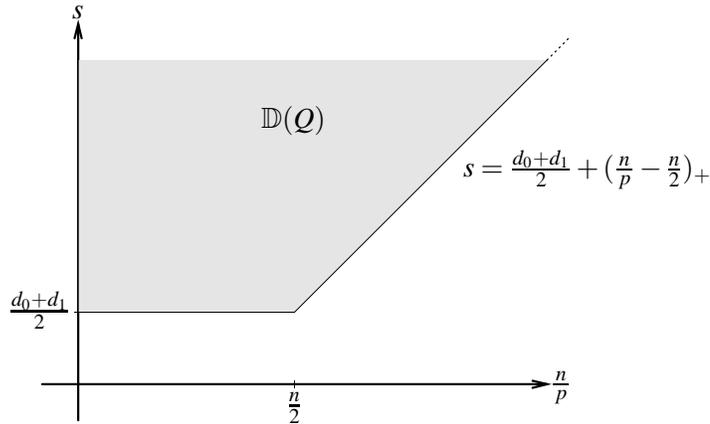

\subsection{Moderate linearisations of product type operators}
  \label{mod-ssect}

The properties of exact paralinearisations of product type operators will
now be derived. This will in two ways give better results
than the usual linearisation theory
in, say \cite{Bon} and \cite[Thm.~16.3]{MeCo97}: 
first of all, the $\pi_2$-terms are incorporated into $L_u$,
which is useful since
they need not be regularising in the context here. 
Secondly, the family $L_u$ is obtained
for $u$ running through the (large) set $\bigcup B^{s}_{p,q}$, 
and it is only in the quadratic standard domain, where $u$ is regular enough
to make $-L_uu=Q(u)$ a meaningful formula, that
$L_u$ is a linearisation of $Q$.
 
\begin{thm}[The Exact Paralinearisation Theorem]
  \label{mod-thm}
Let $B$ be of product type
$(d_0,d_1,d_2)$ with $d_0\le d_1$ as in Definition~\ref{pdtO-defn}; 
and let $\lOm$ be a universal extension from $\Omega$ to $\Rn$.

When $u\in B^{s_0}_{p_0,q_0}(\overline{\Omega})$ for some arbitrary
$(s_0,p_0,q_0)$,  then the exact paralinearisation in 
Definition~\ref{plin-defn} yields a linear operator $L_u$
with parameter domain $\Dm(L_u)$ given by 
\begin{equation}
  s>d_0+d_1-s_0+(\fracnp+\fracc n{p_0}-n)_+.
  \label{DmLu-eq}
\end{equation}
For $\varepsilon>0$ the operator $L_u$ is of order $\omega$ as follows,
\begin{gather}
 L_u\colon B^s_{p,q}(\overline{\Omega})\to
  B^{s-\omega}_{p,q}(\overline{\Omega}) 
  \quad\text{for $(s,p,q)\in \Dm(L_u)$},
   \\
  \omega = d_2+ d_1+(\fracc n{p_0}-s_0+d_0)_+
  +\varepsilon\kd{\fracc n{p_0}-s_0+d_0=0}\kd{q_0>1}, 
  \label{om-eq}
\end{gather}
In particular, when $Q(u):=B(u,u)$ and $(s_0,p_0,q_0)\in \Dm(Q)$, 
cf \eqref{stQdom-eq}, then $L_u$ is a
\emph{moderate linearisation} of $Q$.
Corresponding results hold for Lizorkin--Triebel spaces 
when $u\in F^{s_0}_{p_0,q_0}(\overline{\Omega})$, provided the factor 
$\kd{q_0>1}$ in \eqref{om-eq} is replaced by $\kd{p_0>1}$.
\end{thm}

\begin{rem}
  \label{LuQ-rem}
Clearly $\omega$ is independent of $(s,p,q)$; because it formally
equals $\sigma(s_0,p_0,q_0)$, it can be said that, for a product type
operator, the paralinearisation $L_u$ inherits the order of the non-linear
operator $Q(u)$ on the space $E^{s_0}_{p_0,q_0}\ni u$.
\end{rem}

\begin{proof}
Since the nature of the proof is well known, the
formulation will be brief and based on the estimates recalled in
Remark~\ref{pmest-rem}. 

In the following $(s_1,p_1,q_1)$ is arbitrary in $\Dm(L_u)$, ie
together with the given $(s_0,p_0,q_0)$ it fulfils
\eqref{stdom-eq}.
It is therefore seen from Remark~\ref{pmest-rem} and the Sobolev embeddings
that, with $p_2$ and $q_2$ as in Remark~\ref{pmest-rem},
\begin{equation}
  \begin{split}
  \pi_2(P_0\lOm\cdot, P_1\lOm\cdot)\colon
  B^{s_0}_{p_0,q_0}(\overline{\Omega})\oplus 
  B^{s_1}_{p_1,q_1}(\overline{\Omega})
  & \to B^{s_0-d_0+s_1-d_1}_{p_2,q_2}
\\
  &\hookrightarrow B^{s_1-d_1-(\fracc n{p_0}-s_0+d_0)}_{p_1,q_1}.
  \end{split}  
  \label{pi2-eq}
\end{equation}

The $\pi_1$-term in $L_u$ is straightforward to treat for
$s_0-d_0<\fracc n{p_0}$: in this case the Sobolev embedding 
$B^{s_0-d_0}_{p_0,q_0}\hookrightarrow 
 B^{s_0-d_0-\fracci n{p_0}}_{\infty,\infty}$ goes into a space with negative
smoothness index, so 
the estimate \eqref{pi1BBB-eq} gives, for $\varepsilon_0=0$,
\begin{equation}
  \pi_1(P_0\lOm\cdot, P_1\lOm\cdot)\colon
  B^{s_0}_{p_0,q_0}(\overline{\Omega})\oplus 
  B^{s_1}_{p_1,q_1}(\overline{\Omega})
 \to  B^{s_1-d_1-(\fracc n{p_0}-s_0+d_0)_+-\varepsilon_0}_{p_1,q_1}.
  \label{pi1-eq}
\end{equation}
In the same manner one has, since $u$ appears in the second entry of $\pi_3$,
that for $s_0-d_1<\fracc n{p_0}$ and $\varepsilon_1=0$,
\begin{equation}
  \pi_3(P_0\lOm\cdot, P_1\lOm\cdot)\colon
  B^{s_1}_{p_1,q_1}(\overline{\Omega})\oplus  
  B^{s_0}_{p_0,q_0}(\overline{\Omega})
 \to B^{s_1-d_0-(\fracc n{p_0}-s_0+d_1)_+-\varepsilon_1}_{p_1,q_1}.
  \label{pi3-eq}
\end{equation}
For $s_0-d_0>\fracc n{p_0}$ the estimate \eqref{pi1LB-eq} and 
$B^{s_0-d_0}_{p_0,q_0}\hookrightarrow L_\infty$
clearly yields the conclusion in \eqref{pi1-eq} with $\varepsilon_0=0$. 
The term with
$\pi_3$ may be treated analogously for $s_0-d_1> \fracc n{p_0}$, leading
to \eqref{pi3-eq} once again.
For $s_0-d_j=\fracc n{p_0}$ one can use \eqref{pi1-eq} and \eqref{pi3-eq} at
the expense of some $\varepsilon_j>0$, eg fulfilling $0<\varepsilon_1<d_1-d_0$,
or $\varepsilon_0=\varepsilon_1$ if $d_1=d_0$. This is unless $q_0\le1$ for 
then the embedding into $L_\infty$ applies.

Comparing the three estimates (incl.\ the $\varepsilon$-modifications), 
\eqref{pi2-eq} is the same as \eqref{pi1-eq},
except when $\fracc n{p_0}-s_0+d_0\le 0$, but in this case
 $B^{s_1-d_1}_{p_1,q_1}$ or $B^{s_1-d_1-\varepsilon_0}_{p_1,q_1}$
in \eqref{pi1-eq} clearly contains the space on
the right hand side of  \eqref{pi2-eq}. Similarly the co-domain of
\eqref{pi3-eq} equals the last space in \eqref{pi2-eq}, except 
for $\fracc n{p_0}-s_0+d_1\le 0$, 
but then the assumption that $d_0\le d_1$ yields that
also $\fracc n{p_0}-s_0+d_0\le 0$ so that there is an
embedding into the corresponding space in \eqref{pi1-eq}.
Regardless of whether $(\fracc n{p_0}-s_0+d_j)_+$ equals $0$ for none, one
or both $j$ in $\{0,1\}$, it follows
that $L_u$ is a bounded linear operator
\begin{equation}
  L_u\colon B^{s_1}_{p_1,q_1}\to B^{s_1-\omega}_{p_1,q_1},
  \label{LuOm-eq}
\end{equation}
when $\omega$ is as in \eqref{om-eq} and $(s_1,p_1,q_1)$ 
fulfils \eqref{stdom-eq}.

In the Lizorkin--Triebel case the above argument works with minor
modifications. For one thing  the 
Sobolev 
embedding 
$F^{s_0-d_0+s_1-d_1}_{p_2,t}\hookrightarrow 
F^{s_1-d_1-(\fracci n{p_0}-s_0+d_0)}_{p_1,q_1}$
and \eqref{pi2FFF-eq} give an analogue of \eqref{pi2-eq}.

Secondly, for $s_0-d_0<\fracc n{p_0}$, it is easy to see from the dyadic
corona criterion and the summation lemma (in analogy with the proof of
Lemma~\ref{tensor-lem}) that if $r<0$, 
\begin{equation}
  \pi_1(\cdot,\cdot)\colon
  B^{r}_{\infty,\infty}\oplus F^{s_1}_{p_1,q_1}\to
  F^{s_1+r}_{p_1,q_1}.
\end{equation}
Combining this with $F^{s_0-d_0}_{p_0,q_0}\hookrightarrow
B^{s_0-d_0-\fracci n{p_0}}_{\infty,\infty}$, 
formula \eqref{pi1-eq} is carried over to the
Lizorkin--Triebel case. Otherwise one may proceed as in the Besov case,
noting that $F^{n/p}_{p,q}\hookrightarrow L_\infty $ when $p\le 1$.
\end{proof}

To shed light on \eqref{DmLu-eq}, one could consider an elliptic problem
$\{A,T\}$, say with $A$ of order $2m$, $T$ of class $m$ and a solution $u\in
H^m(\overline{\Omega})$, with $(m,2)\in \Dm(Q)$, of
\begin{align}
  Au+Q(u)&=f \quad\text{in}\quad \Omega
  \\
  Tu&=\varphi \quad\text{on}\quad \Gamma.
\end{align}
According to \eqref{DmLu-eq}, $\Dm(L_u)$ then consists of parameters
$(s,p,q)$ with  
\begin{equation}
  s>\frac{d_0+d_1}{2}+(\frac np-\frac n2)_+-(m-\frac{d_0+d_1}{2}),
\end{equation}
so that $\Dm(L_u)$ is obtained from the quadratic standard domain $\Dm(Q)$
in \eqref{stQdom-eq}
simply by a downward shift given by the last parenthesis, which is positive
for $(m,2)\in \Dm(Q)$. Therefore
$\Dm(L_u)\supset\Dm(Q)$; by an extension of the argument 
this is seen to hold also in general when $(s_0,p_0,q_0)\in \Dm(Q)$. 

When deriving easy-to-apply criteria for $A$-moderacy,
for some given linear operator $A$ of constant order $d_A$ 
on a parameter domain $\Dm(A)$, it is clearly 
a necessary condition that $d_A>d_2+\max(d_0,d_1)$, for both $\sigma$ and
$\omega$ are $\ge d_2+\max(d_0,d_1)$.

\begin{cor}
  \label{Amod-cor}
Let $Q(u)$ be of product type $(d_0,d_1,d_2)$ with $d_0\le d_1$.
When $d_A> d_2+d_1$, then $Q$ is
$A$-moderate on every $E^{s}_{p,q}$ in $\Dm(A)\cap\Dm(Q)$ if $d_1-d_0\ge n$,
or else on the $E^s_{p,q}$ in $\Dm(A)\cap\Dm(Q)$ fulfilling
\begin{equation}
    s>\fracnp-d_A+d_0+d_1+d_2.
  \label{sdddd-eq}
\end{equation}
The exact paralinearisation $L_u$ is $A$-moderate on 
$\Dm(L_u)$ when $Q$ is $A$-moderate on the space  $E^{s_0}_{p_0,q_0}\ni u$.  
\end{cor}
\begin{proof}
Given \eqref{sdddd-eq}
one has $d_A-d_2-d_1>(\fracc n{p}-s+d_0)_+\ge0$.
So by taking $\varepsilon\in \,]0,d_A-d_2-d_1[\,$, clearly this gives
$d_A>\sigma$ so that $Q$ is $A$-moderate on $E^{s}_{p,q}$.
However, if $d_1-d_0\ge n$ it is easy to see, both for $p<2$ and $p\ge 2$,
that every $(s,p,q)$ fulfills 
\begin{equation}
 \tfrac{1}{2}(d_0+d_1)+(\fracc n{p}-\tfrac{n}{2})_+
  \ge
  \fracc n{p}+d_0.
\end{equation}
Concequently $s>\fracc n{p}+d_0$, so $\sigma=d_2+d_1$.
Hence $Q$ is $A$-moderate on the entire domain $\Dm(A)\cap\Dm(Q)$ in this case.

The statement on $L_u$ follows since $\omega$ equals $\sigma$ on
the space containing the linearisation point $u$.
\end{proof}

In cases with $d_1-d_0<n$, there always is a
part of the quadratic standard domain $\Dm(Q)$ where \eqref{sdddd-eq} must be
imposed. Indeed, the last two terms in \eqref{sgm-eq}
contributes to the value of $\sigma$ in the \emph{slanted slice} 
of $\Dm(Q)$ given by
\begin{equation}
  \tfrac{1}{2}(d_0+d_1)+(\fracc np-\tfrac{n}{2})_+
  < s \le \fracc np+d_0.
  \label{slice-eq}
\end{equation}
For $d_1-d_0<n$ any $p<2$ leads to solutions $(s,p)$
of these inequalities, so the slice in
\eqref{slice-eq} is non-empty. Because $\sigma>d_2+d_1$ in the slice,
$A$-moderacy is obtained only where $d_A>\sigma$, ie where \eqref{sdddd-eq}
holds. Note, however, that $L_u$ by the formulae for $\sigma$ and $\omega$ 
is born to be $A$-moderate
on the entire domain $\Dm(L_u)$, if only $Q$ is so on a space containing $u$.

\begin{rem}
  \label{mod-rem}
Concerning the model problem \eqref{bsc-eq} and Example~\ref{modl-ex},
where $d_0=0$, $d_1=1$ and $d_A=2$, 
the above \eqref{stQdom-eq} leads to the quadratic standard domains in 
\eqref{Qdm-eq} and \eqref{Ndm-eq}. Notice that the more important domains
$\Dm(\cal A,Q)$ and $\Dm(A,\cal N)$ in \eqref{delQ-eq} and
\eqref{Ndm-eq'} are obtained from the conjunction of
\eqref{stQdom-eq} and \eqref{sdddd-eq} (the latter is redundant for $n=2$ and
$n=3$). Similarly \eqref{Bdm-eq} follows  from \eqref{DmLu-eq}.
\end{rem}

\begin{rem}
One could compare \eqref{bsc-eq} (or the stationary Navier--Stokes problem)
with the von~Karman problem (cf Section~\ref{vK-sect}).
They both fulfil $d_1-d_0\le1<n$.
In the former problem \eqref{sdddd-eq} is felt, and the quadratic
term is only $\lap_{\gamma_0}$-moderate on the part of
$\Dm(Q)\cap\Dm_1$ where  $s>\fracnp-1$, by \eqref{sdddd-eq}. 
(For the Neumann condition, \eqref{sdddd-eq} gives again $s>\fracnp-1$,
that now should be imposed on the smaller region 
$\Dm(Q)\cap\Dm_2$ because the boundary condition has class $2$.)  
But in the von~Karman problem,
\eqref{sdddd-eq} is not felt, for it is fulfilled on all of the
quadratic standard domain of the form $[\cdot,\cdot]$, and even after this
has been extended to the $B(\cdot,\cdot)$ of
type $(1,1,2)$ given in Example~\ref{MA-ex}, 
it \emph{still} holds that $\omega<4=d_{\lap^2}$ on all of
$\Dm(Q)$. But nevertheless a small portion of $\Dm(Q)$ must be
disregarded to have $\lap^2$-moderacy, simply because the boundary condition
in the Dirichl\'et realisation of $\lap^2$ is felt; cf Figure~\ref{vK-fig}
below. 
In view of this, it seems pointless to generalise beyond 
Corollary~\ref{Amod-cor}.
\end{rem}

\subsection{Boundedness in a borderline case}
  \label{brdl-ssect}
In the cases given by equality in \eqref{nec2-cnd} 
it is more demanding to estimate $L_u$. For later reference
a first result on such extensions of $\Dm(L_u)$ is sketched. It adopts
techniques from a joint work with W.~Farkas and W.~Sickel
\cite{FaJoSi00}, in which approximation spaces
$A^s_{p,q}$ (that go back to S.~M.~Nikol$'$ski\u\i) were useful for the 
borderline investigations.

Recall that $A^s_{p,q}(\Rn)$ for $s\ge (\fracc np-n)_+$, $p,q\in
\,]0,\infty]$ (with $q\le1$ for $s=\fracnp-n$), 
consists of the $u\in \cal S'(\Rn)$ that have 
an $\cal S'$-convergent decomposition $u=\sum_{j=0}^\infty v_j$ fulfilling 
$\supp\hat v_j\subset\{\,|\xi|\le 2^{j+1} \,\}$ for
$v_j\in \cal S'\cap L_p$
with
\begin{equation}
  (\sum_{j=0}^\infty 2^{sjq}\norm{v_j}{L_p}^q)^{\fracci1q}<\infty.
\end{equation}
Then $\norm{u}{A^s_{p,q}}$ is the infimum of
these numbers, over all such decompositions.

The idea of \cite{FaJoSi00} is that, while the dyadic
ball criterion cannot yield convergence for $s=\fracc np-n$ (at least not
for $q>1$), 
one can sometimes show directly that such $\sum v_j$
converges to some $u$ in $L_1$ or $\cal S'$; 
then the finiteness of the above number gives $\sum v_j\in A^s_{p,q}$.
For this purpose the next borderline result is recalled
from \cite[Prop.~2.5]{JJ94mlt}.
\begin{lem}
  \label{F0-lem}
Let $0<q\le 1\le p<\infty$ and let $\sum_{j=0}^\infty u_j$ be such that
$F(q)<\infty$ for $F(q)=\nrm{(\sum|u_j|^q)^{1/q}}{p}$.
Then $\sum u_j$ converges in $L_p$ to a sum $u$ fulfilling 
$\norm{u}{L_p}\le F(q)$. 
\end{lem}
\begin{proof}
With $\sum |u_j(x)|$ as a majorant (since $F(1)\le F(q)$), 
$\norm{\sum_{j=k}^\infty |u_j|}{L_p}
\underset{k\to \infty}{\longrightarrow} 0$. 
Hence $\sum u_j$ is a fundamental series in $L_p$,
and the estimate follows. 
\end{proof}

\begin{thm}
  \label{mod'-thm}
Let $B=\pi_\Omega(P_0\cdot,P_1\cdot)$ with $d_0\le d_1$ and let $u\in
B^{s_0}_{p_0,q_0}(\overline{\Omega})$ be fixed. For $(s,p,q)$ such that
\begin{equation}
  s_0+s=d_0+d_1+(\fracc n{p_0}+\fracc n{p}-n)_+,\qquad
  \fracc1{q_2}:=\fracc1{q_0}+\fracc1{q}\ge 1
\end{equation}
the operator $L_u$ is continuous 
\begin{equation}
  L_u\colon B^{s}_{p,q}(\overline{\Omega})\to
     B^{s-\omega}_{p,\infty}(\overline{\Omega}),
\end{equation}
provided, in case $\fracc1{p_2}:=\fracc1{p_0}+\fracc1{p}>1$, that
$p_2\ge q_2$ or $p\ge1$ holds.   
Moreover, $L_u\colon F^{s}_{p,q}(\overline{\Omega})\to
B^{s-\omega}_{p,\infty}(\overline{\Omega})$ is continuous
if $u\in F^{s_0}_{p_0,q_0}(\overline{\Omega})$, when
$\kd{q_0>1}$ in \eqref{om-eq} is replaced by $\kd{p_0>1}$ (no restrictions
for $p_2<1$).
\end{thm}

\begin{proof}
With notation as in the proof of Theorem~\ref{mod-thm},
the assumption $q_2\le 1$ gives $\ell_{q_2}\hookrightarrow \ell_1$,
so for $p_2\ge1$ insertion of $1= 2^{s_0-d_0+s_1-d_1}$ into a double
application of H{\"o}lder's inequality shows that the series  
defining $\pi_2(P_0\lOm\cdot,P_1\lOm\cdot)$ converges absolutely 
in $L_{p_2}$. There is a Sobolev embedding 
$L_{p_2}\hookrightarrow B^{\tilde s}_{p_1,\infty}$ for
$\tilde s=s_1-d_1-(\fracc n{p_0}-s_0+d_0)$,
since $p_1\ge p_2$, 
so the conclusion of \eqref{pi2-eq} holds with the modification that the
sum-exponent is $\infty$ in this case.

For $p_2<1$ one uses
the Nikol$'$ski\u\i--Plancherel--Polya inequality to estimate $L_1$-norms
by  $2^{\fracci n{p_2}-n}=2^{s_0+s_1-d_0-d_1}$ times corresponding
 $L_{p_2}$-norms, leading to convergence in $L_1$. 
After this convergence has been established, the same estimates also give
the strengthened conclusion that,
for $A^s_{p,q}$ as above,
\begin{equation}
  \pi_2(P_0\lOm\cdot,P_1\lOm\cdot)\colon
  B^{s_0}_{p_0,q_0}\oplus B^{s_1}_{p_1,q_1}\to
  A^{\fracci n{p_2}-n}_{p_2,q_2}.
\end{equation}
By \cite[Thm.~6]{FaJoSi00} the conjunction of $r\ge\max(p_2,q_2)$ and
$o=\infty$ is equivalent to 
\begin{equation} 
    A^{\fracci n{p_2}-n}_{p_2,q_2}\hookrightarrow B^{\fracci nr-n}_{r,o}.
  \label{AB-eq}
\end{equation}
Therefore  $\pi^{12}_{2\Omega}(u,\cdot)
:=\rOm\pi_{2}(P_0\lOm u,P_1\lOm \cdot)$ is continuous 
$B^{s_1}_{p_1,q_1}(\overline{\Omega})\to 
B^{\fracci n{p_2}-n}_{p_2,\infty}(\overline{\Omega})$ for $p_2\ge{q_2}$, 
hence into $B^{\tilde s}_{p_1,\infty}(\overline{\Omega})$
as desired; for $p_1\ge 1$ this is seen directly from
the above $L_1$-estimate. 

Since \eqref{pi1-eq} and \eqref{pi3-eq} also hold in the present context, 
and since this implies weaker statements with the sum-exponents equal to
$\infty$ on 
the right hand sides there,  $L_u$ has the property in \eqref{LuOm-eq} except
that the co-domain should be $B^{s_1-\omega}_{p_1,\infty}$.

For the $F^{s}_{p,q}$-spaces the estimates of 
$\pi^{12}_{2\Omega}(u,\cdot)$ are derived
in the same way, except that the $\ell_{q_2}$-norms are calculated
pointwisely, before the $L_{p_2}$-norms. Indeed, for $p_2\ge 1$, 
Lemma~\ref{F0-lem} gives 
(since $q_2\le 1$ in this case) that $\pi_2(P_0\lOm\cdot,P_1\lOm\cdot)$
maps $F^{s_0}_{p_0,q_0}\oplus F^{s_1}_{p_1,q_1}$ 
to $L_{p_2}$: for $p_2>1$ this co-domain is embedded 
via $F^{\tilde s}_{p_1,q_1}$ into $B^{\tilde s}_{p_1,\infty}$, while  
$L_{p_2}\hookrightarrow B^{0}_{1,\infty}\hookrightarrow
 B^{\tilde s}_{p_1,\infty}$ for $p_2=1$.

For $p_2<1$ one finds by the vector-valued Nikol$'$ski\u\i--Plancherel--Polya
inequality in Lemma~\ref{vNPP-lem} that eg 
(when $f_k:=\Phi_k(D)f$ etc on $\Rn$)
\begin{equation}
  \nrm{\sum_{k=0}^\infty |f_kg_k|}{1} \le
  c\nrm{(\sum_{k=0}^\infty 2^{k(\fracci n{p_2}-n)q_2}
           |f_kg_k|^{q_2})^{\fracci1{q_2}}}{{p_2}} 
 \le c' \norm{f}{F^{s_0-d_0}_{p_0,q_0}}\norm{g}{F^{s_1-d_1}_{p_1,q_1}}.
  \label{AFF-eq}
\end{equation}
In this way $\pi^{12}_{2\Omega}(u,\cdot)$ is shown to map 
$F^{s_1}_{p_1,q_1}$ into $L_1(\Omega)$.
Hence into $B^{\tilde s}_{p_1,\infty}(\overline{\Omega})$ for $p_1\ge1$.
In general there is $p_3\in \,]p_2,p_1[\,$ ($p_0<\infty$) and
the $A^{\fracci n{p_3}-n}_{p_3,p_3}$-norm of $\pi^{12}_{2\Omega}(u,v)$ 
is estimated by an $L_{p_2}(\ell_{q_2})$-norm as in the middle of
\eqref{AFF-eq}, for the sum and integral may be excanged
and the estimate realised through Lemma~\ref{vNPP-lem}.
By \eqref{AB-eq}--\eqref{AFF-eq} this means that
$\pi^{12}_{2\Omega}(u,\cdot)$ maps $F^{s_1}_{p_1,q_1}$ into 
$B^{\fracci n{p_3}-n}_{p_3,\infty}\hookrightarrow
B^{\tilde s}_{p,\infty}$ for $p_2<1$.
Comparison with the $F^{s}_{p,q}$-results for the other terms shows that
 $L_u\colon F^{s_1}_{p_1,q_1}\to B^{\tilde s}_{p_1,\infty}$.
\end{proof}

The above result suffices for the present paper, but it could probably be
sharpened in several ways, 
perhaps with a consistent use of $A^s_{p,q}$ as co-domains.

\subsection{Relations to pseudo-differential 
operators of type $\mathbf{1},\mathbf{1}$}
  \label{sym-ssect}
For the local regularity improvements later, it is useful to express
paralinearisations in terms of pseudo-differential operators 
with symbols in $S^d_{1,1}$. 
Recall that $a(x,\xi)\in C^\infty(\R^{2n})$ belongs to  
$S^d_{1,1}(\Rn\times\Rn)$ for $d\in\R$,
if for all multiindices $\alpha$, $\beta$ there is $c_{\alpha\beta}>0$
such that for $x$, $\xi\in \Rn$,
\begin{equation}
 |D^\beta_{x}D^\alpha_\xi a(x,\xi)|\le c_{\alpha\beta} 
 \ang{\xi}^{d-|\alpha|+|\beta|};
  \qquad \ang{\xi}=(1+|\xi|^2)^{1/2}.
  \label{Sd11-eq}  
\end{equation}
The operator $a(x,D)\varphi(x)
=(2\pi)^{-n}\int_{\Rn} e^{\im x\cdot\xi}a(x,\xi)\hat\varphi(\xi)\,d\xi$ 
obviously induces a bilinear map 
$S^d_{1,1}\times\cal S\to\cal S$ that is continuous with respect to the
Fr\'echet topologies. 
In general $A:=a(x,D)=\OP(a)$ 
cannot be extended to $\cal S'$ by duality, for the
adjoint of $A$ need not be of type $1,1$.
However, $A$ can be defined as a linear operator with domain
$D(A)\subset\cal S'(\Rn)$ by analogy with \eqref{pi-eq}. More precisely
$u\in \cal S'$ is in $D(A)$ when the limit
\begin{equation}
  a_\psi(x,D)u:=\lim_{k\to\infty } \OP(\psi_k(D_x)a(x,\xi)\psi_k(\xi))u
  \label{vfm-eq}
\end{equation}
exists in $\cal D'(\Rn)$ for all $\psi\in C^\infty_0(\Rn)$ with $\psi=1$ in a
neighbourhood of the origin, and when moreover 
$a_{\psi}(x,D)u$ is independent of such $\psi$ so that it
makes sense to let $a(x,D)u=a_\psi(x,D)u$ then. 

This definition by so-called vanishing frequency modulation was
introduced recently and investigated from several perspectives in
\cite{JJ08vfm}. 
As the symbol on the right-hand side of \eqref{vfm-eq} 
is in $S^{-\infty }$ the definition
means roughly that in $a(x,D)u$ one should regularise the symbol $a$
instead of the argument $u$; it clearly gives the integral after
\eqref{Sd11-eq} for $u\in \cal S(\Rn)$.

Previously
L.~H{\"o}rmander  determined (up to a limit point)
the $s$ for which $A$ extends to a continuous map $H^{s+d}_2\to H^s_2$;
cf \cite{H88,H89} and \cite[Ch~9.3]{H97}. 
Eg continuity for all $s\in \R$ is proved there for $a(x,\xi)$ satisfying
his twisted diagonal condition. 
However, it was proved in \cite{JJ04Dcr,JJ05DTL}
that there always are bounded extensions, for $1\le p<\infty$, 
\begin{equation}
  F^d_{p,1}(\Rn)\xrightarrow{a(x,D)} L_p(\Rn),
  \qquad
  B^d_{\infty,1}(\Rn)\xrightarrow{a(x,D)} L_\infty(\Rn) ,
  \label{BFd11-eq}
\end{equation}
and that, without further knowledge about $a(x,\xi)$, this is optimal 
within the $B^{s}_{p,q}$ and $F^{s}_{p,q}$ scales for $p<\infty$. 
For $s>(\fracnp-n)_+$ there is continuity 
\begin{equation}
  B^{s+d}_{p,q}(\Rn)\xrightarrow{a(x,D)} B^{s}_{p,q}(\Rn),
  \qquad
  F^{s+d}_{p,q}(\Rn)\xrightarrow{a(x,D)} F^{s}_{p,r}(\Rn) \quad\text{($r$ as
in \eqref{r-eq})}.  
\end{equation}
This extends to all $s\in \R$ under the twisted diagonal condition; cf
\cite[Cor.~6.2]{JJ05DTL}; cf also \cite{JJ08vfm}.
The reader may consult \cite{H97,JJ08vfm} for 
various aspects of the theory of operators in $\OP(S^d_{1,1})$. 

The just mentioned results will not be
directly used here, but they shed light on how difficult it is to determine
the domain $D(A)$. Nevertheless one has the \emph{pseudo-local} property:
\begin{equation}
  \singsupp A u\subset \singsupp u\quad\text{ for all }\quad u\in D(A).
  \label{ssupp-eq}
\end{equation}

\begin{thm}
  \label{Sd11ss-thm}
Every pseudo-differential operator $a(x,D)$ in
$\op{OP}(S^{d}_{1,1}(\Rn\times\Rn))$ has the property in \eqref{ssupp-eq}.
\end{thm}
This was first proved in \cite[Thm.~6.4]{JJ08vfm}, to which  
the reader is referred. The proof given there 
exploits the definition of type
$1,1$-operators given above as well as the Regular Convergence Lemma; cf
Lemma~\ref{singsupp-lem}.

The exact paralinearisations turn out to factor through
pseudo-differential operators of type $1,1$.
This entails that the former are
pseudo-local:

\begin{thm}
  \label{symb-thm}
Let $B$ be of product type and
$u\in B^{s_0}_{p_0,q_0}(\overline{\Omega})$ for some arbitrary
$(s_0,p_0,q_0)$.
Then the exact paralinearisation in \eqref{plin-eq}  
factors through an operator  
$P_u\in\op{OP}(S^{\omega}_{1,1}(\Rn\times\Rn))$ with
$\omega$ as in \eqref{om-eq}. That is, for every $(s,p,q)$ in 
$\Dm(L_u)$, cf \eqref{DmLu-eq},
there is a commutative diagram
\begin{equation}
\begin{CD}
  E^s_{p,q}(\overline{\Omega})   @>{\ell_\Omega}>> E^s_{p,q}(\Rn)  \\
 @V{L_u}VV    @VV{P_u}V \\
  E^{s-\omega}_{p,q}(\overline{\Omega})   @<<{\rOm}< E^{s-\omega}_{p,q}(\Rn). 
\end{CD}
  \label{LuP-eq} 
\end{equation}
Moreover, $g\mapsto L_ug$ is pseudo-local when
$g\in E^s_{p,q}(\overline{\Omega})$ and $(s,p,q)$ is in $\Dm(L_u)$. 
\end{thm}

\begin{proof}
$1^\circ $. By linearity, it suffices to treat 
$P_m=D^{\eta_m}$ for $|\eta_m|=d_m$, and $d_0\le d_1$, $d_2=0$.
Set $\tilde u=\lOm u$.

$2^\circ $.
Applying $L_u$ to $\lOm g\in \cal S$, 
it is a composite $L_u=\rOm a(x,D)\lOm$ for a symbol
$a(x,\xi)$ satisfying \eqref{Sd11-eq} for $d=\omega$ with $\omega$ as in
\eqref{om-eq}, namely
\begin{equation}
  a(x,\xi)=-\sum_{j=0}^\infty \bigl(\Psi_{j+1}(D_x)D^{\eta_0}_x
            \tilde u(x)\xi^{\eta_1}+ 
            \Psi_{j-2}(D_x)D^{\eta_1}_x \tilde u(x)\xi^{\eta_0}
             \bigr)\Phi_{j}(\xi)
  \label{Lusym-eq}
\end{equation}
Indeed, the formula for $a(x,\xi)$ 
follows directly from Definition~\ref{plin-defn} and
\eqref{paramultiplication-eq} once $a\in S^\omega_{1,1}$ has been verified. 
To prove that 
$P_u=a(x,D)$ is of type $1,1$, note that $a(x,\xi)$ is $C^\infty$ since
each $\xi$ is in $\supp{\Phi_j}$ for at most two
values of $j$, and for these $2^{j-1}\le|\xi|\le 2^{j+1}$,
so that
$|D^\alpha(\xi^{\eta_m} \Phi_j(\xi))|\le c\ang{\xi}^{d_m-|\alpha|}$ 
holds for all $\alpha$. Concerning the estimates for $x\in\Rn$ and
$\xi\in \supp\Phi_j$, so that $2^j\le 2\ang{\xi}$, note that
if $k=j+1$ and $\varepsilon>0$ is fixed, the convenient short-hand
$\varepsilon':=\varepsilon\kd{\fracci n{p_0}-s_0+d_0=0}\kd{q_0>1}$
fulfils $\varepsilon'\ge0$ and gives
\begin{equation}
  |D^{\beta}_x\Psi_k(D)D^{\eta_0}\tilde u(x)|\le c
  \ang{\xi}^{|\beta|+(\fracci n{p_0}-s_0+d_0)_+
   +\varepsilon'}.
  \label{term1-eq}
\end{equation}
In fact, for $q_0\le1$ one has $\ell_q\hookrightarrow \ell_1$, so the
Nikol$'$ski\u\i--Plancherel--Polya inequality yields
\begin{equation}
  \begin{split}
  |\Psi_k(D)D^{\beta+\eta_0}\tilde u(x)| &\le c
  \smash[b]{\sum_{l=0}^k} 2^{l(s_0-|\beta+\eta_0|)}
    \norm{\Phi_l(D)D^{\beta+\eta_0}\tilde u}{L_{p_0}}
   2^{l(|\beta|+\fracci n{p_0}-s_0+d_0)}
\\
  &\le c     \norm{u}{B^{s_0}_{p_0,q_0}}
   \ang{\xi}^{(\fracci n{p_0}-s_0+d_0)_++|\beta|};
  \end{split}
  \label{symb-est}
\end{equation}
for $q_0>1$ H{\"o}lder's inequality applies to the first line in
\eqref{symb-est}, if
$2^{k(\fracci n{p_0}-s_0+d_0)_++k|\beta|}$ 
is taken out in front of the summation (it is less than
$(4\ang{\xi})^{|\beta|+(\fracci n{p_0}-s_0+d_0)_+}$);
except when $\fracci n{p_0}-s_0+d_0=0$, ie $\varepsilon'>0$,
then $|\beta|$ should just have $\varepsilon$ added and subtracted. 
This shows \eqref{term1-eq}.

Terms with $|\Psi_{j-2}(D)D^{\beta+\eta_1}\tilde u(x)|$ are treated
analogously, in the first line of \eqref{symb-est} the factor 
$2^{l(s_0-|\beta+\eta_0|)}$ may be estimated by $2^{l(s_0-|\beta+\eta_1|)}$
(which is absorbed by the Besov norm on $u$) times $2^{j(d_1-d_0)}$;
the latter, together with the estimate of
$D^\alpha(\xi^{\eta_0}\Phi_j(\xi))$, gives the 
estimates in \eqref{Sd11-eq} also for these terms.

$3^\circ $.
To prove \eqref{LuP-eq} also for non-smooth functions, it is noted that
there is a linear map
$P_u\colon E^s_{p,q}(\Rn)\to E^{s-\omega}_{p,q}(\Rn)$ 
that is bounded for $(s,p,q)\in \Dm(L_u)$. 
This is seen as 
in the proof of Theorem~\ref{mod-thm}, cf \eqref{LuOm-eq},
for one can keep the first entry in the 
expressions with $\pi_1$, $\pi_2$ there 
equal to $P_0\tilde u$ while the other entry runs through
$P_1(E^s_{p,q}(\Rn))$, for $\pi_3$ the first entry is taken in
$P_0(E^s_{p,q})$ and the second equal to $P_1\tilde u$.
From the definition of $L_u$ it is then evident that 
$L_u=\rOm\circ P_u\circ\lOm$, hence \eqref{LuP-eq} holds.

$4^\circ $. To show that $P_u$ from 
step $3^\circ $ equals the type $1,1$-operator
$a(x,D)$ given by the symbol in step $2^\circ $, 
it remains by \eqref{vfm-eq} to be verified that one has 
the limit relation
$P_uf=\lim_{m\to\infty }\OP(\psi_m(D_x)a(x,\xi)\psi_m(\xi))f$ for all
$\psi\in C^\infty_0(\Rn)$ with $\psi=1$ around $0$, 
whenever $f\in B^{s}_{p,q}(\Rn)$ with
$(s,p,q)\in \Dm(L_u)$, ie for
\begin{equation}
  s_0-d_0+s-d_1>\max(0,\fracc n{p_0}+\fracc np-n).
  \label{Rm-ineq}
\end{equation}
(If $f\in F^{s}_{p,q}$, then $f\in  B^{s}_{p,\infty }$
that also fulfils \eqref{Rm-ineq}.)
This is tedious but results from consistent use of the techniques that
gave boundedness of $P_u$. 

Indeed, for every $\psi$ and a (large) $m$ as above, it is straightforward
to see that $\Phi_j\psi_m=\Phi_j$ and $\Psi_{j+1}\psi_m=\Psi_{j+1}$ 
for $j$ below a certain limit $J(m)$, so that
the symbol of the approximating operator can be written as follows, when $'$
indicates summation over $l=m-j$ in a fixed finite subset of $\Z$,
\begin{equation}
  \begin{split}
  \psi_m(D_x)a(x,\xi)\psi_m(\xi) 
&=
  -\sum_{j\le J(m)} \bigl(\Psi_{j+1}(D_x)D^{\eta_0}_x
            \tilde u(x)\xi^{\eta_1}+ 
\\[-2\jot]
&\hphantom{-\sum_{j\le J(m)} \bigl(\Psi_{j+1}(D_x)(D_x)}
            \Psi_{j-2}(D_x)D^{\eta_1}_x \tilde u(x)\xi^{\eta_0}
             \bigr)\Phi_{j}(\xi)
\\  
 &\quad -\sideset{}{'}\sum_l\bigl(\psi_m(D_x)\Psi_{m-l+1}(D_x)D^{\eta_0}_x
            \tilde u(x)\xi^{\eta_1}+ 
\\[-1\jot]
 &\qquad
            \psi_m(D_x)\Psi_{m-l-2}(D_x)D^{\eta_1}_x \tilde u(x)\xi^{\eta_0}
             \bigr)\Phi_{m-l}(\xi)\psi_m(\xi).
  \end{split}
\end{equation}
The operator induced by the first sum here converges to $P_u$
for $m\to\infty$, by \eqref{Lusym-eq} and the construction of $P_u$. 
Therefore it suffices to show
that the primed sum defines an operator $R_m$ 
for which $R_mf\to0$ for $m\to\infty $.
Fixing $l$ one has the contribution
\begin{multline}
  R_{l,m}f=
  (\psi_m(D)\Psi_{m-l+1}(D)D^{\eta_0}
            \tilde u\cdot  D^{\eta_1}
\\
   +\psi_m(D)\Psi_{m-l-2}(D)D^{\eta_1}\tilde u(x)\cdot 
            D^{\eta_0})\Phi_{m-l}(D)\psi_m(D) f,
\end{multline}
the worst part of which is 
\begin{equation}
  \tilde R_{l,m}f=
  \psi_m(D)(\Phi_{m-l-1}+\Phi_{m-l}+\Phi_{m-l+1})(D)D^{\eta_0}
            \tilde u\cdot  D^{\eta_1}\Phi_{m-l}(D)\psi_m(D) f.
\end{equation}
Clearly $\supp\cal F\tilde R_{l,m}f$ is contained in $B(0,c2^m)$, ie it
fulfils the dyadic ball condition in Lemma~\ref{B-lem}. To estimate the
quantity $B$ there, note that 
in case $p, p_0\ge 1$ the family $\psi_m(D)$ is uniformly bounded in $L_p$
and $L_{p_0}$,  
so when $\fracc 1{p_2}=\fracc1{p_0}+\fracc1p$ and
$\fracc 1{q_2}=\fracc1{q_0}+\fracc1q$, then
\begin{equation}
  \begin{split}
  (\sum_{m=0}^\infty  2^{(s_0+s-(d_0+d_1))mq_2} 
                     \nrm{\tilde R_{l,m}f}{p_2}^{q_2})^{\fracci1{q_2}}
  &\le c\norm{\tilde u}{B^{s_0}_{p_0,q_0}}
\\[-2,5\jot]
  &\quad \times (\sum_{m=0}^\infty 2^{(s-d_1)mq}
                      \nrm{\psi_m(D)\Phi_{m-l}(D)D^{\eta_1}f}{p}^q)^{\fracci1q}
\\
  &\le c' \norm{f}{B^{s}_{p,q}} <\infty .  
  \end{split}
  \end{equation}
Hence $\sum_{m=0}^\infty \tilde R_{l,m}f$ converges by Lemma~\ref{B-lem}
(cf \eqref{Rm-ineq}), so
as desired $\tilde R_{l,m}f\to 0$. 
If $p$ and/or $p_0$ is in $\,]0,1[\,$ one can use Sobolev embeddings
into $B^{s+n-\fracci n{p}}_{1,q}$ and $B^{s_0+n-\fracci n{p_0}}_{1,q_0}$,
since these spaces also fulfil \eqref{Rm-ineq}.

The rest of $R_{l,m}f$ may be handled with Lemma~\ref{DCC-lem},
as done in the $\pi_1$- and $\pi_3$-parts of $P_u$
(this is also analogous to the proof of Lemma~\ref{tensor-lem}). 
This shows that 
$\sum_{m=0}^\infty R_{l,m}f$ converges in $\cal S'$ so that
$\lim_m R_{l,m}f=0$, 
hence $\lim_m\sideset{}{'}\sum_{l} R_{l,m}f=\lim_m R_mf=0$.
Hence $P_u$ is of type $1,1$ as claimed.

$5^\circ $.
If $g$ is as in the theorem, 
$x\in \singsupp \lOm g$ implies  that $x\in \singsupp g\cup\Rn\setminus\Omega$.
By $4^\circ $ and Theorem~\ref{Sd11ss-thm}, $\singsupp\lOm g$ is not 
enlarged by $P_u$, so $\rOm P_u\lOm g$ is $C^\infty$ in the part of
$\Omega$ where $g$ is so.
\end{proof}

\begin{rem}
As indicated above, the theory of type $1,1$ operators 
is still far from complete. To avoid any ambiguity, the
exact paralinearisations have been defined here without reference
to these operators, and the Paralinearisation Theorem was for the same
reason proved directly, before the factorisation through type $1,1$
operators was established. 
\end{rem}

\begin{rem}   \label{symb-rem}
One way to attempt a symbolic calculus would be to
replace $\lOm$ by $e_\Omega$, ie by extension by
zero outside of $\Omega$. 
The resulting linearisation $\tilde L_u$ would have the form
$\tilde L_ug=\rOm P\eOm g$ where
$P$ is in $\op{OP}(S^{\omega}_{1,1}(\Rn\times\Rn))$
as in Theorem~\ref{symb-thm}.
For $\tilde L_u$ to have
order $\omega$ in spaces with $s>0$, it is envisaged
that the transmission property would be needed for $P$. However, 
transmission \emph{conditions} have been worked out for $S^d_{\rho,\delta}$
with $\delta<1$, cf \cite{GH}. For $\delta=1$ there is 
a fundamental difficulty because $\op{OP}(S^{\omega}_{1,1})$ in general, cf
\eqref{BFd11-eq}, is defined on $H^s_p$ for $s>\omega>0$\,---\,whereas  
the usual induction proof of the continuity 
of truncated pseudo-differential operators with transmission property
effectively requires application to spaces with $s<0$ 
(in the induction step,
$\rOm P$ is applied to distributions supported by the boundary
$\Gamma\subset\Rn$). 
Also the powers $(R_DL_u)^N$ should be covered, so the
general rules of composition with the operators in the Boutet de~Monvel
calculus should be established.
All in all this is better investigated elsewhere; it could
be useful eg in reductions where traces or solution operators
of other problems are applied to the parametrix formula.
\end{rem}

\section{The von Karman equations of non-linear vibration}
  \label{vK-sect}
The preceding sections apply to von Karman's equations
for a thin, buckling plate, initially filling an open domain
$\Omega\subset\R^2$ with $C^\infty $-boundary $\Gamma$. 
The following is inspired by 
\cite[Ch.~1.4]{L} and by the treatise of P.~G.~Ciarlet
\cite[Ch.~5]{Ci97}, that also settles the applicability of the model to
physical problems.

In the stationary case the problem is to find two real-valued functions
$u_1$ and $u_2$ (displacement and stress) defined in $\Omega$ and fulfilling 
\begin{subequations}
  \label{vK-eqs}  
\begin{alignat}{2}
  \lap^2 u_1 - [u_1,u_2]&=f \quad&&\text{in $\Omega$}  \\
  \lap^2 u_2 + [u_1,u_1]&=0 \quad&&\text{in $\Omega$}  \\
  \gamma_k u_1&=0           \quad&&\text{on $\Gamma$ for $k=0$, $1$} \\
  \gamma_k u_2&=\psi_k      \quad&&\text{on $\Gamma$ for $k=0$, $1$}.
\end{alignat}
\end{subequations}
Hereby $\lap^2$ denotes the biharmonic operator, whilst $[\cdot,\cdot]$
as in Example~\ref{MA-ex} stands for the bilinear operator 
\begin{equation}
  [v,w]=D^2_1 vD^2_2 w+
        D^2_2 vD^2_1 w-2D^2_{12}vD^2_{12}w.
\end{equation}
For the real-valued case with $\psi_0=\psi_1=0$, it is well known that
Brouwer's fixed point theorem implies the existence of 
solutions with $u_j\in F^{2}_{2,2}(\overline{\Omega})$ for given data
$f\in F^{-2}_{2,2}(\overline{\Omega})$; cf \cite[Thm.~4.3]{L} and
\eqref{HF-eq}. 
For $\psi_k\in F^{2-k-1/2}_{2,2}(\Gamma)$ solutions are
established by non-linear minimisation in \cite[Thm.~5.8-3]{Ci97}.
Concerning the regularity it was eg shown in \cite[Thm.~4.4]{L} that if
$f\in L_p(\Omega)$ for some $p>1$, then any of the above solutions 
of \eqref{vK-eqs} fulfils
that $u_1\in F^{4}_{p,2}(\overline{\Omega})$ while $u_2$ belongs to
$F^{4}_{q,2}(\overline{\Omega})$ for any $q<\infty$. 
It was also noted in \cite{L} that iteration would give more, eg that the
problem is hypoelliptic. Corresponding results for non-trivial $\psi_0$ and
$\psi_1$ may be found in   \cite[Thm.~5.8-4]{Ci97}.

\bigskip

These results are generalised in three ways in the present paper, as a
consequence of the general investigations: firstly the assumptions on the data
and on the solution $(u_1,u_2)$ are considerably weaker,
including fully inhomogeneous data; secondly the weak solutions
are carried over to a wide range of spaces with $p\ne2$. Thirdly the
non-linear terms are shown to have no influence on the solution's regularity
(within the Besov and Lizorkin--Triebel scales).

In the discussion of \eqref{vK-eqs}, the coupling of the two
non-linear equations is a little inconvenient, since the Exact
Paralinearisation Theorem, \ref{mod-thm}, needs a modification to this
situation. But this can 
be done easily when $u_1$ and $u_2$ are given in the same space, for in the
proof of Theorem~\ref{mod-thm} the mapping properties will then remain
the same regardless of whether $u_1$ or $u_2$ is inserted in the various
$\pi_j$-expressions. For brevity, it is left for the reader to substantiate
this expansion of the theorem. (More general methods follow in
Section~\ref{sys-sect}.)

Because $[v,w]$ is of type $(2,2)$, the quadratic standard domain in 
\eqref{stQdom-eq} is for $Q_0(u):=[u,u]$ given 
by $s>2+(\fracc 2p-1)_+$, and clearly $(s,p,q)=(2,2,2)$ is at the boundary
of and therefore outside of $\Dm(Q_0)$; cf Figure~\ref{vK-fig}. 
Hence Theorem~\ref{mod-thm}  does barely
not apply as it stands.

To carry over weak solutions to other spaces, one can use 
the more refined estimates for the borderlines
in Theorem~\ref{mod'-thm}. In fact the co-domain of type
$B_{p,\infty}$ is embedded into
$E^{s-\omega-\varepsilon}_{p,q}$ for $\varepsilon>0$,
so this gives that $L_{(u_1,u_2)}$ has order $\omega=3+\varepsilon$ 
when both $(s_0,p_0,q_0)$ and $(s,p,q)$ equal $(2,2,2)$. For other choices
of $(s,p,q)$ the continuity properties of $L_{(u_1,u_2)}$ are given by
Theorem~\ref{mod-thm}. In addition, 
$L_{(u_1,u_2)}$ linearises the non-linear terms in \eqref{vK-eqs}, 
for at $(2,2,2)$ these only contains products of $L_2$-functions, whence the
conclusions of Lemma~\ref{Lu-lem} remain valid (the assumption 
$s>\max(d_0,d_1)$ is then not needed in the proof of the lemma).
In this way
Theorem~\ref{ir-thm} can be used for the von~Karman problem, when 
$\Dm(\cal N)$ is taken as $\Dm(Q_0)\cup\{\,(2,2,2)\,\}$ and $\Dm(B_u)$
likewise is the union of $\Dm(L_{(u_1,u_2)})$ and $\{(2,2,2)\}$. (Parameter
domains were not required to be open in Theorem~\ref{ir-thm}.)

One could also envisage
other problems in which the weak solutions belong to spaces at the
borderline of the quadratic standard domain, so that results like
Theorem~\ref{mod'-thm} 
would be the only manageable way to apply Theorem~\ref{ir-thm}.

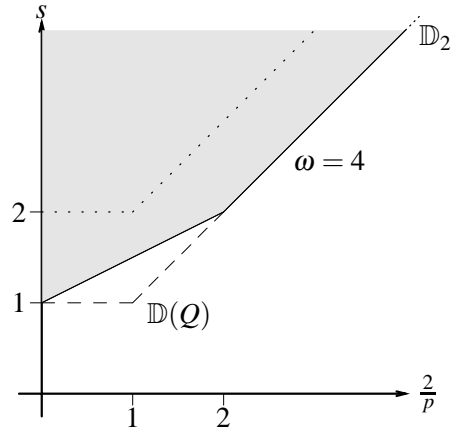
\begin{figure}[htbp]
\setlength{\unitlength}{0.0004in}
{\renewcommand{\dashlinestretch}{30}
\begin{picture}(5705,5800)(0,-10)
\thicklines
\path(405,5092)(375,5212)(345,5092)
\path(375,5212)(375,312)(5175,312)
\path(5055,282)(5175,312)(5055,342)
\path(75,312)(375,312)(375,12)
\thinlines
\path(1575,312)(1575,162)
\path(2775,312)(2775,162)
\path(225,1512)(375,1512)
\path(225,2712)(375,2712)
\texture{0 0 0 888888 88000000 0 0 80808 
	8000000 0 0 888888 88000000 0 0 80808 
	8000000 0 0 888888 88000000 0 0 80808 
	8000000 0 0 888888 88000000 0 0 80808 }
\shade\path(375,5112)(375,1512)(2775,2712)(5175,5112) 
\path(375,1512)(2775,2712)(5175,5112) 
\dottedline{80}(5175,5112)(5375,5312)
\dashline{190}(375,1512)(1575,1512)(2775,2712)
\dottedline{180}(375,2712)(1575,2712)(3988,5125)
\put(375,5262){\makebox(0,0)[cb]{$s$}}
\put(5475,310){\makebox(0,0){$\fracc2p$}}
\put(1575,0){\makebox(0,0)[cc]{$1$}}
\put(2775,0){\makebox(0,0)[cc]{$2$}}
\put(75,1512){\makebox(0,0)[cc]{$1$}}
\put(75,2712){\makebox(0,0)[cc]{$2$}}
\put(1750,1337){\makebox(0,0)[lc]{$\Dm(Q)$}}
\put(3675,3387){\makebox(0,0)[lc]{$\omega=4$}}
\put(5575,5012){\makebox(0,0){$\Dm_2$}}
\end{picture}
}
\caption{The quadratic standard domains of $Q$ and 
 $Q_0$ (in dots) in relation to $\Dm_2$.}
  \label{vK-fig}
 \end{figure}
For the von~Karman problem, however, the symmetry properties of $[v,w]$
make it possible to avoid the rather specialised estimates in
Theorem~\ref{mod'-thm}. Indeed,  as recalled in Example~\ref{MA-ex}, 
$[\cdot,\cdot]$ is a restriction of
\begin{equation}
  B(v,w)= D^2_{12}(D_1vD_2w+D_2vD_1w)-D^2_1(D_2vD_2w)
          -D^2_2(D_1vD_1w).
  \label{BvK-eq}
\end{equation}
Since $B$ is of type $(1,1,2)$, the larger domain $\Dm(Q)$ is given by
$s>1+(\fracc 2p-1)_+$ according to \eqref{stQdom-eq}. 
But by \eqref{Dmk-eq} the appropriate parameter domain for the linear part is
$\Dm_2$, and $\Dm(Q)\cap \Dm_{2}= \Dm_{2}$, cf Figure~\ref{vK-fig}.

On the resulting domain $\Dm_2$, the operator $Q$ is $\lap^2$-moderate
in view of
Corollary~\ref{Amod-cor}. 
It is moreover easy to infer from \eqref{om-eq} that $\omega=4$ holds on
the borderline with $s=2/p$ (for $p<1$) of $\Dm_2$.

This leads to the following result on the fully inhomogeneous problem: 
\begin{thm}
  \label{vK-thm}
Let two functions  $u_1$, $u_2\in B^{s}_{p,q}(\overline{\Omega})$ 
with $(s,p,q)$ in $\Dm_2$ solve
\begin{subequations}\label{vK'-eqs}
\begin{align}
  \lap^2 u_1 - B(u_1,u_2)&=f_1 \quad\text{in $\Omega$}  
  \label{vK'1-eq} \\
  \lap^2 u_2 + B(u_1,u_1)&=f_2 \quad\text{in $\Omega$}  
  \label{vK'2-eq}\\
  \gamma_k u_1 &=\varphi_k \quad\text{on $\Gamma$ for $k=0$, $1$} \\
  \gamma_k u_2&=\psi_k \quad\text{on $\Gamma$ for $k=0$, $1$},
\end{align}
\end{subequations}
for data $f_k\in B^{t-4}_{r,o}(\overline{\Omega})$,
with $k=1$, $2$, together with
 $\varphi_0$, $\psi_0\in B^{t-\fracci1r}_{r,o}(\Gamma)$
and $\varphi_1$, $\psi_1\in B^{t-1-\fracci1r}_{r,o}(\Gamma)$
whereby $(t,r,o)\in \Dm_2\cap\Dm(L_{(u_1,u_2)})$, that is 
\begin{equation}
  \begin{aligned}
  t&> 1+\fracc1r +(\fracc1r-1)_+,
 \\
  t&> 2-s +(\fracc2r+\fracc2p-2)_+.
\end{aligned}
  \label{vKdom-eq}
\end{equation}
Then $u_1$, $u_2$ belong to $B^{t}_{r,o}(\overline{\Omega})$. 
If instead $f_k\in F^{t-4}_{r,o}(\overline{\Omega})$,
 $\varphi_0$, $\psi_0\in B^{t-\fracci1r}_{r,r}(\Gamma)$
and $\varphi_1$, $\psi_1\in B^{t-1-\fracci1r}_{r,r}(\Gamma)$
for some  $(t,r,o)$ fulfilling \eqref{vKdom-eq}, then it follows 
that $u_1$, $u_2\in F^{t}_{r,o}(\overline{\Omega})$. 
\end{thm}
Since $\Dm_2$ is open, it is not a loss of generality here to assume for the
Lizorkin--Triebel case that $u_1$ and $u_2$ are given in a Besov space.

One can prove the theorem directly, as indicated above, but it will follow
from the general considerations in Section~\ref{sys-sect}.
So instead the consequences for existence of solutions in
Besov and Lizorkin--Triebel spaces are given; this amounts to 
a solvability theory for the domain bounded by the dotted lines in 
Figure~\ref{vK-fig}.
It is also noteworthy that solutions exist for data
with arbitrarily large norms:

\begin{cor}
  \label{vK-cor}
Let 
$f\in B^{s-4}_{p,q}(\overline{\Omega})$ and
 $\psi_k\in B^{s-k-\fracpi}_{p,q}(\Gamma)$, for $k=0$, $1$, 
be real-valued data for some $(s,p,q)$ fulfilling
\begin{subequations} \label{vK-cnd}
\begin{align}   
  s&> 2+(\fracc 2p-1)_+, \quad\text{or} 
  \label{vK-cnd'} \\
  s&= 2+(\fracc 2p-1)_+ \quad\text{and}\quad q\le 2.
  \label{vK-cnd''}
\end{align}
\end{subequations}
Then there exists a solution $(u_1,u_2)$ in 
$B^{s}_{p,q}(\overline{\Omega})^2$ of the
equations in \eqref{vK-eqs}.

If $f\in F^{s-4}_{p,q}(\overline{\Omega})$ 
and $\psi_k\in B^{s-k-\fracpi}_{p,p}(\Gamma)$, for $k=0$, $1$,
and $(s,p,q)$ fulfils either \eqref{vK-cnd'} or 
\begin{equation}
  s= 2+(\fracc 2p-1)_+, \quad\text{and $q\le 2$ if $p\ge2$,}
\end{equation}
then \eqref{vK-eqs} has a solution $(u_1,u_2)$ in 
$F^{s}_{p,q}(\overline{\Omega})^2$.
\end{cor}
\begin{proof}
Under the assumptions on $(s,p,q)$, the data $f$ and $\psi_k$ belong 
to $F^{-2}_{2,2}(\overline{\Omega})$ and $B^{2-k-\frac12}_{2,2}(\Gamma)$, as
seen by the usual embeddings. So by invoking \cite[Thm.~5.8-3]{Ci97} there is a
solution $(u_1,u_2)\in F^{2}_{2,2}(\overline{\Omega})^2$; according to
Theorem~\ref{vK-thm} it also belongs to $B^{s}_{p,q}(\overline{\Omega})^2$
or $F^{s}_{p,q}(\overline{\Omega})^2$, respectively.
\end{proof}

\begin{exmp}
  \label{vK-exmp}
Equation \eqref{vK-eqs} may be considered with 
force term $f(x_1,x_2)$ equal to $1(x_1)\otimes\delta_0(x_2)$ and
$0\in\Omega$. Such singular data could model displacements and stresses
generated by a heavy rod lying along the
$x_1$-axis on a table, obtained by clamping a wooden plate along its edges
to a sturdy metal frame.

By \eqref{1delta-ex}, 
this $f\in B^{\fracpi-1}_{p,\infty}(\overline{\Omega})$ for 
every $p\in\,]0,\infty]$.
So Corollary~\ref{vK-cor} gives for 
every set of  $\psi_k\in B^{3-k}_{p,\infty}(\Gamma)$, $k=0,1$,  
with fixed $p\in\,]0,\infty]$,
a solution $(u_1,u_2)$ in $B^{3+\fracpi}_{p,\infty}(\overline{\Omega})^2$ 
of \eqref{vK-eqs}. By Theorem~\ref{vK-thm}, it belongs to this space for
every $p\in\,]0,\infty]$, when $\psi_0=\psi_1=0$. 
\end{exmp}
\begin{rem}
  \label{vK-rem}
Although the coupling of the two non-linear equations in \eqref{vK-eqs}, as
described, could be handled using that $u_1$ and $u_2$ 
are sought after in the same space, it seems more flexible to stick with the
general set-up in Section~\ref{main-sect} by developing a
theory in which the pair $(u_1,u_2)$ is regarded as the unknown, entering
the bilinear form twice. This only requires some projections onto $u_1$ and
$u_2$, cf the details around \eqref{Btilde-eq} below.
For this purpose it is convenient to generalise product
type operators to a framework of vector bundles, as done in the next
section.
\end{rem}

\section{Systems of semi-linear boundary problems}
  \label{sys-sect}

In this section the abstract results of Section~\ref{main-sect} and those on
paralinearisation in Section~\ref{prod-sect} will
be carried over to a general 
framework for semi-linear elliptic boundary problems. This is formulated in
a vector bundle set-up, not just because this is natural for linear elliptic
systems of multi-order, but also because vector bundles are useful for 
handling \emph{non-linearities}, as mentioned in Remark~\ref{vK-rem} above.

\subsection{General linear elliptic systems}
  \label{ell-ssect}
Because the parametrix construction relies on a linear theory 
with the properties in \eqref{XY-cnd}--\eqref{AXY-cnd} of
Section~\ref{main-sect}, it is natural to utilise the Boutet
de~Monvel calculus \cite{BM71}. 
The $L_p$-results for this are reviewed briefly below
(building on \cite{JJ96ell}, that extends $L_p$-results of G.~Grubb
\cite{G3} and J.~Franke \cite{F1,F2}).
Introductions to the calculus may be found in \cite{G97,G2} or
\cite[Sect.~4.1]{JoRu97}, and a thorough account in \cite{G1}.

Recall that $\Omega\subset\Rn$ denotes a
smooth, open, bounded set with $\partial\Omega=\Gamma$.
The main object is then a multiorder Green operator, designated by $\cal A$,
ie,  
\begin{equation}
 \cal A=\begin{pmatrix}P_\Omega+G&K\\ T& S\end{pmatrix}
  \label{grnA-eq} 
\end{equation}
where $P=(P_{ij})$ and $G=(G_{ij})$, $K=(K_{ij})$, $T=(T_{ij})$ 
and $S=(S_{ij})$. Here $i\in I_1:=\{\,1,2,\dots,i_\Omega\,\}$ and
 $i\in I_2:=\{\,i_\Omega+1,\dots,i_\Gamma \,\}$, respectively, 
in the two rows of the block matrix $\cal A$. Similarly it holds that
$j\in J_1:=\{\,1,2,\dots,j_\Omega\,\}$ and
 $j\in J_2:=\{\,j_\Omega+1,\dots,j_\Gamma \,\}$, respectively,  in
the two columns of $\cal A$; that is, $\cal A$ is an $i_\Gamma\times j_\Gamma$
matrix operator with indices belonging to $I\times J$, when $I=I_1\cup I_2$
and $J=J_1\cup J_2$. 

Each $P_{ij}$, $G_{ij}$, $K_{ij}$, $T_{ij}$ 
and $S_{ij}$ belongs to the poly-homogeneous calculus of pseudo-differential
boundary problems. More precisely, $P$  is a pseudo-differential operator
satisfying the uniform two-sided transmission condition (at $\Gamma$), $G$
is a singular Green operator, $K$ a Poisson and $T$ a trace operator, while
$S$ is an ordinary pseudo-differential operator on $\Gamma$. 
(The well-known requirements on the symbols and symbol kernels may be found in
the references above; they are not recalled, since they will not enter the
arguments directly here.)  
The operator in the $ij$\ord{th} entry of $\cal A$ is taken to be of order
$d+b_i+a_j$, where $d\in\Z$, $\Bold{a}=(a_j)\in\Z^{j_\Gamma}$ 
and $\Bold{b}=(b_i)\in\Z^{i_\Gamma}$; for each $j$, both $P_{ij,\Omega}+G_{ij}$ 
and $T_{ij}$ is supposed to be of class  $\kappa+a_j$ for some fixed
$\kappa\in\Z$.  
For short $\cal A$ is then said to be of order $d$ and class $\kappa$
(relatively to $(\Bold{a},\Bold{b})$, more precisely). 

Recall that  the transmission condition ensures that $P_\Omega:=r_\Omega
Pe_\Omega$ has the same order on all spaces on which it is defined. More
explicitly this means that each $P_{ij,\Omega}$ has order $d+a_j+b_i$ on
every $B^{s}_{p,q}$ and $F^{s}_{p,q}$ with arbitrarily high
 $s>\kappa +a_j+1-\fracp$; implying, say that $C^\infty(\overline{\Omega})$ is
mapped into $C^\infty(\overline{\Omega})$, without blow-up at
$\Gamma$. (Thus $P_\Omega$ has the transmission \emph{property}.)

In general the operators act on spaces of sections of
vector bundles $E_j$ over $\Omega$ and $F_j$ over $\Gamma$, with $j$ running
in $J_1$ and $J_2$, respectively; they map into
sections of other such bundles $E'_i$ and $F'_i$. The fibres of
 $E_j$, $F_j$ have dimension $M_j$, $N_j$, while $\dim E'_i=M'_i$ and
 $\dim F'_i=N'_i$. 
Letting
\begin{align}
  V& =(E_1\oplus\dots\oplus E_{j_\Omega})\cup 
      (F_{j_\Omega+1}\oplus\dots\oplus F_{j_\Gamma})
  \label{V-eq}  \\
  V'&=(E'_1\oplus\dots\oplus E'_{i_\Omega})\cup
      (F'_{i_\Omega+1}\oplus\dots\oplus F'_{i_\Gamma}),
  \label{V'-eq}
\end{align}
then $\cal A$ is a map $C^\infty(V)\to C^\infty(V')$. 
One may either regard $C^\infty(V)$ as a short hand 
for $C^\infty(E_1)\oplus\dots\oplus C^\infty(F_{j_\Gamma})$, or view
$V$ as a vector bundle with the dimension of both the base manifold 
$\Omega\cup\Gamma$ and of the fibres over its
points $x$ be depending on whether $x\in\Omega$ or $x\in\Gamma$ (as allowed in
eg the set-up of \cite{Lan}). 
Similarly for $V'$. 

The following spaces are adapted to the orders and classes of $\cal A$,
 \begin{align}
 B^{s+\Bold{a}}_{p,q}(V)&=
 (\bigoplus_{j\le j_\Omega} B^{s+a_j}_{p,q}(E_j) )
 \oplus (\bigoplus_{j_\Omega<j}  B^{s+a_j-\fracpi}_{p,q}(F_j) )
 \label{100} \\
 B^{s-\Bold{b}}_{p,q}(V')&=( \bigoplus_{i\le i_\Omega}  B^{s-b_i}_{p,q}(E'_i))
 \oplus ( \bigoplus_{i_\Omega<i}  B^{s-b_i-\fracpi}_{p,q}(F'_i) ).
 \label{101}
 \end{align}
Here the spaces of $B^{s}_{p,q}$-sections of $E_j$ etc is defined and normed
as usual via local trivialisations.
$F^{s+\Bold{a}}_{p,q}(V)$ and $F^{s-\Bold{b}}_{p,q}(V')$ are
analogous ($p<\infty$), except that $q=p$ in the summands
over $\Gamma$; as usual $F^{s}_{p,p}(F_j)=B^{s}_{p,p}(F_j)$ etc.
For convenience
\begin{align}
  \norm{v}{B^{s+\Bold{a}}_{p,q}}
 &=\big(\norm{v_1}{B^{s+a_1}_{p,q}(E_1)}^q+\dots+
 \norm{v_{j_\Gamma}}{B^{s+a_{j_\Gamma}-\fracpi}_{p,q}(F_{j_\Gamma})}^q\big)^{\fracci1q}
  \label{BV-eq}
  \\
 \norm{v}{F^{s+\Bold{a}}_{p,q}}
 &=\big(\norm{v_1}{F^{s+a_1}_{p,q}(E_1)}^p+\dots+
 \norm{v_{j_\Gamma}}{F^{s+a_{j_\Gamma}-\fracpi}_{p,p}(F_{j_\Gamma})}^p\big)^{\fracpi},
  \label{FV-eq}
\end{align}
with similar conventions for $B^{s-\Bold{b}}_{p,q}$ and $F^{s-\Bold{b}}_{p,q}$.
With respect to these spaces, 
$\cal A$ is {\em continuous\/}
 \begin{equation}
 \cal A\colon B^{s+\Bold{a}}_{p,q}(V)\to B^{s-d-\Bold{b}}_{p,q}(V'),
 \quad \cal A\colon F^{s+\Bold{a}}_{p,q}(V)\to F^{s-d-\Bold{b}}_{p,q}(V'),
 \label{104}
 \end{equation}
for each $(s,p,q)\in\Dm_\kappa$, when $p<\infty$ in the Lizorkin--Triebel
spaces. 

Ellipticity for multi-order Green operators is similar to
this notion for single-order operators, except that the principal symbol 
 $p^0(x,\xi)$ is a
matrix with $p^0_{ij}$ equal to the principal symbol of $P_{ij}$ {\em
relatively\/} to the order $d+b_i+a_j$ of $P_{ij}$; invertibility 
of $p^0(x,\xi)$ should hold for all $x\in\Omega$ and $|\xi|\ge1$.
The principal boundary operator $a^0(x',\xi',D_n)$ is similarly defined and
should be invertible as a map from $\cal S(\Rp)^{M}\times \C^{N}$ to
 $\cal S(\Rp)^{M'}\times \C^{N'}$ with $M:=\sum_{j\le j_{\Omega}} M_j$, 
$N:=\sum_{j_\Omega< j\le j_\Gamma} N_j$ etc.

For the mapping properties of elliptic systems $\cal A$ and their
parametrices one has the next theorem, which is an anbridged version of
\cite[Thm~5.2]{JJ96ell}.

\begin{thm} \label{grn-thm}
Let $\cal A$ denote a multi-order Green operator going from $V$ to $V'$,
and of order $d$ and class $\kappa$ relatively to $(\Bold{a},\Bold{b})$ as
described above. 
If $\cal A$ is injectively or surjectively 
elliptic, then $\cal A$ has, respectively, a left- or right-parametrix
$\widetilde{\cal A}$ in the calculus. $\widetilde{\cal A}$ can
be taken of order $-d$ and class $\kappa-d$, and then 
$\widetilde{\cal A}$ is bounded  in the opposite
direction in \eqref{104} for all the parameters $(s,p,q)\in\Dm_{\kappa}$.
The corresponding is true for $F^{s+\Bold{a}}_{p,q}(V)$ 
and $F^{s-d-\Bold{b}}_{p,q}(V')$.
In the elliptic case, all these properties hold for $\cal A$, and the
parametrices are two-sided.
\end{thm}

The above statement is deliberately rather brief.
It should be added that \eqref{104} is sharp, since it only holds
for $(s,p,q)$ outside
$\overline{\Dm}_\kappa$ if the class is effectively lower than $\kappa$.
Moreover, the kernel of $\cal A$ 
is a finite-dimensional space in $C^\infty(V)$, which is the same for all
$(s,p,q)$ and in the $B$- and $F$-cases; the range is closed with
complements that can be chosen to have similar properties. The reader is
referred to \cite{G3,JJ96ell} for this. In particular the
$(s,p,q)$-invariance of the range complements implies that the compatibility
conditions on the data are fulfilled for all $(s,p,q)$, if they are so for
one parameter. Hence these conditions can be ignored in the following
regularity investigations.

For the inverse regularity properties of an 
injectively elliptic system
$\cal A$, note that, by the above theorem, the
left-parametrix $\widetilde{\cal A}$ may be chosen so that
$\cal R:=I-\widetilde{\cal A}\cal A$
has class $\kappa$ and order $-\infty$, hence is continuous
\begin{equation}
  \cal R\colon B^{s+\Bold{a}}_{p,q}(V)\to C^\infty(V)
  \quad\text{for every}\quad (s,p,q)\in\Dm_{\kappa}.
  \label{Rran-eq}
\end{equation}
So if $\cal A u=f$ for some $u\in B^{s_1+\Bold{a}}_{p_1,q_1}(V)$ 
and data $f\in B^{s_0-d-\Bold{b}}_{p_0,q_0}(V')$, and if $(s_j,p_j,q_j)$
belongs to $\Dm_\kappa$ for $j=0$ and $1$, then
application of $\widetilde{\cal A}$ to $\cal A u=f$ yields 
(cf \eqref{lin-eq}--\eqref{lin-eq'} ff)
\begin{equation}
  u=\widetilde{\cal A}f+\cal Ru\in B^{s_0+\Bold{a}}_{p_0,q_0}(V).
\end{equation}

It can now be explicated how this framework fits with the
conditions \eqref{XY-cnd}--\eqref{AXY-cnd} of Section~\ref{main-sect}:
for each fixed $q\in\,]0,\infty]$ let 
$\Sdm=\{\,(s,p)\mid s\in\R,\ 0<p\le\infty\,\}$ and take
\begin{equation}
  X^s_p = B^{s+\Bold{a}}_{p,q}(V),\qquad
  Y^s_p = B^{s-\Bold{b}}_{p,q}(V'),   \qquad
  A_{(s,p)} = \cal A|_{B^{s+\Bold{a}}_{p,q}(V)},
  \qquad  \Dm(A)= \Dm_{\kappa}.
  \label{Agreen-eq}
\end{equation}
Moreover, $\widetilde{A} =  \widetilde{\cal A}$
should be chosen to be of class $\kappa-d$.
For the corresponding spaces $X^s_p=F^{s+\Bold{a}}_{p,q}(V)$ 
and $Y^s_p=F^{s-\Bold{b}}_{p,q}(V')$
one needs a little precaution because the sum and integral exponents in
\eqref{FV-eq} are equal in the spaces over the boundary bundles $F_j$. 
Then \eqref{fm-emb} is not a direct consequence of \eqref{fmem-eq} ff, but
for $p>r$, 
\begin{equation}
  F^{s+a_j-\fracpi}_{p,p}(F_j)\hookrightarrow
  F^{s+a_j-\fracpi}_{r,p}(F_j)\hookrightarrow
  F^{s+a_j-\fracci 1r}_{r,r}(F_j).
\end{equation}
In this way \eqref{XY-cnd} and \eqref{AXY-cnd} holds also for these spaces.

\begin{exmp}
  \label{biharm-ex}
For the Dirichl\'et problem for $\lap^2$, which enters the von Karman
equations, it is natural to let 
\begin{equation}
  \cal A=  \begin{pmatrix}\lap^2 &0\\ 0 &\lap^2 \\
                          \g& 0\\ \gamma_1&0\\ 
                          0& \g\\ 0& \gamma_1 \end{pmatrix},
  \label{Dbh-eq}
\end{equation}
whereby $d=4$, $\kappa=2$, $\Bold{a}=(0,0)$ and $\Bold{b}=(0,0,-4,-3,-4,-3)$.
The choice in \eqref{Agreen-eq} amounts to
\begin{align}
  X^s_p&= B^{s}_{p,q}(\overline{\Omega})^2
  \\
  Y^{s-4}_p&= B^{s-4}_{p,q}(\overline{\Omega})^2\oplus
    (B^{s-\fracpi}_{p,q}(\Gamma)\oplus B^{s-1-\fracpi}_{p,q}(\Gamma))^2;
\end{align}
this is clear since one can use the trivial bundles $V=\Omega\times\C^2$ and
$V'=(\Omega\times\C^2)\cup(\Gamma\times\C)^4$ for this problem.
\end{exmp}

\subsection{General product type operators}
  \label{prod-ssect}
Together with the Green operator $\cal A$ in \eqref{Dbh-eq} above, a
treatment of the von Karman equation may conveniently use the bilinear operator
$\tilde B$ given on $v=(v_1,v_2)$ and $w=(w_1,w_2)$ by
\begin{equation}
  \tilde B(v,w)=  \begin{pmatrix} -[v_1,w_2]& [v_1,w_1]&
0&0&0&0\end{pmatrix}^{\op{T}} .
  \label{Btilde-eq}
\end{equation}
Indeed, in the set-up of the previous section, a solution $u=(u_1,u_2)$ of
\eqref{vK-eqs} is a section of the trivial bundle $\Omega\times\C^2$, of which
the two canonical projections $u_1$ and $u_2$ enter directly into the
expressions in \eqref{vK-eqs}. The same projections enter for $v=w=u$ in
\eqref{Btilde-eq} above, and this is taken as the guiding principle in a
generalisation of product type operators to vector bundles.

Between vector bundles, a product type operator is roughly just an operator
that locally has the form introduced in Section~\ref{pdty-sect}.
But in relation to a given elliptic system $\cal A$ of order $d$ and class
 $\kappa$ with respect to a fixed set of integers $(\Bold{a},\Bold{b})$, it
is useful to introduce a class of product type operators with compatible
mapping properties.

Since the non-linearities typically send sections over $\Omega$ to other
such sections (so that sections over $\Gamma$ and zero-entries as in
\eqref{Btilde-eq} can be tacitly omitted), the following
framework should suffice for most applications:

Given bundles over $\Omega$ 
as in \eqref{V-eq}--\eqref{V'-eq}, there are bundles 
\begin{gather}
  W=E_1\oplus \dots\oplus E_{j_\Omega},\qquad
  W'=E'_1\oplus \dots\oplus E'_{i_\Omega},
  \label{WW'-eq}
 \\
  \beta_{j}\colon E_j\to\Omega,\qquad\beta'_i\colon E'_i\to\Omega
\end{gather}
in which sections $w$ and $w'$, respectively, may naturally be regarded as
$j_\Omega$- and $i_\Omega$-tuples of sections (by means of projections $\pr_j$ 
and $\pr'_i$)
\begin{equation}
  w=(w_1,\dots,w_{j_\Omega}),\qquad w'=(w'_1,\dots,w'_{i_\Omega}).
  \label{ww'-eq}
\end{equation}
There is also a finite covering 
$\Omega=\bigcup U_{\kappa}$ of local coordinate systems
 $\kappa\colon U_{\kappa}\to \tilde U_{\kappa}$, for disjoint open balls or
half balls $\tilde U_\kappa$ in $\Rn$. 
Alternatively $\tilde U_\kappa$ is written $U_{\tilde\kappa}$,
as it is the domain of $\tilde\kappa:=\kappa^{-1}$; 
then $E^s_{p,q}(\overline{ U_{\tilde\kappa}})$ denotes the function spaces
over $\tilde U_\kappa$.

With this there are associated
trivialisations $\tau_{j{\kappa}}$ and $\tau'_{i{\kappa}}$, for each $j$,
$i$ and ${\kappa}$, 
together with associated projections $\pr_{j{\kappa}m}$
onto the $m^{\op{th}}$ coordinate of $\C^{M_j}$:
\begin{equation}
  \beta_j^{-1}(U_{\kappa})\overset{\tau_{j{\kappa}}}{\longrightarrow} 
    \tilde U_{\kappa}\times\C^{M_j}
  \xrightarrow{\ \pr_{j{\kappa}m}\ } \C.
  \label{tt-eq}
\end{equation}
For short, $\tau_{j{\kappa}m}
:=\pr_{j{\kappa}m}\circ\tau_{j{\kappa}}\circ\pr_j$, and similarly for
$\tau'_{i{\kappa}m}$  and $\pr'_{i{\kappa}m}$ in the sequel. 

\begin{defn}
  \label{comp.ptyp-defn}
An operator $B$ from $W\oplus W$ to $W'$ is of product type
$(d_0, d_1, d_2)$ \emph{compatibly} with integers 
$(\Bold{a},\Bold{b})$ as in \eqref{100}--\eqref{101}~ff if the following
holds:
\begin{rmlist}
  \item  \label{ptyp1-cnd} 
Each map $\tau'_{i{\kappa}m}B(v,w)$ can be written
\begin{equation}
  \tau'_{i{\kappa}m}B(u,v)= \sum_{j_0,m_0,j_1,m_1}
    B^{j_0{\kappa}m_0,j_1{\kappa}m_1}_{i{\kappa}m}
    (\tau_{j_0{\kappa}m_0}(u),\tau_{j_1{\kappa}m_1}(v)) ,
  \label{ptyp1-eq}
\end{equation}
where $B^{j_0{\kappa}m_0,j_1{\kappa}m_1}_{i{\kappa}m}$ maps pairs of
sections of $W$ to sections of $\tilde U_{\kappa}\times\C$ and only
depends on two projections $\tau_{j_0{\kappa}m_0}(v)$
and $\tau_{j_1{\kappa}m_1}(w)$, 
where $1\le m_0\le M_{j_0}$ 
and $1\le m_1\le M_{j_1}$.
  \item   \label{ptyp2-cnd}
Each $B^{j_0{\kappa}m_0,j_1{\kappa}m_1}_{i{\kappa}m}$ is of product
type $(d_0+a_{j_0},d_1+a_{j_1}, d_2+b_i)$
on the open set $\tilde U_{\kappa}$ of $\Rn$.
\end{rmlist}
\end{defn}
\begin{rem}
  \label{prod-rem}
The non-linear operator $\tilde B$ in \eqref{Btilde-eq}, that enters the von
Karman equation, has the structure in
Definition~\ref{comp.ptyp-defn}. Indeed, working in $\Omega\times\C^2$ one has
$i=j=1$, but the choice $m=1$ in \eqref{ptyp1-cnd} gives $-[v_1,w_2]$ 
(if $\Omega$ is flat such as a ball), 
so that the non-trivial terms in \eqref{ptyp1-eq}
have $m_0=1$, $m_1=2$; whilst $m=2$ gives $m_0=m_1=1\ne m$.

As another illustration, the finite sums appear directly in the 
Navier--Stokes equation, where the unknown $(u,\goth{p})$ is a section of
 $W=W'=(\Omega\times\Cn)\oplus(\Omega\times\C)$, at least for the
Dirichl\'et condition. Here $(u,\goth{p})$
enters the non-linear term
$((u\cdot\nabla)u,0)$.  For $i=1$ each $m$ gives
rise to the sum $\sum_{m_0=1}^n v_{m_0}\partial_{m_0}w_m$,
where obviously any 
$m_0\in\{\,1,\dots,n\,\}$ occurs and $m_1=m$. (For $i=2$ the
zero-operator appears.)
\end{rem}

In the next result pseudo-local operators are defined
as usual to be those that decrease or preserve singular supports;
the singular support of 
eg a section $v$ of $W$ is the complement in $\Omega$ of the $x$ for which
$\tau_{j{\kappa}m}\circ v$ is $C^\infty$ from a neighbourhood of $x$ to
$\C$, for all  $U_{\kappa}\ni x$ and all $j$ and $m$.
It is understood that universal extension operators
have been chosen for the sets $\tilde U_{\kappa}$, so the exact
paralinearisations are meaningful on these sets.

\begin{thm}
  \label{ptyp-thm}
Let $B$ be of product type $(d_0,d_1,d_2)$ compatibly with
 $(\Bold{a},\Bold{b})$ and with $d_0\le d_1$;
and let Besov and Lizorkin--Triebel spaces be defined as in
\eqref{100}--\eqref{101}~ff, with the unified 
notation $E^{s+\Bold{a}}_{p,q}(V)$ 
and $E^{s-\Bold{b}}_{p,q}(V')$.
Then $Q(v):=B(v,v)$ is bounded
\begin{equation}
 E^{s+\Bold{a}}_{p,q}(V)\to E^{s-\sigma(s,p,q)-\Bold{b}}_{p,q}(V') 
 \quad\text{for every $(s,p,q)\in\Dm(Q)$,}
  \label{Vsigma-eq}
\end{equation}
whereby $\Dm(Q)$ and $\sigma(s,p,q)$ are given by \eqref{stQdom-eq} 
and \eqref{sgm-eq}, respectively. 

Moreover, for each $u\in E^{s_0+\Bold{a}}_{p_0,q_0}(V)$ there is a moderate
linearisation $L_u$, which with $\omega$ 
as in \eqref{om-eq} is bounded
\begin{equation}
  L_u\colon E^{s+\Bold{a}}_{p,q}(V)\to 
  E^{s-\omega-\Bold{b}}_{p,q}(V')
\end{equation}
for every $(s,p,q)$ in the parameter domain $\Dm(L_u)$ given  
by \eqref{DmLu-eq}. Furthermore, $L_u$ is pseudo-local on every such
$E^{s+\Bold{a}}_{p,q}(V)$. 
\end{thm}

\begin{proof}
Let $(s_0,p_0,q_0)$ and $u\in E^{s_0+\Bold{a}}_{p_0,q_0}(V)$ be given;
 and consider $(s,p,q)$ such that \eqref{DmLu-eq} holds. 
For each pair of projections $\tau_{j_0{\kappa}m_0}(u)\in
E^{s_0+a_{j_0}}_{p_0,q_0}(\overline{U_{\tilde\kappa}})$ 
and $\tau_{j_1{\kappa}m_1}(v)$, 
Theorem~\ref{mod-thm} applies to spaces with parameters
$(s_0+a_{j_0},p_0,q_0)$ and $(s+a_{j_1},p,q)$ since by \eqref{ptyp2-cnd} the
orders are $d_0+a_{j_0}$ and $d_1+a_{j_1}$, 
so there is a $u$-dependent linear operator 
$L^{j_0{\kappa}m_0,j_1{\kappa}m_1}_{i{\kappa}m}$ sending
$E^{s+a_{j_1}}_{p,q}(\overline{U_{\tilde\kappa}})$ continuously 
to $E^{s-\tilde\omega}_{p,q}(\overline{U_{\tilde\kappa}})$ for
\begin{equation}
  \tilde\omega=(d_2+b_i)+(d_1+a_{j_1})+(\fracc n{p_0}-s_0+d_0)_+
  +\varepsilon \qquad (\varepsilon\ge0).
\end{equation}
Therefore $L^{j_0{\kappa}m_0,j_1{\kappa}m_1}_{i{\kappa}m}$ is bounded 
$E^{s+a_{j_1}}_{p,q}(\overline{U_{\tilde\kappa}})\to
E^{s-\omega-b_i}_{p,q}(\overline{U_{\tilde\kappa}})$
for $\omega$ as in \eqref{om-eq}.
In case $(s_0,p_0,q_0)$ is in the domain
$\Dm(Q)$, one can take $(s,p,q)=(s_0,p_0,q_0)$ without violating
\eqref{DmLu-eq}, and then
\begin{equation}
  L^{j_0{\kappa}m_0,j_1{\kappa}m_1}_{i{\kappa}m}(\tau'_{j_1{\kappa}m_1}(u))
  =B^{j_0{\kappa}m_0,j_1{\kappa}m_1}_{i{\kappa}m}(\tau_{j_0{\kappa}m_0}(u),
     \tau_{j_1{\kappa}m_1}(u)).
  \label{Lilk-id}
\end{equation}
Summation over all $j_0$, $m_0$ and $j_1$, $m_1$ as in
\eqref{ptyp1-eq} gives
\begin{equation}
  \tau'_{i{\kappa}m}B(u,v)=\sum L^{j_0{\kappa}m_0,j_1{\kappa}m_1}_{i{\kappa}m}
    (\tau'_{j_1{\kappa}m_1}(v)).
\end{equation}
This determines a linear operator $L_{i{\kappa},u}$, which in the set of
sections of $U_{\tilde\kappa}\times\C^{M'_i}$ is given by
\begin{equation}
  L_{i{\kappa},u}(v)=(\sum   L^{j_0{\kappa}m_0,j_1{\kappa}m_1}_{i{\kappa}m}
    (\tau'_{j_1{\kappa}m_1}(v)))_{m=1,\dots,M'_i}. 
\end{equation}
As a composite map, $L_{i{\kappa},u}(v)$ is continuous 
$E^{s+\Bold{a}}_{p,q}(V)\to E^{s-\omega-b_i}_{p,q}(U_{\kappa})^{M'_i}$. 

Using a partition of unity
$1=\sum_{\kappa}\psi_{\kappa}$ subordinate to the coordinate patches
$U_{\kappa}$,
there is a bounded linear operator $L_u\colon E^{s+\Bold{a}}_{p,q}(V)\to
E^{s-\omega-\Bold{b}}_{p,q}(V')$ given  by 
\begin{equation}
  L_u(v)_i=\sum_{\kappa} (\tau'_{i{\kappa}})^{-1}\circ L_{i{\kappa},u}(\psi_{\kappa} v), 
  \quad\text{for}\quad i\in I_1.
  \label{Lu,i-id}
\end{equation}
It follows from Theorem~\ref{symb-thm} that each $L_{i{\kappa},u}$ is
pseudo-local; 
and so is $L_u$, since the class of pseudo-local operators 
is closed under addition.

When $(s_0,p_0,q_0)$ belongs to the domain $\Dm(Q)$ given by
\eqref{stQdom-eq}, then $v=u$ is possible for $(s_0,p_0,q_0)=(s,p,q)$,
and using \eqref{Lu,i-id}--\eqref{Lilk-id},
\begin{equation}
  L_u(u)_i=\sum_{\kappa} (\tau'_{i{\kappa}})^{-1}\circ \tau'_{i{\kappa}}
           B(u, \psi_{\kappa} u)= \pr_i B(u,\sum_{\kappa}\psi_{\kappa} u)=\pr_iB(u,u).
\end{equation}
Moreover, the value of $\omega$ equals
$\sigma(s,p,q)$, so \eqref{Vsigma-eq} is also proved.
\end{proof}

\subsection{Semi-linear elliptic systems}
  \label{semiell-ssect}
It is now straigthforward to specialise Theorem~\ref{ir-thm} 
to the vector bundle framework of multi-order systems.

For generality's sake it is observed that it suffices, by \eqref{AXY-cnd}, 
to take the linear part $\cal A$ injectively
elliptic, ie with a left parametrix $\widetilde{\cal A}$ 
and regularising operator $\cal R:=I-\widetilde{\cal A}\cal A$.
Recall that for a product type operator $B$,
the linearisation $L_u$ of $Q(u):=B(u,u)$ 
furnished by Theorem~\ref{ptyp-thm} enters the
parametrix 
\begin{equation}
  P^{(N)}=I+\widetilde{\cal A}L_u+\dots+(\widetilde{\cal A}L_u)^{N-1}.
  \label{parametrix-id}
\end{equation}
As above, 
$\Dm(\cal A,Q)=\{\,(s,p,q)\in\Dm_{\kappa}\cap \Dm(Q)
   \mid \sigma(s,p,q)<d\,\}$
is the domain where $Q$ is $\cal A$-moderate.
Using these ingredients, one has the following main result:

\begin{thm}
  \label{semiell-thm}
Let $\cal A$ be an injectively elliptic Green operator of order $d$ and
class $\kappa$ relatively to $(\Bold{a},\Bold{b})$, and
assume that $B$ is of product type $(d_0,d_1,d_2)$
compatibly with $(\Bold{a},\Bold{b})$, and with $d_0\le d_1$, 
so that $Q$ has order function $\sigma(s,p,q)$ 
on $\Dm(Q)$ and moderate linearisations $L_u$, 
according to Theorem~\ref{ptyp-thm}. 

For a section $u$ of $B^{s_0+\Bold{a}}_{p_0,q_0}(V)$
with $(s_0,p_0,q_0)\in \Dm(\cal A,Q)$,
and any choice $\widetilde{\cal A}$ 
of a left parametrix of $\cal A$ of class $\kappa-d$,
the parametrices $P^{(N)}$ in \eqref{parametrix-id} 
are bounded endomorphisms on
$B^{s+\Bold{a}}_{p,q}(V)$ for every $(s,p,q)$ 
in $\Dm_{\kappa}\cap\Dm(L_u)$.
And for $(s_1,p_1,q_1)$ and $(s_2,p_2,q_2)$ in $\Dm_{\kappa}\cap\Dm(L_u)$
the linear operator $(\widetilde{\cal A}L_u)^N$
maps $B^{s_1+\Bold{a}}_{p_1,q_1}(V)$ 
to $B^{s_2+\Bold{a}}_{p_2,q_2}(V)$ for all sufficiently large $N$.
If such a section $u$ solves the equation 
\begin{equation}
  \cal Au +Q(u)=f
  \label{AQ-eq}
\end{equation}
for data $f\in B^{t-d-\Bold{b}}_{r,o}(V')$ with 
$(t,r,o)\in\Dm_{\kappa}\cap\Dm(L_u)$, then
\begin{equation}
  u= P^{(N)}(\widetilde{\cal A}f +\cal Ru) + (\widetilde{\cal A}L_u)^N u
  \label{pmV-id}
\end{equation}
and $u\in B^{t+\Bold{a}}_{r,o}(V)$. 
Analogous results are valid for the 
scales $F^{s+\Bold{a}}_{p,q}(V)$ and $F^{s-\Bold{b}}_{p,q}(V')$.
\end{thm}

\begin{proof}
As observed in \eqref{Agreen-eq}, the choice
$X^s_p=B^{s+\Bold{a}}_{p,q}(V)$ and $Y^s_p=B^{s-\Bold{b}}_{p,q}(V')$ makes
conditions \eqref{XY-cnd} and \eqref{AXY-cnd} satisfied. 
As the $B_u$ in \eqref{B-cnd} one can take $L_u$, for its construction
via paramultiplication implies that it is unambigously defined on
intersections of the form $X^s_p\cap X^{s'}_{p'}$. Similarly there is
commutative diagrams for $\cal A$ and $\widetilde{\cal A}$ by the general
constructions in the Boutet de~Monvel calculus and the results in
Section~\ref{ext-ssect}. 

Moreover, $\Dm(\cal A,Q)$ is 
connected and $\delta=d-\omega(s,p,q)$ is constant
and positive; hence \eqref{Dm-cnd} and \eqref{del-cnd} hold.
The claims on $P^{(N)}$ may now be read off from Theorem~\ref{ir-thm}.
For $(\widetilde{\cal A}L_u)^N$ the sum exponents should also be controlled,
but $\Dm_{\kappa}\cap\Dm(L_u)$ is
open, hence contains $(s_1-\varepsilon,p_1,q_2)$ for $\varepsilon>0$, 
so that the larger space
$B^{s_1-\varepsilon+\Bold{a}}_{p_1,q_2}(V)$ is mapped into
$B^{s_2+\Bold{a}}_{p_2,q_2}$ for all sufficiently large $N$, according to
Theorem~\ref{ir-thm}. 

Finally, since $(s_0-\varepsilon,p_0,q_0)$ also belongs
to $\Dm_{\kappa}\cap\Dm(L_u)$ for sufficiently small $\varepsilon>0$, 
one can assume $q_0=o$. So according to Theorem~\ref{ir-thm} the section $u$
fulfils \eqref{pmV-id} and
belongs to $X^t_r=B^{t+\Bold{a}}_{r,o}(V)$. 
\end{proof}

It should be mentioned that while the abstract framework in
Theorem~\ref{ir-thm} was formulated with only $s$ and $p$ as parameters, for
convenience, the third parameter $q$ was easily handled in the proof above
by simple embeddings. 

From the given examples it is clear that 
Theorem~\ref{vK-thm} on the von~Karman problem is just a special case of the
above result. One also has

\begin{cor}
For operators $\cal A$ and $B$ as in Theorem~\ref{semiell-thm}, the equation
\begin{equation}
  \cal A u+Q(u)=f
\end{equation}
is hypoelliptic, ie for $f$ in $C^\infty(V')$ any solution $u$ belongs to
$C^\infty(V)$. 
\end{cor}

As an application of the parametrix formula
\eqref{pmV-id} it is shown that this corollary has a sharper local version.
This also uses the obvious fact that the class of pseudo-local
maps is stable under composition, in particular $\widetilde{\cal A}L_u$ is
pseudo-local. (This really only involves the $P_\Omega+G$-part of
$\widetilde{\cal A}$, since $L_u$
goes from $W$ to $W'$. And the pseudo-differential part clearly inherits
pseudo-locality from the operators on $\Rn$, since $P_\Omega=\rOm P\eOm$. 
For the singular Green part one
can extend \cite[Cor.~2.4.7]{G1} by means of
Rem.~2.4.9 there on $(x_n,y_n)$-dependent singular Green operators to get the
pseudo-local property. Details are omitted since it is 
outside of the subject.) 

Let $\Xi\subset\Omega$ be an open subregion 
with positive distance to the boundary, that is 
$\Xi\Subset \Omega$.  
Then, if $f$ in \eqref{AQ-eq} in addition fulfils
$f\in B^{t_1-d-\Bold{b}}_{r_1,o_1}(V'_{|_\Xi};\loc)$
it will be shown for any 
solution $u$ of \eqref{AQ-eq} that $u\in
B^{t_1+\Bold{a}}_{r_1,o_1}(V_{|_\Xi},\loc)$. 

More precisely, $f\in B^{t_1-d-\Bold{b}}_{r_1,o_1}(V'_{|_\Xi};\loc)$ 
means that
$\varphi f$ is in $B^{t_1-d-\Bold{b}}_{r_1,o_1}(V')$ for every $\varphi$
in $C^\infty(\overline{\Omega})$ with compact support contained in
$\Xi$. Hereby $\varphi f$ 
is calculated fibrewisely for the components of $f$, both in the bundles
$E'_i$ over $\Omega$, for $i\le i_\Omega$, and in the $F'_i$ over $\Gamma$, 
for $i_\Omega<i\le i_\Gamma$ 
(the last part is always $0$ for $\Xi\Subset\Omega$).  
That $u\in B^{t_1+\Bold{a}}_{r_1,o_1}(V_{|_\Xi},\loc)$ 
is defined similarly, and
these conventions extend to the $F$-spaces. 

\begin{thm}
  \label{semiell'-thm}
Under hypotheses as in Theorem~\ref{semiell-thm}, suppose 
$f\in E^{t_1-d-\Bold{b}}_{r_1,o_1}(V'_{|_\Xi};\loc)$
holds in addition to \eqref{AQ-eq} for some $(t_1,r_1,o_1)$ in 
$\Dm_{\kappa}\cap\Dm(L_u)$, for an open set $\Xi\Subset \Omega$. 
Then $u$ is also a section of
$E^{t_1+\Bold{a}}_{r_1,o_1}(V_{|_\Xi};\loc)$.
\end{thm}

\begin{proof}
Let $\psi,\chi_0$ and $\chi_1\in C^\infty(\overline{\Omega})$ be chosen so
that $\supp \chi_1\subset\Xi$ and
\begin{equation}
  \chi_0+\chi_1\equiv 1,\qquad \chi_j\equiv j 
  \text{ on a neighbourhood of $\supp\psi$.}
\end{equation} 
By the parametrix formula \eqref{pmV-id},
\begin{equation}
  \psi u=
  \psi P^{(N)}\bigl(\widetilde{\cal A}(\chi_1 f)+\cal R u\bigr)
  +\psi P^{(N)}\widetilde{\cal A}(\chi_0 f) + 
  \psi (\widetilde{\cal A}L_u)^N u
  \label{pmVloc-id}
\end{equation}
and here the last term belongs to $E^{t_1+\Bold{a}}_{r_1,o_1}(V)$ for a
sufficiently large $N$, according to the first part of
Theorem~\ref{semiell-thm}. 
Since $\widetilde{\cal A}L_u$ is pseudo-local so is $P^{(N)}$, and therefore
the inclusion $\singsupp \widetilde{\cal A}(\chi_0 f)\subset\supp\chi_0$
implies that $\psi P^{(N)}\widetilde{\cal A}(\chi_0 f)$ is in
$C^\infty_0(V)\subset E^{t_1+\Bold{a}}_{r_1,o_1}(V)$. 
And because $\widetilde{\cal A}(\chi_1f)+\cal Ru$ is in
$E^{t_1+\Bold a}_{r_1,o_1}(V)$, the fact that $P^{(N)}$ has order zero gives
that also the first term on the right hand side of
\eqref{pmVloc-id}
is in $E^{t_1+\Bold a}_{r_1,o_1}(V)$.
\end{proof}

When $\Xi$ adheres to the boundary of $\Omega$ one can depart from the
parametrix formula in the same way. But it seems to require more techniques
to show that $\psi P^{(N)}\widetilde{\cal A}(\chi_0 f)$ is in
$E^{t_1+\Bold{a}}_{r_1,o_1}(V)$, for although this term is in $C^\infty(V)$,
a possible blow-up at the boundary should be ruled out.

\section{Final remarks}
  \label{finrem-sect}

To sum up, a semi-linear elliptic boundary problem of product type as
in \eqref{AQ-eq} can conveniently be treated by determining
\begin{description}
  \item[$\Dm(\cal A)$] to have well-defined boundary conditions, ie
$\Dm(\cal A)=\Dm_{\kappa}$ when $\cal A$ is of class $\kappa$,
cf~\eqref{Dmk-eq};
  \item[$\Dm(Q)$] the quadratic standard domain of $Q$, cf~\eqref{stQdom-eq};
  \item[$\Dm(\cal A,Q)$] the domain where $Q$ is $\cal A$-moderate, 
obtained from $\Dm(\cal A)\cap\Dm(Q)$ and the inequality 
$s>\fracnp-d+d_0+d_1+d_2$ (unnecessary if $d_1-d_0\ge n$),
 cf \eqref{sdddd-eq};
  \item[$\Dm(L_u)$] the domain of the exact paralinearisation at
$u$, that is given by the inequality 
$s>-s_0+d_0+d_1+(\fracc np+\fracc n{p_0}-n)_+$,
cf~\eqref{DmLu-eq};
  \item[$\Dm_u$] equal to $\Dm(\cal A)\cap\Dm(L_u)$, ie the domain where the
parametrices $P^{(N)}_u$ induced by a given solution $u$ are defined
and the parametrix formula \eqref{pmV-id} holds.
\end{description}
Stated briefly, any given solution $u$ in $\Dm(\cal A,Q)$ then
leads to the parametrix formula \eqref{pmV-id}, 
and $u$ belongs to any space associated with the data, as long as
this space is in the larger domain $\Dm_u$.
Theorems~\ref{semiell-thm}--\ref{semiell'-thm} contain the precise statements,
including hypoellipticity and local properties in
subregions $\Xi\Subset\Omega$.

\subsection{A last example}
The use of parameter domains is finally illustrated 
by the following polyharmonic 
Dirichl\'et problem perturbed by $Q(u)=u^2$, and with
$\gamma u=(\gamma_1 u,\dots ,\gamma_{m-1}u)$:
\begin{subequations}   \label{mlap-eqs}
\begin{align}
  (\mlap)^m u +u^2&= cf \quad\text{in $\Omega\subset\Rn$}
  \\
  \gamma u&=0 \quad\text{on $\Gamma$}.
\end{align}
\end{subequations} 
Data are taken as
a constant $c>0$ times the function $f(x)=|x^2_1+\dots+x^2_k|^{a/2}$ 
in a domain $\Omega\subset\Rn$, $n\ge2$, with $\Omega\ni 0$. 
For $a\in \,]-k,0[\,$ it is clear that $|x'|^a$ is locally integrable on
$\R^k$, hence is in $\cal D'(\R^k)$, 
so by Proposition~\ref{x32-prop} $f$ is in
$B^{k/p+a}_{p,\infty}(\overline{\Omega})$ for all $p>0$.  

Now $(\mlap)^m\colon H^m_0(\overline{\Omega})\to
H^{-m}(\overline{\Omega})$ is a bijection by Lax--Milgram's lemma, and 
$f\in B^{k/2+a}_{2,\infty}\subset H^{-m}$ for $k+2a+2m>0$, so under
this condition data are consistent with the linear problem; because
$H^m=B^m_{2,2}$ this means that $(m,2,2)\in \Dm(\lap^m_{\gamma})
=\Dm_m$.
(Here $\lap^m_{\gamma}$ denotes the realisation of $(\mlap)^m$ 
induced by the condition $\gamma u=0$.)

If moreover $Q$ is of order $<2m$ on $H^m_0$, ie $(m,2,2)$
is in $\Dm(\lap^m_{\gamma},Q)$, that
by \eqref{sdddd-eq} holds for $m> n/6$,
then \eqref{mlap-eqs} is by Proposition~\ref{small-prop} 
solvable for certain $c>0$.

However, it is a consequence of the theory here that any solution $u$ in
$H^m_0$ also is an element of $W^{2m}_1(\overline{\Omega})$. This is an
improvement in the sense that any derivative $D^\alpha u$ with $|\alpha|\le
2m$ is a \emph{function}, which is not true for every element of $H^m$.

\begin{thm}
  \label{mlap-thm}
Let $\Omega\subset\Rn$ ($n\ge2$)
be smooth open and bounded, $0\in \Omega$. When $m\ge n/6$, $k\in
\{1,\dots,n \}$ and $-k<a<0$, then
problem \eqref{mlap-eqs} with $f(x)=c|x'|^a$
has a solution $u\in H^m_0(\overline{\Omega})$ for sufficiently small
$c>0$. Every solution in $H^m(\overline{\Omega})$ is then also in 
$B^{k+a+2m}_{1,\infty}(\overline{\Omega})$, which is a subspace of
$W^{2m}_1(\overline{\Omega})$. 
\end{thm}

\begin{proof}
Solvability was noted above. To see that any solution $u$ in
$H^m\subset B^m_{2,\infty}$ is in $B^{k+a+2m}_{1,\infty}$, note first that
$f\in B^{k+a}_{1,\infty}$ by the above. Moreover, to see that
the parameter $(k+a+2m,1,\infty)$ is in $\Dm(L_u)$, it suffices to apply
\eqref{DmLu-eq} with $(s_0,p_0,q_0)=(m,2,\infty)$, that yields 
$k+a+3m>n/2$. This inequality is  fulfilled since $m>n/6$ and
$k+a>0$ are assumed. So by Theorem~\ref{semiell-thm}, $u$ is in
$B^{k+a+2m}_{1,\infty}\subset W^{2m}_1$.
\end{proof}

In dimensions $n\in \{2,3,4,5\}$ the theorem allows $m=1$, hence 
covers eg the Dirichl\'et problem
of $\mlap$ 
for any choice of $k$, and every $a\in \,]-k,0[\,$.
For $n\in \{6,\dots,11\}$ the requirement that
$m>n/6$ shows that one gets the $W^{2m}_1$ regularity at least for 
$m=2$, ie for the
biharmonic Dirichl\'et problem; etc in higher dimensions.  

\begin{rem}   \label{illdef-rem}
In  many cases the square $Q(u)=u^2$ is
ill-defined on the `target' space $B^{k+a+2m}_{1,\infty}$, for this space
is outside of $\Dm(Q)$ if \eqref{stQdom-eq} is violated, ie if
$k+a+2m\le n/2$. But by taking $k+a>0$ close to $0$, it will be enough to
have $m<n/4$, so there are examples of such target spaces whenever 
$m$ can be taken in 
$\,]\tfrac{n}{6},\tfrac{n}{4}[\,\cap\N$, which is non-empty for
$n\in \{5,9,10,11\}$ and for $n\ge13$.
For the slightly larger space $W^{2m}_1$ one can refer to 
Remark~\ref{Wm-rem} for a specific proof that $Q$ cannot be continuous from
$W^{2m}_1\to \cal D'$ for  $m<n/4$.
Note that the result is sharp: 
if it could be shown that
$u\in B^t_{1,\infty}(\overline{\Omega})$ for $t$ so large that
$(t,1,\infty)$ is in $\Dm(\lap^m_{\gamma},Q)$, ie $t>n-2m>\tfrac{n}{2}$,
then $cf=\mlap^mu+u^2$ would be in
$B^{t-2m}_{1,\infty}$, which by Proposition~\ref{x32-prop} would imply
$t\le k+a+2m\le \tfrac{n}{2}$, giving a contradiction. 
Hence the ill-definedness of $Q$ 
at $B^{k+a+2m}_{1,\infty}$ is not explained by partial knowledge at $p=1$,
but rather by the fact that $Q$ is defined on $H^m\ni u$. 
\end{rem}

\bigskip

All in all there are legion examples of regularity properties
corresponding to spaces outside of the parameter domains of $A$-moderacy.
They are of importance for the general theory of partial
differential equations, albeit at some distance from the most common boundary
problems of mathematical physics. 

\subsection{Other types of problems}
  \label{pbs-ssect}
The analysed product type operators are obtained roughly by inserting
derivatives of the unknown $u$ in a polynomial of degree two;
cf~Section~\ref{pdty-sect}.  
This restriction to the second order case could seem artificial, but
it has been made in order not to burden the exposition.

In fact products $u_1(x)\dots u_m(x)$ have been analysed in $B^{s}_{p,q}$
and $F^{s}_{p,q}$ spaces by
paramultiplication in eg \cite[Ch.~4.5]{RuSi96}. The approach is the
same as for $m=2$ with collection of terms in two groups to which the dyadic
ball and corona criteria applies, respectively, but the complexity of this
is rather larger for $m>2$ because of the many indices. When
needed one can undoubtedly obtain, say $u^m=-L_u(u)$ and analyse the exact
paralinearisation $L_u$ along the lines of Theorem~\ref{mod-thm}, using the
framework of \cite[Ch.~4.5]{RuSi96}. Therefore these applications are left
for the future, while the second order case is treated here with its
consequences for eg the von Karman problem in Section~\ref{vK-sect}; as
mentioned the developed results also apply to the stationary Navier--Stokes
equation. 

An extension of the parametrix formulae to quasi-linear
problems seems to require further techniques.

\begin{rem}
  \label{parabolic-rem}
Parabolic boundary value problems could also be covered by
Theorem~\ref{ir-thm}, by taking $A$ as the full parabolic system
$(\partial_t-a(x,D_x),r_0,T)$ acting in anisotropic spaces ($r_0$ is
restriction to $t=0$, and $T$ a trace operator defining the boundary
conditions). 
For the linear problems, the reader is
referred to \cite[Sect.~4]{G4} for the $L_p$-theory (using classical
Besov and Bessel potential spaces) with a complete set of
compatibility conditions on fully inhomogeneous data. In particular
Corollary~4.5 there applies because the underlying 
manifold $\,]0,b[\,\times\Omega$ for $0<b<\infty$ is bounded, so that the
solution spaces $X^s_p$ fulfil \eqref{XY-cnd} above. Because of the stronger
data norms introduced to control the compatibility of the boundary- and
initial-data for exceptional values of $s$, cf \cite[(4.16)]{G4}, it is here
convenient  
that the $Y^s_p$-scale is not required to fulfil
\eqref{s-emb}--\eqref{fm-emb}. (The compatibility conditions
may force one to work with rather small parameter domains, once data are
given. But even so the present results may well allow considerable
improvements of the solution's integrability.)
For the non-linear terms, the
product type operators of Section~\ref{prod-sect} are
straightforward to treat in the corresponding anisotropic spaces, since the
necessary paramultiplication estimates have been established in this
framework \cite{Y1,JJ94mlt}.
\end{rem}

For problems of composition type, T.~Runst and the author \cite{JoRu97}
obtained solutions using the Leray--Schauder fixed point theorem and carried
the existence over to a large domain of $B^{s}_{p,q}$- and
$F^{s}_{p,q}$-spaces with a boot-strap argument. 
However, the domain of $A$-moderacy $\Dm(\cal A,Q)$ 
is not convex for such problems, cf \cite[Fig~1]{JoRu97}, so the iteration
almost developed into a formal algorithm.
J.-Y.~Chemin and C.-J.~Xu \cite{ChXu97} used a boot-strap method to
give a simplified proof of the smoothness of
weak solutions to the Euler--Lagrange equations of harmonic maps;
the basic step was to obtain hypoellipticity of a class of semi-linear
problems with terms of the form 
$\sum a_{j,k}(x,u(x))\partial_ju\partial_ku$.
Formally this incorporates both composition and product type
non-linearities, but the difficulties met in \cite{JoRu97} did not show up
in \cite{ChXu97}, since the weak solutions in this case are
known to be bounded (so that consideration of $u\mapsto F(u)$ on the
\emph{full} spaces $H^s_p$ or $B^{s}_{p,q}$ with $1<s<\fracnp$ was
unnecessary).  
However, this well indicates that larger families of non-linearities
will be relevant and potentially require disturbingly many additional
efforts. 

It has therefore been natural to treat only the class of product
type operators in the present article,
although Section~\ref{main-sect} applies at least to bounded solutions of
composition type problems.
But the latter sphere of problems 
could in general deserve stronger methods,
say to get rid of the boot-strap algorithm in \cite{JoRu97}. However, it seems
rather demanding to analyse the exact paralinearisation of $F(u)$
when $u$ is an unbounded function, say an element of $H^s_p$ for
$s<\fracnp$; this is probably an open problem.

\section*{Acknowledgement}
I am grateful to prof.\ J.-M.~Bony who asked for `inversion'
of semi-linear boundary problems.
Thanks are also due to prof.\ G.~Grubb for suggesting local regularity
improvements via pseudo-locality.

%
%
\providecommand{\bysame}{\leavevmode\hbox to3em{\hrulefill}\thinspace}
\providecommand{\MR}{\relax\ifhmode\unskip\space\fi MR }
\providecommand{\MRhref}[2]{%
  \href{http://www.ams.org/mathscinet-getitem?mr=#1}{#2}
}
\providecommand{\href}[2]{#2}

\end{document}